\documentclass[12pt]{article}

\usepackage{amsthm,amsmath,amssymb}
\usepackage{bbm}
\usepackage{enumerate, nicefrac}
\usepackage{geometry}

\usepackage[latin1]{inputenc}
\newcommand{\del}{\delta}
\newcommand{\smallsum}{\textstyle\sum}

\newtheorem{lemma}{lemma}[section]
\newtheorem{corollary}[lemma]{Corollary}

\newtheorem{theorem}[lemma]{Theorem}
\newtheorem{proposition}[lemma]{Proposition}

\renewcommand{\d}{C}
\renewcommand{\c}{c}

\providecommand{\N}{{\ensuremath{\mathbbm{N}}}}
\providecommand{\R}{{\ensuremath{\mathbbm{R}}}}
\providecommand{\E}{{\ensuremath{\mathbb{E}}}}
\renewcommand{\P}{{\ensuremath{\mathbb{P}}}}
\renewcommand{\H}{{\ensuremath{\mathbb{H}}}}
\providecommand{\1}{{\ensuremath{\mathbbm{1}}}}

\providecommand{\HS}{{\ensuremath{\textup{HS}}}}

\begin{document}

\title{Strong convergence rates for an explicit numerical 
approximation method for stochastic evolution equations with 
non-globally Lipschitz continuous nonlinearities}
\author{Arnulf Jentzen and Primo\v z Pu\v snik \\\\ ETH Z\"urich, Switzerland}

\maketitle
%
\begin{abstract}
In this article we propose a new, explicit and easily implementable
numerical method for approximating a class of semilinear stochastic 
evolution equations with non-globally Lipschitz continuous 
nonlinearities. We establish strong convergence rates for this 
approximation method in the case of semilinear stochastic evolution 
equations with globally monotone coefficients. Our strong convergence result, 
in particular, applies to a class of 
stochastic reaction-diffusion partial differential equations.
\end{abstract}

\tableofcontents
%

\section{Introduction}
\label{sec:intro}
In this article we are interested in strong approximations of stochastic evolution
equations (SEEs) with non-globally Lipschitz continuous nonlinearities.
In the literature, there are nowadays a number of strong approximation results 
for such stochastic evolution equations on finite 
dimensional state spaces, that is, for finite dimensional stochastic ordinary 
differential equations (SODEs). For example, Theorem~2.1 in Hutzenthaler et 
al.~\cite{HutzenthalerJentzenKloeden2013} (see also Theorem~2.1 in 
Hutzenthaler et al.~\cite{hjk11}) proves that the classical explicit 
Euler scheme (also known as Euler-scheme or Euler-Maruyama scheme; 
see Maruyama~\cite{m55}) diverges strongly and numerically weakly 
in finite time time when applied to a SODE with superlinearly growing (and hence 
non-globally Lipschitz continuous) nonlinearities. 
Theorem~2.4 in Hu~\cite{Hu1996} establishes that the drift-implicit Euler scheme 
(also known as Backward Euler scheme or implicit Euler scheme)
overcomes this lack of strong convergence of the explicit Euler scheme
and converges with the usual strong order $ \nicefrac{ 1 }{ 2 } $ to
the solution process in the case of some SODEs with non-globally Lipschitz continuous 
but globally monotone coefficients. However, the drift-implicit Euler scheme 
can often only be realized approximatively and this approximation of the drift-implicit 
Euler scheme is computationally more expensive than the explicit Euler scheme, particularly
when the state space of the considered SEE is high dimensional (see, e.g., Figure~4 
in Hutzenthaler et al.~\cite{HutzenthalerJentzenKloeden2012}),
because the solution of a nonlinear equation has to be computed approximatively
at each time step. In Hutzenthaler et al.~\cite{HutzenthalerJentzenKloeden2012}
a modified version of the explicit Euler scheme, which is explicit and easy to implement, 
has been proposed and shown to converge with the usual strong 
order $ \nicefrac{ 1 }{ 2 } $ to the solution process in the case of some 
SODEs with non-globally Lipschitz continuous but globally monotone coefficients. 
The above mentioned articles contain just a few selected illustrative results and a number of further and 
partially significantly improved strong approximation results for SODEs with non-globally Lipschitz 
continuous nonlinearities are available in the literature; see, e.g., 
\cite{BeynIsaakKruse2014},
\cite{HutzenthalerJentzen2014PerturbationArxiv},
\cite{HutzenthalerJentzen2014Memoires},
\cite{KloedenNeuenkirch2013}, 
\cite{Sabanis2013Arxiv},
\cite{Sabanis2013ECP},
\cite{Szpruch2013Vstable},
\cite{TretyakovZhang2013},
\cite{WangGan2013},
\cite{Zhang2014},
and the references mentioned in the above named references 
for some strong numerical approximations results for explicit schemes and multi-dimensional SODEs
with non-globally Lipschitz continuous coefficients. 
At least parts of the above outlined story has already been extended to SEEs 
on infinite dimensional state spaces including stochastic partial differential 
equations (SPDEs) as special cases. In particular, it is clear that 
Theorem~2.1 in Hutzenthaler et al.~\cite{HutzenthalerJentzenKloeden2013}
also extends to some SEEs with superlinearly 
growing nonlinearities on infinite dimensional state space
(see Section~5.1 in Kurniawan~\cite{MasterRyan}).
More specifially, 
the explicit, the exponential, and the linear-implicit Euler method
are known to diverge in the strong and numerically weak sense in the case 
of some SPDEs with superlinearly growing coefficients.
Moreover, strong convergence but with no rate of convergence of an full-discrete drift-implicit Euler method 
has, e.g., been proven in Theorem~2.10 in Gy\"{o}ngy \& Millet~\cite{GyoengyMillet2005}
in the case of some SEEs with non-globally Lipschitz continuous nonlinearities;
see also, e.g., 
Theorem~7.1 in Brze\'{z}niak et al.~\cite{BrzezniakCarelliProhl2013}
and
Theorem~5.4 in Kov\'{a}cs et al.~\cite{KovacsLarssonLindgren}.
Furthermore, 
in Gy\"{o}ngy et al.~\cite{SabanisSiskaGyongy2014Arxiv}, 
in Hutzenthaler \& Jentzen~\cite[equation~(3.145)]{HutzenthalerJentzen2014Memoires}, 
and in Kurniawan~\cite{MasterRyan} appropriately modified,
explicit and easily realizable versions of the 
explicit, the exponential and the linear-implicit Euler scheme 
have been considered for approximating semilinear SEEs with non-globally Lipschitz continuous nonlinearities.
In addition, 
in 
Gy\"{o}ngy et al.~\cite{SabanisSiskaGyongy2014Arxiv}
and in Kurniawan~\cite{MasterRyan},
it has also been proved that the considered approximation methods converge strongly 
to the solution processes of the investigated SEEs.
The results in Gy\"{o}ngy et al.~\cite{SabanisSiskaGyongy2014Arxiv}
and in Kurniawan~\cite{MasterRyan} do not prove any rate of strong convergence.
In this article we propose a modified variant of the scheme considered 
in Kurniawan~\cite[Section 2]{MasterRyan}
and prove for every $ p \in (0,\infty) $ that this scheme convergences in strong $ L^p $-distance with an appropriate 
strong rate of convergence in the case of a class of semilinear SEEs with non-globally
Lipschitz continuous but globally monotone nonlinearities; see Theorem~\ref{Finale} 
(the main result of this article) in Section~\ref{section5} 
below for details.
To the best of our knowledge, Theorem~\ref{Finale} below is the first result in the literature 
which establishes a strong convergence rate for an explicit and easily implementable
full-discrete numerical approximation method for semilinear SPDEs with non-globally Lipschitz continuous 
nonlinearities.

In the remainder of this introductory section 
we illustrate Theorem~\ref{Finale} by presenting a consequence of it
in Theorem~\ref{Izrek1} below.
For this we consider the following setting (see Section \ref{section5} below for our 
general framework). 
Let $\left( H, \langle \cdot, \cdot \rangle _H, \left \| \cdot \right\|_H \right)$ 
and 
$\left( U, \langle \cdot, \cdot \rangle _U, \left\| \cdot \right \|_U \right)$
be separable $ \R$-Hilbert spaces, let $ \H \subseteq H $ be a countable orthonormal basis of $ H,$
let $ \mathbb{U} \subseteq U $ be an orthonormal 
basis of $ U,$
let $ \lambda \colon \H \to \R $ be a function satisfying $ \sup_{h \in \H} \lambda_h < 0,$
let $A \colon D( A ) \subseteq H \to H$ 
be the linear operator such that
$ 
  D(A) 
  = \big\{
    v \in H \colon \sum_{ h\in \H} \left| \lambda_h \langle h, v \rangle_H \right|^2 
    < \infty
  \big\} 
$
and such that for all 
$v \in D(A)$
it holds that
$ A v = \sum_{ h\in \H} \lambda_h \langle h, v \rangle _H h,$
let
$(H_r, \langle \cdot, \cdot \rangle_{H_r}, \left \| \cdot \right \|_{H_r} ), r\in \R,$ be a family of interpolation spaces associated to $-A$
(see, e.g., Theorem and Definition $2.5.32$ in \cite{Jentzen2014SPDElecturenotes}), 
let $ T \in (0,\infty),$ $ c \in [1,\infty),$
$\gamma\in [0,\nicefrac{1}{2}),$ 
$\alpha \in [0, 1-\gamma),$ 
$\beta \in [0, \nicefrac{1}{2}-\gamma),$
$\delta\in [0,\gamma],$
$\xi \in H_{\nicefrac{1}{2}},$
$\theta\in (0, \nicefrac{1}{4}], $ 
$p\in [2,\infty),$ 
$\kappa\in(\nicefrac{2}{p},\infty),$
$\varepsilon\in(0,\infty),$
$
F\in \mathcal{C} 
(H_\gamma, 
H) 
,
B\in \mathcal{C}
(H_\gamma,
\HS(U, H) 
)
,
$
let
$( \Omega, \mathcal{ F }, \P)$
be a probability space with a normal filtration 
$( \mathcal{ F }_t )_{t \in [ 0, T]},$
let
$ \left( W_t \right)_{t \in [0, T]}$
be a cylindrical
$\operatorname{Id}_U$-Wiener process with respect to 
$ ( \mathcal{F}_t )_{t \in [0, T]},$
let $X \colon [0,T]\times \Omega \to H_\gamma$ be 
an $(\mathcal{F}_t)_{t\in [0,T]}$-adapted stochastic process
with continuous sample paths
such that for all $t\in [0, T]$ it holds $\P$-a.s.\,that
\begin{equation}
\label{have.solve}
X_t = e^{tA} \xi + \int_0^t e^{(t-s)A}F(X_s) \, ds
+
\int_0^t e^{(t-s)A} B( X_s ) \, dW_s,
\end{equation}
let $(P_I)_{I\in \mathcal{P}(\H) }\subseteq L(H) $
and
$(\hat P_J)_{J\in \mathcal{P}(\mathbb{U})}\subseteq L(U)$ 
be the linear operators with the property 
that\footnote{Here and below we denote for a set $ S $ by
$ \mathcal{P}( S ) $ the power set of $ S $
and we denote for a set $ S $ by $ \mathcal{P}_0( S ) $ the set given by 
$
  \mathcal{P}_0( S ) = \{ M \in \mathcal{P}( S ) \colon M \text{ is a finite set} \}
$.}
for all
$ x \in H $, $ y \in U $, $ I \in \mathcal{P}( \H ) $, $ J \in \mathcal{P}(\mathbb{U}) $ 
it holds that
$ P_I(x) =\sum_{h\in I} \langle h, x \rangle_H h$
and
$ \hat P_J(y) = \sum_{ u \in J } \langle u, y \rangle_U u $,
let 
$
  Y^{ N, I, J } 
  \colon [0, T] \times \Omega \to H_{ \gamma } 
$, 
$ N \in \N $, 
$ I \in \mathcal{P}( \H ) $, 
$ J \in \mathcal{P}( \mathbb{U} ),$ 
be 
$(\mathcal{F}_t)_{t\in [0,T]}$-adapted 
stochastic processes 
such that for all 
$ t \in [0,T] $, $ N \in \N $, 
$ I \in \mathcal{P}( \H ) $, $ J \in \mathcal{P}( \mathbb{U} ) $
it holds $\P$-a.s.\ that
\begin{equation}
\label{Main.scheme.last.pertubation}
\begin{split}
Y^{N,I,J}_t 
&
=
e^{tA} P_I \xi
+
\int_0^t
e^{(t- \lfloor s \rfloor_{T/N})A}
\,
\1_{ \big \{ 
 \| 
P_I F  (Y^{N,I,J}_{\lfloor s \rfloor_{T/N}}  )
 \|_H 
+ 
 \| P_I B( Y^{N,I,J}_{\lfloor s \rfloor_{T/N}} )  \|_{HS(U, H)} 
\,
\leq 
\left ( \frac{N}{T}  \right)^\theta \big \}}
P_I F( Y^{N,I,J}_{\lfloor s \rfloor_{T/N}} ) 
\,
ds
\\
&
\quad
+
\int_0^t e^{(t- \lfloor s \rfloor_{T/N})A}
\,
\1_{ \big\{ \| P_I F(Y^{N,I,J}_{\lfloor s \rfloor_{T/N}} )\|_H 
+ 
 \| P_I B( Y^{N,I,J}_{\lfloor s \rfloor_{T/N}}  )  \|_{HS(U, H)} 
\leq 
\left ( \frac{N}{T} \right )^\theta \big \}}
\,
 P_I B( Y^{N,I,J}_{\lfloor s \rfloor_{T/N}}  ) \hat P_J\,dW_s
 ,
\end{split}
\end{equation}
and
assume that for all $ x, y \in H_\gamma $, $ v, w \in H_1 $ it holds that
$
  \langle v, F(v) \rangle_H
  +
  \frac{
    2 c (c + 1) 
    p 
    \max\{ \kappa, \nicefrac{ 1 }{ \theta } \} 
    - 1
  }{ 2 }
  \left\| 
    B(v) 
  \right \|_{ \HS(U,H) }^2
\leq
  c 
  \left( 
    1 + \| v \|_H^2
  \right)
  ,
$
$
  \max\{ 
    \| F(x) \|_{ H_{ - \alpha } } , \| B(x) \|_{ \HS( U, H_{ - \beta } ) }  
  \} 
  \leq 
  c \left( 1 + \| x \|_H^c \right)
$,
$
  \langle v - w , A v - A w + F(v) - F(w) \rangle_H 
  + 
  \tfrac{ ( p - 1 ) ( 1 + \varepsilon ) }{ 2 } 
  \left\| B(v) - B(w) \right\|_{ \HS( U, H ) }^2  
  \leq 
  c \left\| v - w \right\|_H^2
$,
and
$
  \max\{ 
    \left\| 
      F(x) - F(y) 
    \right\|_H , 
    \left\| 
      B(x) - B(y) 
    \right\|_{
      \HS( U, H ) 
    } 
  \}
\leq 
  c
  \left \| x - y \right \|_{H_\delta}
  ( 1 + \left \| x \right \|_{H_\gamma}^\c + \left \| y \right \|_{H_\gamma}^\c )
$.
In the following we refer to the numerical approximations in \eqref{Main.scheme.last.pertubation}
as nonlinearities-stopped exponential Euler approximations 
(cf., e.g., Kurniawan~\cite[Section~2]{MasterRyan}).
\begin{theorem}
\label{Izrek1}
Assume the setting in the second paragraph of Section~\ref{sec:intro}. 
Then for every $ \eta \in [0, \nicefrac{ 1 }{ 2 } ) $ 
there exists a real number $K \in [0,\infty)$ such that for all $N \in \N, I\in \mathcal{P}_0(\H), J\in \mathcal{P}_0(\mathbb{U})$
it holds that
\begin{equation}
\begin{split}
  \sup_{ t \in [0,T] }
  \big\|
    X_t - Y_t^{ N, I, J } 
  \big\|_{ L^p( \P; H ) }
\leq
  K
  \left(
    N^{  \delta - \eta }
    +
    \Big|
      \!
      \sup_{
        h \in \H \setminus I
      }
      \lambda_h
    \Big|^{
      ( \delta - \eta )
    }
    +
    \sup_{ v \in H_\eta }
    \Big[
      \tfrac{
        \| B(v) \hat{P}_{ \mathbb{U} \backslash J } \|_{ \HS(U, H_{-\eta}) }
      }{ 
        ( 1 + \|v\|_{ H_\eta } )^\kappa
      }
    \Big]
  \right)
.
\end{split}
\end{equation}
\end{theorem}
Theorem \ref{Izrek1} is a direct consequence of Theorem~\ref{Finale} in Section~\ref{section5} below.
In the following we give an outline of the proof of Theorem~\ref{Finale} and we 
also sketch the content of the remaining sections of this article.
The proof of Theorem~\ref{Finale} is divided into several pieces.
First, in Section \ref{section2} we establish a priori moment estimates 
for the approximation scheme \eqref{Main.scheme.last.pertubation} 
in the $ H_0 $-norm.
In Section~\ref{section3} we use twice suitable bootstrap-type arguments 
to strengthen these a priori moment bounds in the $ H_0 $-norm to obtain 
for any $ \eta \in ( - \infty, \nicefrac{ 1 }{ 2 } ) $ 
a priori moment estimates 
for the approximation scheme \eqref{Main.scheme.last.pertubation} 
in the $ H_{ \eta } $-norm.
In Section~\ref{section4} we use the a priori moment bounds established in 
Sections~\ref{section2} and \ref{section3} to estimate the temporal 
discretization errors of the nonlinearities-stopped exponential Euler approximations 
in \eqref{Main.scheme.last.pertubation}; see Corollary~\ref{corollary.combined} 
in Section~\ref{section4}.
Our main idea in the proof of Corollary~\ref{corollary.combined} is not to estimate
the error of the numerical approximations~\eqref{Main.scheme.last.pertubation}
and the solution process $ X $ of the SEE~\eqref{have.solve} directly but instead 
to plug, similar as in 
Jentzen \& Kurniawan \cite[(11), (70), (136)]{KurniawanJentzen2015Arxiv}, 
appropriate approximation processes, so-called semilinear integrated 
counterparts of \eqref{Main.scheme.last.pertubation}, 
in between, to estimate the difference of the numerical approximations~\eqref{Main.scheme.last.pertubation}
and their semilinear integrated counterparts in a straightforward way 
(see Lemma \ref{Lemma12}) and to employ the perturbation estimate 
in Theorem~2.10 in Hutzenthaler \& Jentzen~\cite{HutzenthalerJentzen2014PerturbationArxiv} to estimate the differences of
the solution process of the considered SEE and the semilinear
integrated counterparts of the nonlinearities-stopped exponential Euler
approximations (see Lemma \ref{Lemma13}). Combining Lemma \ref{Lemma12} and Lemma \ref{Lemma13} with
the triangle inequality will then immediately result in Corollary \ref{corollary.combined}.
In Section~\ref{section.known} and Section~\ref{section.galerkin}
we establish an auxiliary spatial approximation result
(see Proposition~\ref{convergence_probability} in Section~\ref{section.galerkin})
which we use in Section~\ref{section5} to prove Theorem~\ref{Finale}.
In addition, we use consequences of the perturbation estimate 
in Theorem~2.10 in 
Hutzenthaler \& Jentzen~\cite{HutzenthalerJentzen2014PerturbationArxiv}
to establish strong convergence rates for spatial spectral Galerkin 
approximations (see Lemma~\ref{lem:spatial_rate})
and for noise approximations 
(see Lemma~\ref{noise.galerkin})
of the considered SEEs.
Combining the spatial approximation result in Lemma~\ref{lem:spatial_rate}
in Section~\ref{section5}
and the noise approximation result in Lemma~\ref{noise.galerkin} 
in Section~\ref{section5} with the results established 
in the earlier sections of this article (especially the temporal
approximation result in Corollary~\ref{corollary.combined} in Section~\ref{section4})
will then allow us to complete the proof of Theorem~\ref{Finale}
in Section~\ref{section5}.
In Section~\ref{section8} we illustrate the consequences 
of Theorem~\ref{Finale} 
and Theorem~\ref{Izrek1} respectively 
in the case of an illustrative example SPDE.

More formally, suppose that 
$ 
  ( H, \left< \cdot , \cdot \right>_H , \left\| \cdot \right\|_H ) 
  =
  ( U, \left< \cdot , \cdot \right>_U , \left\| \cdot \right\|_U ) 
$
is the $ \R $-Hilbert space of equivalence classes of Lebesgue square integrable
functions from $ (0,1) $ to $ \R $,
let 
$ \rho \in (0,\infty) $,
$ ( r_n )_{ n \in \N } \subseteq \R $,
and 
$ ( e_n )_{ n \in \N } \subseteq H $
satisfy that
$ \H = \{ e_1, e_2, \dots \} $,
that
$
  \sup_{ n \in \N } \left( n \left| r_n \right| \right) < \infty
$,
and that\footnote{Here and below we denote for a natural number $ d \in \N $
and a Borel measurable set $ B \in \mathcal{B}( \R^d ) $ by 
$ \mu_B \colon \mathcal{B}( B ) \to [0,\infty] $ the 
Lebesgue-Borel measure on $ B $.}
for all $ n \in \N $
and $ \mu_{ (0,1) } $-almost all $ x \in (0,1) $
it holds that 
$
  e_n( x ) = \sqrt{2} \sin( n \pi x )
$,
$
  \lambda_{ e_n } = - n^2 \pi^2
$
and 
$
  \xi( x ) \geq 0
$,
let $ Q \in L_1(H) $ satisfy that
for all $ v \in H $ it holds that
$
  Q v = \sum_{ n = 1 }^{ \infty } r_n \left< e_n, v \right>_H e_n
$,
assume that 
$ \gamma \in ( \nicefrac{ 1 }{ 4 } , \nicefrac{ 1 }{ 2 } ) $,
and assume that
for all $ v \in H_{ \gamma } $,
$ u \in H $ 
and $ \mu_{ (0,1) } $-almost all $ x \in (0,1) $
it holds that
$ \big( F( v ) \big)(x) = | v( x ) | \left( \rho - v(x) \right) $
and $ \big( B( v ) u \big)( x ) = v(x) \cdot ( Q^{ 1 / 2 } u )( x ) $.
The stochastic process $ X $ is thus a solution process of 
the stochastic reaction-diffusion partial differential equation
\begin{equation*}
  d X_t( x ) = 
  \left[ 
    \tfrac{ \partial^2 }{ \partial x^2 } X_t( x )
    +
    X_t( x ) \left( \rho - X_t( x ) \right)
  \right] 
  dt
  +
  \sigma X_t( x ) \, dW_t(x) ,
  \; 
  X_0( x ) = \xi( x ) ,
  \; 
  X_t( 0 ) = X_t( 1 ) = 0 
\end{equation*}
for $ t \in [0,T] $, $ x \in (0,1) $.
Then we show in Section~\ref{section8} that Theorem~\ref{Izrek1} 
ensures that
for every $ q, \iota \in (0, \infty ) $ 
there exists a real number $ K \in [0,\infty) $ 
such that for all $ N, n, m \in \N $
it holds that
\begin{equation}
\begin{split}
  \sup_{ t \in [0,T] }
  \big\| 
    X_t - Y_t^{ N, \{ e_1, e_2, \dots, e_n \} , \{ e_1, e_2, \dots, e_m \} } 
  \big\|_{ L^q(\P ; H) }
\leq
K
\left( 
  \frac{ 
    1 
  }{
    N^{ 
      ( \nicefrac{ 1 }{ 2 } - \iota ) 
    }
  }
  +
  \frac{ 1 }{
    n^{ ( 1 - \iota ) }
  }
  +
  \frac{ 1 }{ 
    m^{ ( 1 - \iota ) }
  }
\right)
.
\end{split}
\end{equation}
In particular, this shows that for every $ q, \iota \in (0, \infty ) $ 
there exists a real number $ K \in [0,\infty) $ 
such that for all $ n \in \N $
it holds that
$
  \sup_{ t \in [0,T] }
  \| 
    X_t - Y_t^{ n^2, \{ e_1, e_2, \dots, e_n \} , \{ e_1, e_2, \dots, e_n \} } 
  \|_{ L^q(\P ; H) }
\leq
  K \cdot 
  n^{ ( \iota - 1 ) }
$.

\subsection{Notation}
\label{Notation}

Throughout this article the following notation is used.
For a set $S$ we denote by $\operatorname{Id}_S\colon S\to S$ the identity mapping on $S,$
that is, it holds for all $x\in S$ that
$\operatorname{Id}_S(x)=x. $
For a set $A$ we denote by $\mathcal{P}(A)$ its power set, we denote by $|A|\in \{0, 1, 2, \ldots \} \cup \{ \infty\}$ the number of elements of $A,$ and we denote by $\mathcal{P}_0(A)$ the set given by $\mathcal{P}_0(A)=\{ B\in \mathcal{P}(A)\colon|B|<\infty\}.$
For measurable spaces $(\Omega_1, \mathcal{F}_1)$ and
$(\Omega_2, \mathcal{F}_2)$ we denote by $\mathcal{M}( \mathcal{F}_1 , \mathcal{F}_2 )$ the set of all $\mathcal{F}_1/\mathcal{F}_2$-measurable functions.
For topological spaces $(X, \tau_X)$ and $(Y, \tau_Y)$ we denote by
$\mathcal{C}(X,Y)$ the set of all continuous functions from $X$ to $Y.$ 
For a topological space $(X, \tau)$ we denote by $\mathcal{B}(X)$ the sigma-algebra of all Borel measurable sets in $X.$
Let 
$\lfloor \cdot \rfloor_h \colon \R \to \R, h\in (0,\infty),$
be the mappings with the property that for all $h\in (0,\infty), t\in \R$ it holds that
$
\lfloor t \rfloor_h = \max \big\{ (-\infty, t]\cap 
\{0, h, -h, 2h, -2h,\ldots \} \big\}
.
$
For a natural number $ d \in \N $ and 
a Borel measurable set 
$ A \in \mathcal{B}(\R^d) $
we denote by $\mu_A \colon \mathcal{B}(A)\to [0, \infty]$ the Lebesgue-Borel measure on $A.$
For a normed $ \R $-vector space $ ( V, \left\| \cdot \right\|_V ) $ and a 
real number $ \rho \in [0,\infty) $ we denote by $ L^{ \rho }( V ) $ the 
set given by $ L^{ \rho }( V ) = \big\{ A \in L(V) \colon \| A \|_{ L(V) 
} \leq \rho \big\} $.

\subsection{Setting}
\label{Main.setting}

Throughout this article the following setting is frequently used. Let
$\left( H, \langle \cdot, \cdot \rangle _H, \left \| \cdot \right\|_H \right), $
$\left( U, \langle \cdot, \cdot \rangle _U, \left\| \cdot \right \|_U \right) $
be separable $\R$-Hilbert spaces, 
let 
$\H \subseteq H$
be a non-empty orthonormal basis of $H,$ 
let $ \mathbb{U} \subseteq U $ be an orthonormal basis of $U,$
let
$ \lambda \colon \H \to \R $
be a function satisfying
$ \sup_{h \in \H} \lambda_h < 0,$
let
$A \colon D( A ) \subseteq H \to H$ 
be a linear operator such that
$ D(A) = \big \{
v \in H \colon \sum_{ h \in \H} \left | \lambda_h \langle h, v \rangle_H \right | ^2 <  \infty
\big\} $
and such that for all 
$v \in D(A)$
it holds that
$ A v = \sum_{ h \in \H} \lambda _h \langle h, v \rangle _H h,$
let
$(H_r, \langle \cdot, \cdot \rangle_{H_r}, \left \| \cdot \right \|_{H_r} ), r\in \R,$ be a family of interpolation spaces associated to $-A$
(see, e.g., Theorem and Definition $2.5.32$ in \cite{Jentzen2014SPDElecturenotes}),
let 
$ \alpha, \beta, \gamma \in [0, \infty), 
 T \in (0,\infty),$
$ a, c, C \in [1,\infty) $,
$ \delta \in [0, \gamma] $,
$ \theta\in (0, \tfrac{1}{4}] $, 
$ p \in [2,\infty) $, 
let
$( \Omega, \mathcal{ F }, \P)$
be a probability space with a normal filtration 
$( \mathcal{ F }_t )_{t \in [ 0, T]},$
let $\xi \in \mathcal{M}(\mathcal{F}_0, \mathcal{B}(H_\gamma)),$
and
let
$ \left( W_t \right)_{t \in [0, T]}$
be a cylindrical
$\operatorname{Id}_U$-Wiener process with respect to 
$ ( \mathcal{F}_t )_{t \in [0, T]}.$

\subsection{Acknowledgements}

We gratefully acknowledge Martin Hutzenthaler for bringing several typos into our notice.
This project has been supported through the SNSF-Research project $ 200021\_156603 $ ``Numerical 
approximations of nonlinear stochastic ordinary and partial differential equations''.

\section{Strong a priori moment bounds for 
nonlinearities-stopped schemes}
\label{section2}
In this section we establish strong a priori moment bounds for a class 
of nonlinearities-stopped exponential Euler approximations; see Lemma \ref{Lemma1}
below for the main result of Section \ref{section2}. Related arguments/results can, 
e.g., be found in Section 2 in Kurniawan~\cite{MasterRyan} and in 
Section 3 in Gy\"{o}ngy et al.\ \cite{SabanisSiskaGyongy2014Arxiv}.
\subsection{Setting}
\label{setting1}
Assume the setting in Section \ref{Main.setting},
let
$
F\in \mathcal{M} 
\big( \mathcal{B}(H_\gamma), 
\mathcal{B}(H) 
\big)
,
B\in \mathcal{M}
\big(
\mathcal{B}(H_\gamma),
\mathcal{B}(\HS(U, H) )
\big)
,
$
$N\in \N, K\in [0,\infty)$
satisfy 
$
K
=  
3(p-2)
+
2C^{\nicefrac{p}{2}}
+
2
\big [
T^{1 - 2\theta} 
+
\tfrac{p}{2}
T^{\frac{1}{2} - 2 \theta}
\big ]^{\nicefrac{p}{2}}
,
$
let $Y \colon  [0,T] \times \Omega \to H_\gamma$ be an $ ( \mathcal{F}_t )_{ t \in [0,T] } $-adapted stochastic process
such that for all $t\in [0,T]$ it holds $\P$-a.s.\,that
\begin{equation}
\label{Main.scheme}
\begin{split}
Y_t 
=
e^{tA} \xi
&
+
\int_0^t
e^{(t- \lfloor s \rfloor_{T/N})A}
\,
\1_{ \big \{ 
\| 
F  (Y_{\lfloor s \rfloor_{T/N}}  )
\|_H 
+ 
 \| B( Y_{\lfloor s \rfloor_{T/N}} )  \|_{HS(U, H)} 
\,
\leq \left( \frac{N}{T} \right )^\theta \big \}}
F( Y_{\lfloor s \rfloor_{T/N}} ) 
\,
ds
\\
&
+
\int_0^t e^{(t- \lfloor s \rfloor_{T/N})A}
\,
\1_{ \big\{ \| F(Y_{\lfloor s \rfloor_{T/N}} )\|_H 
+ 
 \| B( Y_{\lfloor s \rfloor_{T/N}} )  \|_{HS(U, H)} 
\leq \left ( \frac{N}{T} \right )^\theta \big \}}
\,
 B(Y_{\lfloor s \rfloor_{T/N}}) \,dW_s
 ,
\end{split}
\end{equation}
and
assume that for all $x \in H_\gamma$ it holds that
$
\langle x, F(x)\rangle_H
+
\tfrac{p-1}{2}
\left \|B (x) \right \|_{\HS(U,H)}^2
\leq
\d \, (1 + \left \| x \right\|_H^2).
$
\subsection{Strong a priori moment bounds for 
nonlinearities-stopped schemes}
\begin{lemma}
\label{alternative}
Assume the setting in Section \ref{setting1} and
let $x\in H, h\in (0,T], t\in [0,h].$
Then 
\begin{equation}
\label{prop50}
\begin{split}
&
\E
\bigg [
\Big \| 
e^{tA}
\Big (x+ \, 
\1_{ \left \{ \left \|F(x)\right \|_H+\left \|B(x) \right \|_{HS(U,H)}\leq  h^{-\theta} \right\} }\, 
\left [t F(x)+
\smallint\nolimits_0^t B(x) \, dW_s
\right]
\Big )
\Big \|_H^p
\bigg]
\leq
e^{ K t } \left( \| x \|_H^p + K t \right)
.
\end{split}
\end{equation}
\end{lemma}
\begin{proof}[Proof
of Lemma~\ref{alternative}]
W.l.o.g.\ we assume that
$
\left \| F(x) \right \|_H
+
\left \| B(x) \right \|_{\HS( U, H ) }
\leq
 h^{-\theta}$
(otherwise \eqref{prop50} is clear).
Let $Y^x \colon [0, T] \times \Omega \to H, x\in H,$ be stochastic processes
such that for all $t\in [0, T], x\in H$ it holds $\P$-a.s.\,that
$Y_t^x 
= 
x 
+ 
\int_0^t F(x)\, ds + \int_0^t B(x)\, dW_s
$
and
let $f \colon H \to \R$ be the function with the property that for all $x\in H$ it holds that
$f(x) = \left \| x \right \|_H^p$.
Then $f$ 
is twice continuously differentiable and for all
$ x, v, w \in H $ it holds that
\begin{equation}
\begin{split}
f'(x)(v)
&=
p
\left \| x \right \|_H^{p-2}
\langle x, v \rangle_H, 
\\
f''(x)(v,w)
&= 
\begin{cases}
p
\left \| x \right \|_H^{p-2}
\langle v, w \rangle_H
+
p(p-2)
\left\| x \right\|_H^{p-4}
\langle x, v \rangle_H
\langle x, w \rangle_H
,
& x\neq 0 \\
p
\left \| x \right \|_H^{p-2}
\langle v, w \rangle_H
,
& x = 0
.
\end{cases}
\end{split}
\end{equation}
It\^o's formula hence 
proves that it holds $ \P $-a.s.\,that
\begin{align}
\label{Ito}
\nonumber
\left \| Y_t^x \right \|_H^p
&
=
\left \| Y_0^x \right \|_H^p
+
\int_0^t
f'(Y_s^x) 
F(x)
\,ds
+
\tfrac{1}{2}
\sum_{u \in \mathbb{U}}
f''( Y_s^x )(B(x) u, B(x)u)
\, ds
+
\int_0^t 
f'(Y_s^x) B(x)\, dW_s
\\
&
=
\left \| Y_0^x \right \|_H^p
+
\int_0^t
p
\left \| Y_s^x \right \|_H^{p-2}
\langle Y_s^x, F(x) \rangle_H
\,
ds
+
\int_0^t 
p
\left \| Y_s^x \right \|_H^{p-2}
\langle Y_s^x, B(x)\,dW_s \rangle_H
\\
\nonumber
&
\quad
+
\tfrac{1}{2}
\int_0^t
\sum_{u\in \mathbb{U}} 
\Big[
p
\left \| Y_s^x \right \|_H^{p-2}
\langle B(x)u, B(x)u \rangle_H
+
p(p-2)
\,
\1_{\{ Y_s^x \neq 0 \}}
\,
\left \| Y_s^x \right \|_H^{p-4}
| \langle Y_s^x, B(x)u \rangle_H | ^2
\Big]
\, ds
.
\end{align}
The triangle inequality, Fubini's theorem, and the H\"older inequality therefore show that
\begin{equation}
\label{FIRST}
\begin{split}
&
\E
[
\left \| Y_t^x \right \|_H^p
]
\\
&
\leq
\E
[ \left \| Y_0^x \right \|_H^p]
+
p
\int_0^t
\E
\!
\left [
\left \| Y_s^x \right \|_H^{p-2}
\langle Y_s^x, F(x) \rangle_H
\right]\!
\, ds
+
\tfrac{p(p-1)}{2}
\int_0^t
\E
\Big[
\! \left \| Y_s^x \right \|_H^{p-2}
\smallsum_{u \in \mathbb{U}}
\left \|
B(x) u
\right \|_H^2\!
\Big ]
\,
ds
\\
&
=
\| x \|_H^p
+
p
\left[
\langle x, F(x) \rangle_H
+
\tfrac{p-1}{2}
\left \| B(x) \right \|_{\HS(U, H)}^2
\right]
\!
\int_0^t
\E\!
\left [
\left \| Y_s^x \right\|_H^{p-2}
\right ]
\, 
ds
\\
&
\quad
+
p
\int_0^t
\E
\!
\left[
\left \| Y_s^x \right \|_H^{p-2}
\left < F(x) s + \smallint\nolimits_0^s B(x)\, dW_r, F(x) \right >_H
\right]\!
\, 
ds
\\
&
\leq
\| x \|_H^p
+
p
\left [
\langle x, F(x) \rangle_H
+
\tfrac{p-1}{2}
\left \| B(x) \right \|_{\HS(U, H)}^2
\right ]^+
\!
\int_0^t
\left \| Y_s^x \right \|_{L^p (\P ; H)}^{p-2} ds
\\
&
\quad
+
p
\int_0^t
\left \| Y_s^x \right \|_{L^p(\P; H)}^{p-2}
\left \|
\left <
F(x) s + \smallint\nolimits_0^s B(x)\, dW_r,
F(x)
\right>_H
\right \|_{L^{\nicefrac{p}{2}}(\P; H)}
\,
ds
.
\end{split}
\end{equation}
In the next step we apply the Cauchy-Schwarz inequality and 
the Burkholder-Davis-Gundy type inequality in Lemma~$7.7$ in Da Prato \& Zabczyk \cite{dz92}
to \eqref{FIRST} to obtain that
\begin{equation}
\begin{split}
\E\!
\left [
\left \| Y_t^x \right \|_H^p
\right]
&
\leq
\left \| x \right \|_H^p
+
p \d
\left [
1 
+
\left \| x \right \|_H^2
\right ]
\int_0^t
\left \|
Y_s^x
\right \|_{L^p (\P; H)}^{p-2}\!
\, ds
\\
&
\quad
+
p
\int_0^t
\left \| Y_s^x \right \|_{L^p (\P; H)}^{p-2}
\left [
\left \| F(x) \right \|_H^2 s
+
\sqrt{s} \tfrac{p}{2}
\left \| B(x)\right \|_{\HS(U, H)}
\left \| F(x) \right \|_H
\right ]
\, 
ds
.
\end{split}
\end{equation}
Young's inequality hence proves that
\begin{align}
\nonumber
&
\E\!
\left [
\left \| Y_t^x \right \|_H^p
\right]
\leq
\left \| x \right \|_H^p
+
p
\int_0^t
\left [
\tfrac{2}{p}
\d^{\frac{p}{2}}
+
\tfrac{p-2}{p}
\left \|
Y_s^x
\right \|_{L^p (\P; H)}^p
\right ]
\,
ds
+
p
\int_0^t
\left [
\tfrac{2}{p}
\d^{\frac{p}{2}}
\left \| x \right \|_H^p
+
\tfrac{p-2}{p}
\left \|
Y_s^x
\right \|_{L^p (\P; H)}^p
\right ]\!
\,
ds
\\
\nonumber
&
\quad
+
p
\int_0^t
\tfrac{p-2}{p}
\left \| Y_s^x \right \|_{L^p (\P; H)}^p
+
\tfrac{2}{p}
\left [
\left \| F(x) \right \|_H^2 s
+
\sqrt{s} 
\tfrac{p}{2}
\left \| B(x) \right \|_{\HS(U,H)}
\left \| F(x) \right \|_H
\right ]^{\frac{p}{2}}\!
\, 
ds
\\
\nonumber
&
=
\left \| x \right \|_H^p
+
\int_0^t
2 \d^{\frac{p}{2}}
+
(p-2)
\left \|
Y_s^x
\right \|_{L^p (\P; H)}^p
\,
ds
+
\int_0^t
2
\left \| x \right \|_H^p
\d^{\frac{p}{2}}
+
(p-2)
\left \|
Y_s^x
\right \|_{L^p (\P; H)}^p
\,
ds
\\
&
\quad
+
\int_0^t
(p-2)
\left \| Y_s^x \right \|_{L^p (\P; H)}^p
+
2
\left [
\left \| F(x) \right \|_H^2 s
+
\sqrt{s} 
\tfrac{p}{2}
\left \| B(x)\right \|_{\HS(U, H)} 
\left \| F(x) \right \|_H
\right ]^{\frac{p}{2}}\!
\, 
ds
\\
\nonumber
&
=
\left \| x \right \|_H^p
+
3(p-2)
\int_0^t
\left \| Y_s^x \right \|_{L^p (\P; H)}^p\!
\,
ds
\\
\nonumber
&
\quad
+
\int_0^t
2 \d^{\frac{p}{2}}
( 1 + \left \| x \right \|_H^p )
+
2
\left [
\left \| F(x) \right \|_H^2 s
+
\sqrt{s} 
\tfrac{p}{2}
\left \| B(x) \right \|_{\HS(U, H)}
\left \| F(x) \right \|_H
\right ]^{\frac{p}{2}}\!
\,
ds
\\
\nonumber
&
\leq
\left \| x \right \|_H^p
+
3(p-2)
\int_0^t
\left \| Y_s^x \right \|_{L^p(\P; H)}^p\!
\,
ds
+
\left(
2 \d^{\frac{p}{2}}
\left [
1
+
\left \| x \right \|_H^p
\right ] 
+
2
\left [
t^{1 - 2\theta} 
+
\tfrac{p}{2}
t^{\frac{1}{2} - 2 \theta}
\right ]^{\frac{p}{2}}
\right)
t
.
\end{align}
Gronwall's lemma therefore shows that
\begin{equation}
\label{eq19}
\begin{split}
\E\!
\left [
\| Y_t^x \|_H^p 
\right ]
&
\leq
e^{3(p-2) t}
\left [
\left \| x \right \|_H^p
+
t
\left(
2 \d^{\frac{p}{2}}
\left [
1
+
\left \| x \right \|_H^p
\right ] 
+
2
\big [
T^{1 - 2\theta} 
+
\tfrac{p}{2}
T^{\frac{1}{2} - 2 \theta}
\big ]^{\frac{p}{2}}
\right)
\right ]
\\
&
=
e^{3(p-2) t}
\left(
\left (
1
+
2
t
\d^{\nicefrac{p}{2}}
\right)
\left \| x \right \|_H^p
+
t
\left[
2 \d^{\frac{p}{2}}
+
2
\big [
T^{1 - 2\theta} 
+
\tfrac{p}{2}
T^{\frac{1}{2} - 2 \theta}
\big ]^{\frac{p}{2}}
\right ]
\right)
.
\end{split}
\end{equation}
Note that for all $t\in [0,T]$ it holds that $1+2tC^{\nicefrac{p}{2}}\leq e^{2tC^{\nicefrac{p}{2}}}.$
Combining this with \eqref{eq19} shows that for all $t\in [0,T]$ it holds that
\begin{equation}
\begin{split}
&
\E\!
\left [
\left \| Y_t^x \right \|_H^p 
\right ]
\leq
e^{(3(p-2)+2C^{\nicefrac{p}{2}})t}
\left(
\| x  \|_H^p
+
t
\left(
2C^{\nicefrac{p}{2}}
+
2
\big [
T^{1 - 2\theta} 
+
\tfrac{p}{2}
T^{\frac{1}{2} - 2 \theta}
\big ]^{\frac{p}{2}}
\right )
\right)
.
\end{split}
\end{equation}
The proof of Lemma \ref{alternative} is thus completed.
\end{proof}
\begin{lemma}
\label{Lemma1}
Assume the setting in Section \ref{setting1} and let $t\in [0,T].$
Then
\begin{equation}
\label{apriori}
\E
\!
\left [
\left \|
Y_t
\right \|_H^p
\right ]
\leq
\left (
\E
\!
\left [
\left \|
Y_0
\right \|_H^p
\right ]
+
K t
\right )
e^{K t}
.
\end{equation}
\end{lemma}
\begin{proof}[Proof
of Lemma~\ref{Lemma1}]
Lemma \ref{alternative} implies that
for all $n\in \{ 1,2,\ldots, N\}$
it holds that
\begin{equation}
\begin{split}
&
\E
\Big[ \big\|Y_{\frac{n}{N}T} \big\|_H^p \Big]
\leq
e^{ \nicefrac{ K T}{N} }
\E
\!
\left(
\left [
\big \| Y_{\frac{n-1}{N}T} \big \|_H^p
\right]
+
K
\tfrac{T}{N}
\right )
\\
&
\leq
e^{ \nicefrac{ K T}{N} }
\left (
\left(
e^{\nicefrac{K T}{N} }
\E
\!
\left [
\big \| Y_{\frac{n-2}{N}T} \big \|_H^p
\right ]
+
K
\tfrac{T}{N}
\right)
+
K
\tfrac{T}{N}
\right )
\leq
\ldots
\leq
\E
\!
\left [
\left \|
Y_0
\right \|_H^p
\right ]
e^{\nicefrac{T K n}{N} }
+
K \tfrac{T}{N} 
\sum_{j=1}^{n}
e^{\nicefrac{T K j}{N}}
.
\end{split}
\end{equation}
Again Lemma \ref{alternative} hence proves that
\begin{equation}
\begin{split}
\E
\!
\left [
\left\|
Y_t 
\right \|_H^p
\right ]
&
=
\E
\bigg [
\Big \|
e^{(t-\lfloor t \rfloor_{T/N} )A}
\Big(
Y_{\lfloor t \rfloor_{T/N} }
+
\,
\1_{ \big \{  \| F(Y_{\lfloor t \rfloor_{T/N} } ) \|_H 
+ 
\| B(Y_{\lfloor t \rfloor_{T/N}})  \|_{HS(U, H)} 
\,
\leq \left ( \frac{N}{T} \right)^\theta \big \}}
\,
\\
&
\quad
\cdot
\big [
( t - \lfloor t \rfloor_{T/N} )
F( Y_{\lfloor t \rfloor_{T/N} } ) 
+
\smallint\nolimits_{\lfloor t \rfloor_{T/N}}^t
B(Y_{\lfloor t \rfloor_{T/N}})\, dW_s
\big ]
\Big)
\Big \|_H^p
\bigg]
\\
&
\leq
e^{K ( t - \lfloor t \rfloor_{T/N} )}
\left(
\E
\!
\left [
\big \| Y_{\lfloor t \rfloor_{T/N} } \big \|_H^p
\right ]
+
K
( t - \lfloor t \rfloor_{T/N} )
\right)
\\
&
\leq
e^{K(t-\lfloor t \rfloor_{T/N})}
\bigg[
\E
\!
\left [
\left \|
Y_0
\right \|_H^p
\right ]
e^{K \lfloor t \rfloor_{T/N} }
+
K \tfrac{T}{N} 
\sum_{j=1}^{ \lfloor t \rfloor_{T/N} \frac{N}{T}  }
e^{\nicefrac{TKj}{N} }
+
K
( t - \lfloor t \rfloor_{T/N} )
\bigg]
\\
&
\leq
\E
\!
\left [
\left \|
Y_0
\right \|_H^p
\right ]
e^{K t}
+ 
K t e^{K t} 
.
\end{split}
\end{equation}
The proof of Lemma \ref{Lemma1} is thus completed.
\end{proof}
\section{Strengthened strong a priori moment bounds based on bootstrap-type 
arguments}
\label{section3}
In this section we use the strong a priori moment bounds established in 
Section \ref{section2} to derive appropriately strengthened strong a priori moment 
bounds for numerical approximation processes and solution processes of 
SPDEs; see Lemma \ref{bootstrap} and Lemma \ref{more.improved} below. The proofs of Lemma \ref{bootstrap} and 
Lemma \ref{more.improved} are based on suitable bootstrap-type arguments. Bootstrap-type 
arguments of this kind have been intensively used in the literature to 
establish regularity properties of solutions of (stochastic) evolution 
equations.
\subsection{Setting}
\label{setting2}
Assume the setting in Section \ref{Main.setting}, 
let
$
F\in \mathcal{M} 
\big( \mathcal{B}(H_\gamma), 
\mathcal{B}(H) 
\big)
,
B\in \mathcal{M}
\big(
\mathcal{B}(H_\gamma),
\mathcal{B}(\HS(U, H) )
\big)
,
$
$\kappa \in \mathcal{M}\big(\mathcal{B}([0,T]), \mathcal{B}([0,T])\big)$
satisfy that for all $s\in[0,T]$ it holds that $\kappa(s)\leq s,$
and
let $Y, Z \colon  [0,T] \times \Omega \to H_\gamma$ be $ ( \mathcal{F}_t )_{ t \in [0,T] } $-adapted stochastic processes
such that for all $t\in [0,T]$ it holds $\P$-a.s.\,that
$
\int_0^t \|e^{(t-\kappa(s))A}F(Z_s)\|_H 
+ 
\|e^{(t-\kappa(s))A}B(Z_s)\|_{\HS(U,H)}^2
\,ds
<\infty$ and 
\begin{equation}
\begin{split}
Y_t 
=
e^{tA} \xi
&
+
\int_0^t
e^{(t- \kappa(s))A}
F( Z_s ) 
\,
ds
+
\int_0^t e^{(t- \kappa(s) )A}
 B(Z_s) \,dW_s
.
\end{split}
\end{equation}
\subsection{A first bootstrap-type argument for a priori bounds}
\begin{lemma}
\label{bootstrap}
Assume the setting in Section \ref{setting2}, 
let $ t \in [0,T] $,
assume that
$ \gamma < \min\{ 1 - \alpha , \nicefrac{ 1 }{ 2 } - \beta \} $,
and 
assume that for all $ x \in H_\gamma$ it holds that
$
  \max\{ 
    \left\| F(x) \right\|_{ H_{ - \alpha } } 
    ,
    \left\| B(x) \right\|_{\HS(U, H_{-\beta})}
  \}
  \leq C ( 1 + \left \| x \right \|_H^a )
$.
Then 
\begin{equation}
\label{improved.bound}
\begin{split}
\| Y_t \|_{L^p(\P; H_\gamma)}
&
\leq
\left \| \xi \right \|_{L^p(\P; H_\gamma)}
+
C
\Big[
\tfrac{ t^{1-(\gamma+\alpha)}}{1-(\gamma+\alpha)} 
+
\sqrt{\tfrac{p(p-1)}{2(1-2(\gamma + \beta))}}
t^{\nicefrac{1}{2}-(\gamma + \beta)}
\Big]
\big[1 + \sup\nolimits_{s\in [0,t]} \| Z_s \|_{L^{p a}(\P; H)}^{a} \big]
.
\end{split}
\end{equation}
\end{lemma}
\begin{proof}[Proof
of Lemma~\ref{bootstrap}]
The triangle inequality and the
Burkholder-Davis-Gundy type inequality in Lemma~7.7 in Da Prato \& Zabczyk \cite{dz92}
imply 
\begin{equation}
\begin{split}
\label{ena}
\| Y_t  \|_{L^p(\P; H_\gamma)}
&
\leq
\left \|  \xi \right\|_{L^p(\P; H_\gamma)}
+
\int_0^t \big \| e^{( t- \kappa(s)  )A} F(Z_s) \big \| _{L^p(\P; H_\gamma)}  \,ds 
\\
&
\quad
+
\left[
\tfrac{p(p-1)}{2}
\int_0^t \big \| e^{( t- \kappa(s)  )A} 
B(Z_s)\big\| _{L^p(\P; \HS(U ,H_\gamma))}^2 \,ds
\right]^{\nicefrac{1}{2}}
\!\!
.
\end{split}
\end{equation}
Note that, e.g., \cite[Theorem~2.5.34 and Lemma~2.5.35]{Jentzen2014SPDElecturenotes} proves that 
\begin{align}
\label{dva}
&
\int_0^t \big \| e^{( t- \kappa(s)  )A} F(Z_s) \big \| _{L^p(\P; H_\gamma)}  \,ds 
\leq
\int_0^t 
\left \| e^{(t- \kappa(s)  )A} \right \|_{L(H_{-\alpha}, H_\gamma)}
 \| F(Z_s) \|_{L^p(\P; H_{-\alpha})} ds
\\
&
\nonumber
\leq
\int_0^t (t-\kappa(s)  )^{-(\gamma+\alpha)} 
\| F(Z_s) \|_{L^p(\P; H_{-\alpha})} \, ds
\leq
C
\int_0^t (t-s)^{-(\gamma+\alpha)} 
\big \| 1+  \| Z_s \|_H^{a} \big \|_{L^p(\P; \R)} \, ds
\\
&
\nonumber
\leq
C
\int_0^t (t-s)^{-(\gamma+\alpha)}\!\left (1 +  \| Z_s \|_{ L^{p a} (\P; H)}^{ a}\right )\! \, ds
\leq
\tfrac{ C }{1-(\gamma+\alpha)}
t^{1-(\gamma+\alpha)}
\!
\left (
1
+
\sup\nolimits_{s\in [0,t]}
\| Z_s \|_{L^{p a}(\P; H)}^{ a }
\right)
\end{align}
and 
\begin{equation}
\begin{split}
\label{tri}
&
\int_0^t 
\big \| e^{( t-\kappa(s)  )A} B(Z_s)
\big\| _{L^p(\P; \HS(U, H_\gamma) )}^2
\,
ds
\leq 
\int_0^t 
\left \| e^{(t- \kappa(s)  )A} \right \|_{L(H_{-\beta}, H_{\gamma})}^2
 \| B( Z_s)  \|_{L^p(\P; \HS(U, H_{-\beta}) )}^2
\, ds
\\
&
\leq
\int_0^t
\tfrac{C^2}
{
(t- \kappa(s)  )^{2(\gamma + \beta)}
}
\big \| 1+  \|Z_s  \|_H^a \big \|_{L^p(\P; \R )}^2
\, ds
\leq
\tfrac{C^2}{1-2(\gamma+ \beta)}
t^{1-2(\gamma+ \beta)}
\left (
1
+
\sup\nolimits_{s\in [0,t]}  \| Z_s \|_{L^{p a}(\P; H)}^a
\right)^2
.
\end{split}
\end{equation}
Putting \eqref{dva} and \eqref{tri} into \eqref{ena} yields \eqref{improved.bound}.
The proof of Lemma \ref{bootstrap} is now completed.
\end{proof}

\subsection{A second bootstrap-type argument for a priori bounds}
\begin{lemma}
\label{more.improved}
Assume the setting in Section \ref{setting2},
let $ \eta\in[0,  \nicefrac{1}{2} ),$
$t\in (0,T],$
and
assume that
$
\sup\nolimits_{s\in [0,T]}
\big \| 
\|
F(Z_s)
\|_H
+
\|
B(Z_s)\|_{\HS(U,H)}
\big \|_{L^p(\P; \R)}
<\infty
.
$
Then 
it holds that 
$\P(Y_t\in \cap_{r\in (-\infty, \nicefrac{1}{2})} H_r)=1$ and
\begin{equation}
\label{sova0}
\begin{split}
\| Y_t \|_{L^p(\P; H_\eta)}
&
\leq
\| e^{tA} \xi \|_{L^p(\P; H_\eta)}
+ 
\tfrac{1}{1-\eta}
t^{1-\eta}
\sup\nolimits_{s\in [0,t]}
\|
F(Z_s)
\|_{L^p(\P; H)} 
\\
&
\quad
+
\sqrt{\tfrac{p(p-1)}{2(1 - 2 \eta)}}
t^{\nicefrac{1}{2} -\eta}
\sup\nolimits_{s\in [0,t]}
\|
B(Z_s)\|_{L^p(\P; \HS(U,H))}
.
\end{split}
\end{equation}
\end{lemma}
\begin{proof}[Proof of Lemma~\ref{more.improved}]
First observe that, e.g., Theorem $2.5.34$ in \cite{Jentzen2014SPDElecturenotes} proves that
for all $r\in [0,\nicefrac{1}{2})$ it holds that
\begin{equation}
\label{sova1}
\begin{split}
&
\int_0^t \| e^{ ( t - s ) A } F( Z_s ) \|_{ L^p( \P; H_r ) } ds 
\leq
\sup\nolimits_{s\in[0,T]}
\| F(Z_s)\|_H
\int_0^t \| (-A)^{r} e^{(t-s)A}\|_{L(H)} \, ds
\\
&
\leq
\sup\nolimits_{s\in[0,T]}
\| F(Z_s)\|_H
\int_0^t (t-s)^{-r} \, ds
=
\sup\nolimits_{s\in[0,T]}
\| F(Z_s)\|_H
\tfrac{t^{1-r}}{1-r}<\infty
\end{split}
\end{equation}
and
\begin{equation}
\label{sova2}
\begin{split}
&
\sqrt{ \int_0^t \| e^{ ( t - s ) A } B( 
Z_s ) \|_{\HS(U, H_r)}^2 ds } 
\leq
\sup\nolimits_{s\in [0,T]}
\| B(Z_s ) \|_{\HS(U, H)}
\sqrt{ \int_0^t \| (-A)^r e^{ ( t - s ) A }\|_{L(H)}^2 ds }
\\
&
\leq
\sup\nolimits_{s\in [0,T]}
\| B(Z_s ) \|_{\HS(U, H)}
\sqrt{ \int_0^t (t-s)^{-2r} ds } 
=
\sup\nolimits_{s\in [0,T]}
\| B(Z_s ) \|_{\HS(U, H)}
\tfrac{t^{\nicefrac{1}{2}-r}}{\sqrt{1-2r}}<\infty
.
\end{split}
\end{equation}
Next note that \eqref{sova1} and \eqref{sova2} prove that
$\P(Y_t\in \cap_{r\in (-\infty, \nicefrac{1}{2})} H_r)=1.$
Moreover, observe that \eqref{sova1}, \eqref{sova2}, and the
Burkholder-Davis-Gundy type inequality in Lemma $7.7$ in Da Prato \& Zabczyk~\cite{dz92}
imply \eqref{sova0}.
The proof of Lemma~\ref{more.improved} is thus completed.
\end{proof}
\section{Strong temporal error estimates for 
nonlinearities-stopped schemes}
\label{section4}
In this section we estimate temporal discretization errors of nonlinearities-stopped 
exponential Euler approximations; see Corollary~\ref{corollary.combined}
below. For this we introduce similar as in Jentzen \& Kurniawan~\cite[(11),
(70), (136)]{KurniawanJentzen2015Arxiv} suitable semilinear integrated counterparts of the
nonlinearities-stopped exponential Euler approximations. Then we estimate the 
differences of the nonlinearities-stopped exponential Euler approximations and
their semilinear counterparts in a straightforward way (see Lemma \ref{Lemma12}) and we employ the perturbation estimate in 
Theorem~2.10 in Hutzenthaler \& Jentzen~\cite{HutzenthalerJentzen2014PerturbationArxiv} to estimate the differences of
the solution process of the considered SPDE and the semilinear
integrated counterparts of the nonlinearities-stopped exponential Euler
approximations (see Lemma \ref{Lemma13}). Combining Lemma \ref{Lemma12} and Lemma \ref{Lemma13} with
the triangle inequality will then immediately result in Corollary \ref{corollary.combined}.
\subsection{Setting}
\label{setting3}
Assume the setting in Section \ref{Main.setting},
let 
$N\in \N,$
$
F\in \mathcal{M} 
\big( \mathcal{B}(H_\gamma), 
\mathcal{B}(H) 
\big)
,
B\in \mathcal{M}
\big(
\mathcal{B}(H_\gamma),
\mathcal{B}(\HS(U, H) )
\big)
,
$
assume that 
$\gamma < \min \{ 1-\alpha, \nicefrac{1}{2}-\beta\},$
let $X \colon [0,T]\times \Omega \to H_\gamma$ be 
an $(\mathcal{F}_t)_{t\in [0,T]}$-adapted stochastic process
with continuous sample paths
such that for all $t\in [0, T]$ it holds $\P$-a.s.\,that
\begin{equation}
X_t = e^{tA} \xi + \int_0^t e^{(t-s)A}F(X_s) \, ds
+
\int_0^t e^{(t-s)A} B(X_s) \, dW_s,
\end{equation}
let $Y \colon  [0,T] \times \Omega \to H_\gamma$ be an $(\mathcal{F}_t)_{t\in [0,T]}$-adapted stochastic process 
such that for all $t\in [0,T]$ it holds $\P$-a.s.\,that
\begin{equation}
\begin{split}
Y_t
=
e^{tA} \xi
&
+
\int_0^t
e^{(t- \lfloor s \rfloor_{T/N} )A}
\,
\1_{ \big \{ 
 \| F(Y_{\lfloor s \rfloor_{T/N} } ) \|_H 
+ 
\| B( Y_{\lfloor s \rfloor_{T/N} }  ) \|_{HS(U, H)} 
\,
\leq 
\left ( \frac{N}{T} \right )^\theta \big \}}
\,
F( Y_{\lfloor s \rfloor_{T/N} } ) 
\,
ds
\\
&
+
\int_0^t e^{(t- \lfloor s \rfloor_{T/N} )A}
\,
\1_{\big \{ 
\| F(Y_{\lfloor s \rfloor_{T/N} } )\|_H 
+ 
\| B(Y_{\lfloor s \rfloor_{T/N} } )  \|_{HS(U, H)} 
\leq 
\left ( \frac{N}{T} \right )^\theta \big\}}
\,
 B(Y_{\lfloor s \rfloor_{T/N} } ) \,dW_s
,
\end{split}
\end{equation}
and
let $\bar Y \colon [0, T]\times \Omega \to H_\gamma$
be an $(\mathcal{F}_t)_{t\in [0,T]}$-adapted stochastic process 
with continuous sample paths
such
that
for all $t\in [0, T]$ it holds  
$\P$-a.s.\,that
\begin{equation}
\begin{split}
\bar Y_t
= 
e^{tA} \xi 
&
+ 
\int_0^t e^{(t-s)A}
F( Y_{\lfloor s \rfloor_{T/N} })
\, ds
+
\int_0^t e^{(t-s)A} 
B(Y_{\lfloor s \rfloor_{T/N} } ) \, d W_s.
\end{split}
\end{equation}
\subsection{Analysis of the differences between 
nonlinearities-stopped exponential Euler approximations and their 
semilinear counterparts}
\begin{lemma}
\label{Lemma.Markov}
Assume the setting in Section \ref{Main.setting} and let 
$Z \in \mathcal{M} \big( \mathcal{F}, \mathcal{B}(H_\gamma) \big), \kappa \in [\nicefrac{\theta}{p},\infty).$
Then 
\begin{equation}
\Big \|
1-
\,
\1_
{
\{
\left \| F( Z ) \right\|_H
+
\left \| B( Z ) \right\|_{HS(U, H)}
\leq
\left (
\frac{N}{T}
\right )^\theta
\}
}
\,
\Big\|_{L^{p}(\P; \R)}
\leq
\left ( \tfrac{T}{N} \right ) ^\kappa 
\left (
\left \|F( Z ) \right \|_{L^{\nicefrac{\kappa p}{\theta}}(\P; H)}
+
\| B( Z ) \|_{L^{\nicefrac{\kappa p}{\theta}}(\P; \HS(U,H))}
\right )^{\frac{\kappa}{\theta}}
.
\end{equation}
\end{lemma}
\begin{proof}[Proof
of Lemma~\ref{Lemma.Markov}]
Markov's inequality shows that
\begin{equation}
\begin{split}
&
\big \|
1
-
\,
\1_
{
\left
\{
\left \| F( Z)\right\|_H
+
\left \| B( Z ) \right \|_{HS(U, H)}
\leq
\left (
\frac{N}{T}
\right )^\theta
\right
\}
}
\,
\big \|_{L^{p}(\P; H)}
=
\left |
\P
\!
\left[
\big(
\|F(Z) \|_H
+
 \|B(Z) \|_{HS(U, H)}
\big)^{\nicefrac{\kappa p}{\theta}}
>
\left ( \tfrac{N}{T} \right )^{ \kappa p }
\right]
\right |^{\frac{1}{p}}
\\
&
\leq
\left ( \tfrac{T}{N} \right ) ^{\kappa  }
\left |
\E
\Big[
\big(
\left \|F(Z)\right \|_H
+
\left \|B(Z) \right \|_{HS(U, H)}
\big)^{\frac{\kappa p}{\theta}}
\Big]
\right |^{\frac{1}{p}}
=
\left ( \tfrac{T}{N} \right ) ^{\kappa }
\left \|
\|F(Z)\|_H
+
\|B(Z)  \|_{HS(U, H)}
\right \|_{L^{\nicefrac{\kappa p}{\theta}}(\P; \R)}^{\frac{\kappa}{\theta}}
\\
&
\leq
\left ( \tfrac{T}{N} \right ) ^\kappa 
\left (
\|F(Z)  \|_{L^{\nicefrac{\kappa p}{\theta}}(\P; H)}
+
 \| B(Z) \|_{L^{\nicefrac{\kappa p}{\theta}}(\P; \HS(U,H))}
\right )^{\frac{\kappa}{\theta}}
.
\end{split}
\end{equation}
The proof of Lemma \ref{Lemma.Markov} is now completed.
\end{proof}
\begin{lemma}
\label{Lemma12}
Assume the setting in Section \ref{setting3} and let $\rho \in [\delta, 1-\delta),$
$t\in [0,T].$
Then
\begin{equation}
\label{enica}
\begin{split}
&
\| Y_t - \bar Y_t \|_{L^p(\P; H_\delta)}
\\
&
\leq
\tfrac{\max \{1, T^{\nicefrac{3}{2}} \}}{1-\delta-\rho}
N^{\nicefrac{(\delta - \rho)}{2}  }
\Big [
1
+
\!
\sup_{s\in [0,T]} \|F( Y_s )  \|_{L^{\nicefrac{ p}{\theta}}(\P; H)}
+
\tfrac{ 
\sqrt{ p ( p - 1 ) } }{ \sqrt{ 2 } }
\!
\sup_{s\in [0,T]} \| B( Y_s ) \|_{L^{\nicefrac{ p}{\theta}}(\P; \HS(U,H))}
\Big]^{1 + \frac{1}{ 2 \theta}}
\!\!\!\!
.
\end{split}
\end{equation}
\end{lemma}
\begin{proof}[Proof of Lemma~\ref{Lemma12}]
Note that
\begin{equation}
\begin{split}
\label{first_estimate}
&
\| Y_t - \bar Y_t  \|_{L^p(\P; H_\del)}
\leq
\int_0^t
\Big \| 
\big (e^{(t-\lfloor s \rfloor_{T/N}) A} - e^{(t-s)A} \big )
F(Y_{\lfloor s \rfloor_{T/N}})
\Big \|_{L^p(\P; H_\del)}
\,
ds
\\
&
+
\int_0^t 
\Big \| 
e^{(t-s)A}
\big (1 
-
\,
\1_{
\big \{
\| F(Y_{\lfloor s \rfloor_{T/N}}) \|_H + \| B(Y_{\lfloor s \rfloor_{T/N}} )\|_{\HS(U,H)}\leq \left (\frac{N}{T} \right)^\theta
\big \}
}
\,
\big)
F(Y_{\lfloor s \rfloor_{T/N}})
\,
\Big \|_{L^p(\P; H_\del)}
ds
\\
&
+
\left \| 
\int_0^t
\big ( e^{(t-\lfloor s \rfloor_{T/N}) A} - e^{(t-s)A} \big )
\,
\1_{
\big \{
\| F(Y_{\lfloor s \rfloor_{T/N}}) \|_H + \| B(Y_{\lfloor s \rfloor_{T/N}} )\|_{\HS(U,H)}\leq \left (\frac{N}{T} \right)^\theta
\big \}
}
\,
B( Y_{\lfloor s \rfloor_{T/N}} )
\,
dW_s
\right \|_{L^p(\P; H_\del)}
\\
&
+
\left \| 
\int_0^t
e^{(t-s)A}
\big (
1
-
\,
\1_{
\big \{
\| F(Y_{\lfloor s \rfloor_{T/N}}) \|_H + \| B(Y_{\lfloor s \rfloor_{T/N}} )\|_{\HS(U,H)}\leq \left (\frac{N}{T} \right)^\theta
\big \}
}
\,
\big)
B( Y_{\lfloor s \rfloor_{T/N}} )
\,
dW_s
\right \|_{L^p(\P; H_\del)}
.
\end{split}
\end{equation}
Moreover, observe that, e.g., \cite[Theorem $2.5.34$ and Lemma $2.5.35$]{Jentzen2014SPDElecturenotes} implies that
\begin{equation}
\label{long}
\begin{split}
&
\int_0^t
\Big \|
\big (
e^{( t - \lfloor s \rfloor_{T/N})A} 
-
e^{( t - s) A}
\big )
F( Y_{\lfloor s \rfloor_{T/N}} ) 
\Big \|_{L^{p}(\P;H_\del)}
\,
ds
\\
&
\leq
\int_0^t
\big \|
(-A)^{\rho+\nicefrac{\delta}{2} } 
e^{(t-s)A}
\big \|_{L(H)}
\big \|
(-A)^{ \nicefrac{\delta}{2}  -\rho}
\big(
\operatorname{Id}_H
-
e^{( s - \lfloor s \rfloor_{T/N})A} 
\big)
\big \|_{L(H)}
\|
F( Y_{\lfloor s \rfloor_{T/N}} ) 
\|_{L^{p}(\P;H)}
\,
ds
\\
&
\leq
\int_0^t
\tfrac{
( s-\lfloor s \rfloor_{T/N} )^{\rho - \nicefrac{\delta}{2} }
}
{
( t-s )^{\rho + \nicefrac{\delta}{2} }
}
 \| F( Y_{\lfloor s \rfloor_{T/N} })  \|_{L^{p}(\P; H)} 
\,
ds
\leq
\tfrac{T^{1-\delta}}{1-\rho- \nicefrac{\delta}{2} }
N^{-\rho+\nicefrac{\delta}{2} }
\sup\nolimits_{s \in [0,T]} \left \| F( Y_s ) \right \|_{L^p(\P;H)}
.
\end{split}
\end{equation}
In addition, note that
H\"older's inequality, e.g., Theorem 2.5.34 in \cite{Jentzen2014SPDElecturenotes},
and
Lemma \ref{Lemma.Markov} prove that
\begin{equation}
\label{zaspan2}
\begin{split}
&
\int_0^t 
\Big \| 
e^{(t-s)A}
\big (1 
-
\,
\1_{
\big \{
\| F(Y_{\lfloor s \rfloor_{T/N}}) \|_H + \| B(Y_{\lfloor s \rfloor_{T/N}} )\|_{\HS(U,H)}\leq \left (\frac{N}{T} \right)^\theta
\big \}
}
\,
\big)
F(Y_{\lfloor s \rfloor_{T/N}})
\,
\Big \|_{L^p(\P; H_\del)}
ds
\\
&
\leq
\int_0^t 
\big \| 
e^{(t-s)A}
F(Y_{\lfloor s \rfloor_{T/N}})
\,
\big \|_{L^{2p}(\P; H_\del)}
\left \| 
1 
-
\,
\1_{
\big \{
\| F(Y_{\lfloor s \rfloor_{T/N}}) \|_H + \| B(Y_{\lfloor s \rfloor_{T/N}} )\|_{\HS(U,H)}\leq \left (\frac{N}{T} \right)^\theta
\big \}
}
\right \|_{L^{2p}(\P; \R)}
ds
\\
&
\leq
\int_0^t 
(t-s)^{-\delta}
\big \| 
F(Y_{\lfloor s \rfloor_{T/N}})
\,
\big \|_{L^{2p}(\P; H)}
\left \| 
1 
-
\,
\1_{
\big \{
\| F(Y_{\lfloor s \rfloor_{T/N}}) \|_H + \| B(Y_{\lfloor s \rfloor_{T/N}} )\|_{\HS(U,H)}\leq \left (\frac{N}{T} \right)^\theta
\big \}
}
\right \|_{L^{2p}(\P; \R)}
ds
\\
&
\leq
\tfrac{
t^{1-\delta}
}
{1-\delta}
\sqrt{ \tfrac{T}{N} } 
\sup_{s\in [0,T]}
\| F(Y_s) \|_{L^{2p}(\P; H)}
\left [
\sup_{s\in [0,T]} \|F( Y_s )  \|_{L^{\nicefrac{ p}{\theta}}(\P; H)}
+
\sup_{s\in [0,T]} \| B( Y_s ) \|_{L^{\nicefrac{ p}{\theta}}(\P; \HS(U,H))}
\right ]^{\frac{1}{ 2 \theta}}
\\
&
\leq
\tfrac{
T^{\nicefrac{3}{2}-\delta}
}
{1-\delta}
N^{-\nicefrac{1}{2}}
\sup_{s\in [0,T]}
\| F(Y_s) \|_{L^{2p}(\P; H)}
\left [
\sup_{s\in [0,T]} \|F( Y_s )  \|_{L^{\nicefrac{ p}{\theta}}(\P; H)}
+
\sup_{s\in [0,T]} \| B( Y_s ) \|_{L^{\nicefrac{ p}{\theta}}(\P; \HS(U,H))}
\right ]^{\frac{1}{ 2 \theta}}
.
\end{split}
\end{equation}
Furthermore,
the Burkholder-Davis-Gundy type inequality in Lemma $7.7$ in Da Prato \& Zabczyk~\cite{dz92}, and, e.g., 
\cite[Theorem $2.5.34$ and Lemma $2.5.35$]{Jentzen2014SPDElecturenotes} show that
\begin{align}
\label{eq48}
\nonumber
&
\Big \| 
\int_0^t
\big ( e^{(t-\lfloor s \rfloor_{T/N}) A} - e^{(t-s)A} \big )
\,
\1_{
\big \{
\| F(Y_{\lfloor s \rfloor_{T/N}}) \|_H + \| B(Y_{\lfloor s \rfloor_{T/N}} )\|_{\HS(U,H)}\leq \left (\frac{N}{T} \right)^\theta
\big \}
}
\,
B( Y_{\lfloor s \rfloor_{T/N}} )
\,
dW_s
\Big \|_{L^p(\P; H_\del)}
\\
\nonumber
&
\leq
\sqrt{\tfrac{p(p-1)}{2}
\int_0^t
\big\|
\big (
e^{(t - \lfloor s \rfloor_{T/N} ) A} 
-
e^{(t-s)A}
\big )
B ( Y_{\lfloor s \rfloor_{T/N}} )
\big \|_{L^{p}(\P;\HS(U,H_\del ))}^2
\, ds
}
\\
\nonumber
&
\leq
\sqrt{\tfrac{p(p-1)}{2}
\!
\int_0^t
\!
\big \|
(-A)^{\frac{\rho + \delta}{2} } 
e^{(t-s)A}
\big \|_{L(H)}^2
\big \|
(-A)^{\frac{-\rho + \delta}{2} }
\big(
\operatorname{Id}_H
-
e^{( s - \lfloor s \rfloor_{T/N})A} 
\big)
\big \|_{L(H)}^2
\|
B( Y_{\lfloor s \rfloor_{T/N}} ) 
\|_{L^{p}(\P;\HS(U, H))}^2
\, 
ds
}
\\
&
\leq
\tfrac{ 
\sqrt{ p ( p - 1 ) } }{ \sqrt{ 2 } }
\sup\nolimits_{s\in [0,T]}
\| B(Y_s) \|_{L^p(\P; \HS(U, H ))}
\sqrt{
\int_0^t
\tfrac{
(  s-\lfloor s \rfloor_{T/N} )^{\rho - \delta}
}
{
( t-s )^{ \rho + \delta}
}
\,
ds
}
\\
\nonumber
&
\leq
\sqrt{\tfrac{p(p-1)}{2(1-\rho-\delta )}}
T^{\nicefrac{1}{2}-\delta}
N^{\frac{- \rho + \delta}{2} }
\sup\nolimits_{s\in [0,T]}
\| B(Y_s) \|_{L^p(\P; \HS(U, H))}
.
\end{align}
Moreover,
the Burkholder-Davis-Gundy type inequality in Lemma $7.7$ in Da Prato \& Zabczyk \cite{dz92},
H\"older's inequality, 
Lemma~\ref{Lemma.Markov}, and, e.g.,
Theorem~2.5.34 in \cite{Jentzen2014SPDElecturenotes} prove that
\begin{equation}
\begin{split}
\label{zaspan1}
&
\tfrac{2}{p(p-1)}
\left \| 
\int_0^t
e^{(t-s)A}
\big (
1
-
\,
\1_{
\big \{
\| F(Y_{\lfloor s \rfloor_{T/N}}) \|_H + \| B(Y_{\lfloor s \rfloor_{T/N}} )\|_{\HS(U,H)}\leq \left (\frac{N}{T} \right)^\theta
\big \}
}
\,
\big)
B( Y_{\lfloor s \rfloor_{T/N}} )
\,
dW_s
\right \|_{L^p(\P; H_\del)}^2
\\
&
\leq
\int_0^t
\Big \| 
e^{(t-s)A}
\big (
1
-
\,
\1_{
\big \{
\| F(Y_{\lfloor s \rfloor_{T/N}}) \|_H + \| B(Y_{\lfloor s \rfloor_{T/N}} )\|_{\HS(U,H)}\leq \left (\frac{N}{T} \right)^\theta
\big \}
}
\,
\big)
B( Y_{\lfloor s \rfloor_{T/N}} )
\Big \|_{L^p(\P; \HS(U, H_\del) )}^2
\,
ds
%
\\
&
\leq
\int_0^t
\big \| 
e^{(t-s)A}
B( Y_{\lfloor s \rfloor_{T/N}} )
\big \|_{L^{2p}(\P; \HS(U, H_\del) )}^2
\big\|
1
-
\,
\1_{
\big \{
\| F(Y_{\lfloor s \rfloor_{T/N}}) \|_H + \| B(Y_{\lfloor s \rfloor_{T/N}} )\|_{\HS(U,H)}\leq \left (\frac{N}{T} \right)^\theta
\big \}
}
\,
\big \|_{L^{2p}(\P; \R )}^2
\,
ds
%
\\
&
\leq
\int_0^t
(t-s)^{-2\delta}
\|
B( Y_{\lfloor s \rfloor_{T/N}} )
\|_{L^{2p}(\P; \HS(U, H) )}^2
\big\|
1
-
\,
\1_{
\big \{
\| F(Y_{\lfloor s \rfloor_{T/N}}) \|_H + \| B(Y_{\lfloor s \rfloor_{T/N}} )\|_{\HS(U,H)}\leq \left (\frac{N}{T} \right)^\theta
\big \}
}
\,
\big \|_{L^{2p}(\P; \R )}^2
\,
ds
%
\\
&
\leq
\tfrac{
T
t^{1-2\delta}
}
{(1-2\delta) N }
\sup_{s\in [0,T]}
\| B(Y_s) \|_{L^{2p}(\P; \HS(U, H) )}^2
\left [
\sup_{s\in [0,T]} \|F( Y_s )  \|_{L^{\nicefrac{ p}{\theta}}(\P; H)}
+
\sup_{s\in [0,T]} \| B( Y_s ) \|_{L^{\nicefrac{ p}{\theta}}(\P; \HS(U,H))}
\right ]^{\frac{1}{ \theta}}
\\
&
\leq
\tfrac{
T^{2(1-\delta)}
}
{N(1-2\delta)}
\sup_{s\in [0,T]}
\| B(Y_s) \|_{L^{2p}(\P; \HS(U, H) )}^2
\left [
\sup_{s\in [0,T]} \|F( Y_s )  \|_{L^{\nicefrac{ p}{\theta}}(\P; H)}
+
\sup_{s\in [0,T]} \| B( Y_s ) \|_{L^{\nicefrac{ p}{\theta}}(\P; \HS(U,H))}
\right ]^{\frac{1}{ \theta}}
.
\end{split}
\end{equation}
Combining \eqref{first_estimate}--\eqref{zaspan1} 
shows that
\begin{equation}
\begin{split}
&
\| Y_t - \bar Y_t \|_{L^p(\P; H_\delta)}
\leq
\tfrac{1}{\sqrt{N}}
\left [
\sup\nolimits_{s\in [0,T]} \|F( Y_s )  \|_{L^{\nicefrac{ p}{\theta}}(\P; H)}
+
\sup\nolimits_{s\in [0,T]} \| B( Y_s ) \|_{L^{\nicefrac{ p}{\theta}}(\P; \HS(U,H))}
\right ]^{\frac{1}{ 2 \theta}}
\\
&
\cdot
\Big[
\tfrac{
T^{\nicefrac{3}{2}-\delta}
}
{1-\delta}
\sup\nolimits_{s\in [0,T]}
\| F(Y_s) \|_{L^{2p}(\P; H)}
+
\sqrt{\tfrac{p(p-1)}{2(1-2\delta)}}
T^{1-\delta}
\sup\nolimits_{s\in [0,T]}
\| B(Y_s) \|_{L^{2p}(\P; \HS(U, H) )}
\Big]
\\
&
+
\tfrac{T^{1-\delta} }{1-\rho - \frac{\delta}{2} }N^{\frac{\delta-2\rho}{2} } 
\!
\sup\nolimits_{s \in [0,T]} \left \| F( Y_s ) \right \|_{L^p(\P;H)}
+
\sqrt{\tfrac{p(p-1)}{2(1-\rho-\delta )}}
T^{\frac{1}{2}-\delta}
N^{\frac{\delta - \rho}{2} }
\sup\nolimits_{s\in [0,T]}
\| B(Y_s) \|_{L^p(\P; \HS(U,H))}
.
\end{split}
\end{equation}
This completes the proof of Lemma \ref{Lemma12}.
\end{proof}

\subsection{Analysis of the differences 
between semilinear integrated nonlinearities-stopped exponential Euler 
approximations and solution processes of stochastic evolution equations}
\begin{lemma}
\label{Lemma13}
Assume the setting in Section \ref{setting3}, let $\rho \in [\delta, 1-\delta),$
$\eta\in [\delta,\tfrac{1}{2}),$
$\varepsilon\in (0,\infty),$
assume that $\sup_{h\in \H} | \lambda_h | <\infty,$
and
assume that for all 
$x,y\in H_\gamma$ 
it holds that
$
\max \{
\| F(x)- F(y) \|_H 
,
\| B(x)  - B(y) \|_{\HS(U, H)}
\}
\leq 
\d \| x - y \|_{H_\del} ( 1 + \| x \|_{H_\gamma}^\c + \| y \|_{H_\gamma}^\c )
$
and
$
\langle x-y, Ax + F(x)- Ay - F(y)\rangle_H + \tfrac{(p-1)(1+\varepsilon)}{2} \| B(x) - B(y) \|_{\HS(U,H)}^2 
\leq C \| x-y\|_H^2.
$
Then
\begin{equation}
\begin{split}
&
\sup\nolimits_{t\in [0,T]}
\| X_t - \bar Y_t \|_{L^p(\P;H)}
\leq
N^{ \delta - \min\{ \eta, \frac{ \rho + \delta }{ 2 } \} }
\tfrac{\max\{1, T^2\}
}
{(1-\delta-\rho)}
\big (C^2(1+\nicefrac{1}{\varepsilon}) p\big)^{\nicefrac{1}{p}}\exp\!\left(\tfrac{T C^2 p ( 1 + 1 / \varepsilon )}{2} \right)
\\
&
\cdot
\bigg(
2
\left [
1+
\!
\sup\nolimits_{s\in [0,T]} \|F( Y_s )  \|_{L^{\nicefrac{ 2p}{\theta}}(\P; H)}
+
\sqrt{p(2p-1)}
\sup\nolimits_{s\in [0,T]} \| B( Y_s ) \|_{L^{\nicefrac{ 2p}{\theta}}(\P; \HS(U,H))}
\right ]^{1+\frac{1}{ 2 \theta}}
\\
&
\quad
+
\sup\nolimits_{s\in [0,T]}
\| Y_s\|_{L^{2p}(\P; H_\eta)}
\bigg)
\left(
1 
+  
\sup\nolimits_{s\in [0,T]} \|\bar Y_s\|_{L^{2pc}(\P; H_\gamma)}^c 
+ 
\sup\nolimits_{s\in [0,T]}
\| Y_s\|_{L^{2pc}(\P; H_\gamma)}^c
\right)
.
\end{split}
\end{equation}
\end{lemma}
\begin{proof}[Proof of Lemma~\ref{Lemma13}]
Throughout this proof let $\chi\in [0,\infty)$ be the real number given by
$\chi=\frac{C(p-1)}{p} 
( 1 + \frac{C(p-2)(1+\nicefrac{1}{\varepsilon})}{2} ) .$
We intend to prove Lemma~\ref{Lemma13} through an application of Theorem~2.10 in 
Hutzenthaler \& Jentzen~\cite{HutzenthalerJentzen2014PerturbationArxiv}. To this end we now check the assumptions of Theorem 2.10 in 
Hutzenthaler \& Jentzen~\cite{HutzenthalerJentzen2014PerturbationArxiv}.
Let $\tilde X\colon [0,T]\times \Omega \to H_\gamma$ be the stochastic process which satisfies
that for all $s\in [0,T]$ it holds that
$\tilde X_s =e^{-s A} X_s.$
It\^o's formula then proves that for all $s\in [0,T]$ it holds $\P$-a.s.\,that
$
X_s= e^{sA} \tilde X_s=
\xi + \int_0^s [ A X_u + F(X_u) ] \, du + \int_0^s B(X_u)\, dW_u
.
$
Similar we see that for all $s\in [0,T]$ it holds $\P$-a.s.\,that
$
\bar Y_s= \xi + \int_0^s [ A \bar Y_u + F(Y_{\lfloor u \rfloor_{T/N}}) ] \, du + \int_0^s B( Y_{\lfloor u \rfloor_{T/N}})\, dW_u
.
$
Next observe that for all $x\in H_\gamma$ it holds that
$\|F(x)\|_H \leq \| F(0)\|_H + C \| x \|_{H_\delta} ( 1 + \| x \|_{H_\gamma}^c)$
and
$\|B(x)\|_{\HS(U,H)} \leq \| B(0)\|_{\HS(U,H)} + C \| x \|_{H_\delta}  ( 1 + \| x \|_{\HS(U, H_\gamma)}^{ c}).$
Combining this with the continuity of $X$ and
$\bar{Y}$ implies that 
$ \int_0^T 
\| A \bar Y_s \|_H 
+ \| F(Y_{\lfloor s \rfloor_{T/N}})\|_H
+ \| B(Y_{\lfloor s \rfloor_{T/N}} ) \|_{\HS(U, H)}^2
+ \| F(\bar Y_s )\|_H
+ \| B(\bar Y_s ) \|_{\HS(U, H)}^2
+ \| A X_s \|_H
+ \|F(X_s)\|_H 
+ \| B(X_s)\|_{\HS(U,H)}^2 
\, ds 
< \infty.$
We can thus apply
Theorem 2.10 in 
Hutzenthaler \& Jentzen~\cite{HutzenthalerJentzen2014PerturbationArxiv} 
to obtain that
\begin{multline}
\sup\nolimits_{t\in[0,T]}
\| X_t -  \bar Y_t   \|_{L^p(\P; H)}
\leq
e^{(C + \chi)T}
\Big\|
p
\| X - \bar Y  \|_H^{p-2}
\big [ 
\langle X - \bar Y ,  F( \bar Y  )  
-
F( Y_{\lfloor \cdot \rfloor_{T/N}} ) \rangle_H 
\\
+ 
\tfrac{(p-1)(1+\nicefrac{1}{\varepsilon})}{2}
\|
B( Y_{\lfloor \cdot \rfloor_{T/N}} )
-
B( \bar Y)
\|_{\HS(U,H)}^2
-
\chi
 \| X - \bar Y  \|_H^2
\big ]^+
\Big \|_{L^1(\mu_{[0,T]}\otimes \P ; \R)}^{\nicefrac{1}{p}}
.
\end{multline}
The Cauchy-Schwarz 
inequality therefore implies that
\begin{multline}
\label{tres}
\sup\nolimits_{t\in[0,T]}
\| X_t -  \bar Y_t  \|_{L^p(\P; H)}
\leq
e^{(C+ \chi)T}
\Big\|
p
\| X - \bar Y  \|_H^{p-2}
\big [ 
\| X - \bar Y \|_H 
\| F( \bar Y  )  
-
F( Y_{\lfloor \cdot \rfloor_{T/N}} ) \|_H 
\\
+ 
\tfrac{(p-1)(1+\nicefrac{1}{\varepsilon})}{2}
\|
B( Y_{\lfloor \cdot \rfloor_{T/N}} )
-
B( \bar Y)
\|_{\HS(U,H)}^2
-
\chi
\| X - \bar Y  \|_H^2
\big ]^+
\Big \|_{L^1(\mu_{[0,T]}\otimes \P ; \R)}^{\nicefrac{1}{p}}
.
\end{multline}
Next note that the assumption that
$\forall \, x,y\in H_\gamma \colon
\| F(x)-F(y)\|_H\leq \d \| x -y\|_{H_\delta} (1+\|x \|_{H_\gamma}^\c +\| y\|_{H_\gamma}^\c)$
and Young's inequality imply that for all $s\in[0,T]$ it holds that
\begin{equation}
\label{uno}
\begin{split}
&
\| X_s - \bar Y_s \|_H^{p-1}
\| F( \bar Y_s )- F( Y_{\lfloor s \rfloor_{T/N}} ) \|_H
\\
&
\leq
\tfrac{C(p-1)}{p} \| X_s - \bar Y_s \|_H^p
+
\tfrac{C}{p}
\| \bar Y_s - Y_{\lfloor s \rfloor_{T/N}} \|_{H_\del}^p
\big ( 1 +  \|\bar Y_s  \|_{H_\gamma}^\c + \| Y_{\lfloor s \rfloor_{T/N}}\|_{H_\gamma}^\c \big)^p
.
\end{split}
\end{equation}
Moreover, note that the assumption that
$\forall \, x,y\in H_\gamma \colon
\| B(x)-B(y)\|_{\HS(U,H)}\leq \d \| x -y\|_{H_\delta} (1+\|x \|_{H_\gamma}^\c +\| y\|_{H_\gamma}^\c)$
and again Young's inequality imply that for all $s\in[0,T]$ it holds that
\begin{equation}
\label{dos}
\begin{split}
&
\| X_s - \bar Y_s \|_H^{p-2}
\| B( Y_{\lfloor s \rfloor_{T/N}}) - B(\bar Y_s)\|_{\HS(U,H)}^2
\\
&
\leq
\tfrac{C^2(p-2)}{p}
\| X_s - \bar Y_s \|_H^p
+
\tfrac{2C^2}{p} 
\| Y_{\lfloor s \rfloor_{T/N}} - \bar Y_s \|_{H_{ \del}}^p
\big ( 1 +  \|\bar Y_s  \|_{H_{\gamma} }^{\c} + \| Y_{\lfloor s \rfloor_{T/N}}\|_{H_{ \gamma}}^{\c} \big)^p
.
\end{split}
\end{equation}
Combining \eqref{tres}--\eqref{dos} with H\"older's inequality shows that
\begin{equation}
\label{negative}
\begin{split}
&
\sup\nolimits_{t\in[0,T]}
\| X_t -  \bar Y_t  \|_{L^p(\P; H)}
\leq
e^{(C + \chi)T}
\Big\|
C
\| \bar Y - Y_{\lfloor \cdot \rfloor_{T/N}} \|_{H_{ \del}}^p
\big ( 1 +  \|\bar Y  \|_{H_\gamma}^\c + \| Y_{\lfloor \cdot \rfloor_{T/N}}\|_{H_\gamma}^\c \big)^p
\\
&
\quad
+
C^2 (p-1)(1+\nicefrac{1}{\varepsilon})
\| Y_{\lfloor \cdot \rfloor_{T/N}} - \bar Y \|_{H_{ \del}}^p
\big ( 1 +  \|\bar Y  \|_{H_{ \gamma}}^{\bar \c} + \| Y_{\lfloor \cdot \rfloor_{T/N}}\|_{H_{\gamma} }^{\c} \big)^p
\Big \|_{L^1(\mu_{[0,T]}\otimes \P ; \R)}^{\frac{1}{p}}
\\
&
=
e^{(C+\chi)T}
\big(
C + C^2(p-1)(1+\nicefrac{1}{\varepsilon})
\big)^{\nicefrac{1}{p}}
\Big\|
\| \bar Y - Y_{\lfloor \cdot \rfloor_{T/N}} \|_{H_{ \del}}
\big ( 1 +  \|\bar Y  \|_{H_\gamma}^\c + \| Y_{\lfloor \cdot \rfloor_{T/N}}\|_{H_\gamma}^\c \big)
\Big \|_{L^p(\mu_{[0,T]}\otimes \P ; \R)}
\\
&
\leq
e^{(C+\chi)T}
\big(
T
C^2 p (1+\nicefrac{1}{\varepsilon})
\big)^{\nicefrac{1}{p}}
\sup\nolimits_{s\in [0,t]}
\big\|
\bar Y_s - Y_{\lfloor s \rfloor_{T/N}} 
\big \|_{L^{2p}(\P ; H_\del)}
\\
&
\qquad
\cdot
\big(
1 
+
\sup\nolimits_{s\in [0,T]}  
\|\bar Y_s  \|_{L^{2pc}(\P ; H_\gamma)}^c
+ 
\sup\nolimits_{s\in [0,T]}
\| Y_{\lfloor s \rfloor_{T/N}} \|_{L^{2pc}(\P ; H_\gamma)}^c
\big)
.
\end{split}
\end{equation}
Next observe that
the triangle inequality implies that
\begin{equation}
\label{k11}
\begin{split}
&
\sup_{s\in [0,T]}
\|
\bar Y_s
-
Y_{\lfloor s \rfloor_{T/N}} 
\|_{L^{2p}(\P; H_\del)}
\leq
\sup_{s\in [0,T]}
\|
\bar Y_s
-
Y_s
\|_{L^{2p}(\P; H_\del)}
+
\sup_{s\in [0,T]}
\|
Y_s
-
Y_{\lfloor s \rfloor_{T/N}} 
\|_{L^{2p}(\P; H_\del)}
.
\end{split}
\end{equation}
In addition, observe that
the triangle inequality proves that for all $s\in [0,T]$ it holds that
\begin{equation}
\label{k0}
\begin{split}
&
\|
Y_s
-
Y_{\lfloor s \rfloor_{T/N}} 
\|_{L^{2p}(\P; H_\del)}
%
\leq
\big\|
( 
e^{(s-\lfloor s \rfloor_{T/N})A}
-
\operatorname{Id}_H ) Y_{\lfloor s \rfloor_{T/N}}
\big\|_{L^{2p}(\P; H_\del)}
\\
&
+
\bigg \|
\int_{\lfloor s \rfloor_{T/N}}^s
e^{(s - \lfloor u \rfloor_{T/N})A}
\,
\1_{
\big \{
\| F(Y_{\lfloor u \rfloor_{T/N}}) \|_H + \| B(Y_{\lfloor u \rfloor_{T/N}} )\|_{\HS(U,H)}\leq \left (\frac{N}{T} \right)^\theta
\big \}
}
\,
F(Y_{\lfloor u \rfloor_{T/N}})
\,
du
\bigg \|_{L^{2p}(\P; H_\del)}
\\
&
+
\bigg \|
\int_{\lfloor s \rfloor_{T/N}}^s
e^{( s - \lfloor u \rfloor_{T/N})A}
\,
\1_{
\big \{
\| F(Y_{\lfloor u \rfloor_{T/N}}) \|_H + \| B(Y_{\lfloor u \rfloor_{T/N}} )\|_{\HS(U,H)}\leq \left (\frac{N}{T} \right)^\theta
\big \}
}
\,
B(Y_{\lfloor u \rfloor_{T/N}})
\,
dW_u
\bigg \|_{L^{2p}(\P; 
H_\del)}
.
\end{split}
\end{equation}
Furthermore, observe that, e.g., 
\cite[Lemma 2.5.35]{Jentzen2014SPDElecturenotes} proves that for all $ s \in [0,T] $ it holds that
\begin{equation}
\begin{split}
\label{k1}
&
\big \|
( 
e^{(s-\lfloor s \rfloor_{T/N})A}
-
\operatorname{Id}_H ) Y_{\lfloor s \rfloor_{T/N}}
\big \|_{L^{2p}(\P; H_\del)}
\leq
\sup\nolimits_{u\in [0,T]}
\|
 Y_u \|_{L^{2p}(\P; H_\eta)}
 \left \|
e^{(s-\lfloor s \rfloor_{T/N})A}
-
\operatorname{Id}_H
 \right \|_{L( H_\eta, H_\del) }
 \\
 &
 \leq
\sup\nolimits_{u \in [0,T]}
\| Y_u \|_{L^{2p}(\P; H_\eta)}
\left ( s- \lfloor s \rfloor_{T/N} \right )^{\eta-\delta}
\leq
\sup\nolimits_{u \in [0,T]}
\| Y_u \|_{L^{2p}(\P; H_\eta)}
\left ( \tfrac{T}{N} \right )^{\eta-\delta}
.
\end{split}
\end{equation}
Moreover, note that, e.g., Theorem~2.5.34 in 
\cite{Jentzen2014SPDElecturenotes} proves that for all $s\in [0,T]$ it holds that
\begin{equation}
\begin{split}
\label{k2}
&
\bigg \|
\int_{\lfloor s \rfloor_{T/N}}^s
e^{( s - \lfloor u \rfloor_{T/N})A}
\,
\1_{
\big \{
\| F(Y_{\lfloor u \rfloor_{T/N}}) \|_H + \| B(Y_{\lfloor u \rfloor_{T/N}} )\|_{\HS(U,H)}\leq \left (\frac{N}{T} \right)^\theta
\big \}
}
\,
F(Y_{\lfloor u \rfloor_{T/N}})
\,
du
\bigg \|_{L^{2p}(\P; H_\del)}
\\
&
\leq
\int_{\lfloor s \rfloor_{T/N}}^s
\Big\|
e^{(s - \lfloor u \rfloor_{T/N})A}
F(Y_{\lfloor u \rfloor_{T/N}})
\Big \|_{L^{2p}(\P;H_\delta)}
\,
du
\\
&
\leq
\sup\nolimits_{u\in [0,T]}
\| F(Y_u)\|_{L^{2p}(\P; H)}
\int_{\lfloor s \rfloor_{T/N}}^s
(s - \lfloor u \rfloor_{T/N} )^{ - \delta}
\,
du
\leq
\sup\nolimits_{u\in [0,T]}
\| F(Y_u)\|_{L^{2p}(\P; H)}
\left( \tfrac{T}{N} \right)^{1-\delta}
.
\end{split}
\end{equation}
The Burkholder-Davis-Gundy type inequality in Lemma 7.7 
in Da Prato \& Zabczyk~\cite{dz92} and, e.g.,
Theorem~2.5.34 in \cite{Jentzen2014SPDElecturenotes}
prove that for all $ s \in [0,T] $ it holds that
\begin{equation}
\begin{split}
\label{k3}
&
\bigg \|
\int_{\lfloor s \rfloor_{T/N}}^s
e^{(s - \lfloor u \rfloor_{T/N})A}
\,
\1_{
\big \{
\| F(Y_{\lfloor u \rfloor_{T/N}}) \|_H + \| B(Y_{\lfloor u \rfloor_{T/N}} )\|_{\HS(U,H)}\leq \left (\frac{N}{T} \right)^\theta
\big \}
}
\,
B(Y_{\lfloor u \rfloor_{T/N}})
\,
dW_u
\bigg \|_{L^{2p}(\P; H_\del)}
\\
&
\leq
\sqrt{p(2p-1)
\int_{\lfloor s \rfloor_{T/N}}^s
\big \|
e^{(s - \lfloor u \rfloor_{T/N})A}
B(Y_{\lfloor u \rfloor_{T/N}})
\big \|_{L^{2p}(\P; \HS(U, H_\delta))}^2
du
}
\\
&
\leq
\sqrt{p(2p-1)}
\sup\nolimits_{u\in [0,T]}
\| B(Y_u)\|_{L^{2p}(\P; \HS(U, H) )}
\left( \tfrac{T}{N} \right)^{ \nicefrac{1}{2} -\delta}
.
\end{split}
\end{equation}
Combining \eqref{k0}--\eqref{k3} implies that
\begin{equation}
\label{M1}
\begin{split}
&
\sup\nolimits_{t\in [0,T]}
\| Y_t - Y_{\lfloor t \rfloor_{T/N}} \|_{L^{2p}( \P; H_\delta)}
\leq
\left( \tfrac{T}{N} \right)^{\nicefrac{1}{2}-\delta}
\sqrt{p(2p-1)}
\sup\nolimits_{u\in [0,T]}
\| B(Y_u)\|_{L^{2p}(\P; \HS(U, H) )}
\\
&
+
\left( \tfrac{T}{N} \right)^{\eta-\delta}
\sup\nolimits_{u\in [0,T]}
\| Y_u\|_{L^{2p}(\P; H_\eta)}
+
\left( \tfrac{T}{N} \right)^{1-\delta}
\sup\nolimits_{u\in [0,T]}
\| F(Y_u)\|_{L^{2p}(\P; H)}
.
\end{split}
\end{equation}
Furthermore, note that Lemma \ref{Lemma12} proves that
\begin{align}
\label{M2}
&
\sup\nolimits_{t\in [0,T]}
\| \bar Y_t - Y_t \|_{L^{2p}( \P; H_\delta)}
\\
\nonumber
&
\leq
\tfrac{\max \{1, T^{\nicefrac{3}{2}} \}}{1-\delta-\rho}
N^{\nicefrac{(\delta - \rho)}{2}  }
\Big [
1
+
\!
\sup_{u\in [0,T]} \|F( Y_u )  \|_{L^{\nicefrac{ 2 p}{\theta}}(\P; H)}
+
\sqrt{p(2p-1)}
\!
\sup_{u\in [0,T]} \| B( Y_u ) \|_{L^{\nicefrac{ 2 p}{\theta}}(\P; \HS(U,H))}
\Big]^{1 + \frac{1}{ 2 \theta}}
\!\!\!\!
.
\end{align}
Combining \eqref{k11}, \eqref{M1}, and \eqref{M2} shows that
\begin{align}
\label{one}
\nonumber
&
\sup\nolimits_{t\in [0,T]}
\| \bar Y_t - Y_{\lfloor t \rfloor_{T/N}} \|_{L^{2p}(\P; H_\delta) }
\leq
\left( \tfrac{T}{N} \right)^{\eta-\delta}
\sup\nolimits_{u\in [0,T]}
\| Y_u\|_{L^{2p}(\P; H_\eta)}
\\
&
+
\left( \tfrac{T}{N} \right)^{1-\delta}
\sup\nolimits_{u\in [0,T]}
\| F(Y_u)\|_{L^{2p}(\P; H)}
+
\left( \tfrac{T}{N} \right)^{\nicefrac{1}{2}-\delta}
\sqrt{p(2p-1)}
\sup\nolimits_{u\in [0,T]}
\| B(Y_u)\|_{L^{2p}(\P; \HS(U, H) )}
\\
\nonumber
&
+
\tfrac{\max \{1, T^{\nicefrac{3}{2}} \}N^{\nicefrac{(\delta - \rho)}{2} }}{1-\delta-\rho}
\Big [
1
+
\!
\sup_{u\in [0,T]} \|F( Y_u )  \|_{L^{\nicefrac{ 2 p}{\theta}}(\P; H)}
+
\sqrt{p(2p-1)}
\!
\sup_{u\in [0,T]} \| B( Y_u ) \|_{L^{\nicefrac{ 2 p}{\theta}}(\P; \HS(U,H))}
\Big]^{1 + \frac{1}{ 2 \theta}}
.
\end{align}
Combining \eqref{negative} and \eqref{one} implies that
\begin{equation}
\label{dvojica}
\begin{split}
&
\sup\nolimits_{t\in[0,T]}
\| X_t - \bar Y_t \|_{L^p(\P;H)}
\leq
e^{ \frac{T C^2 p ( 1 + 1 / \varepsilon )}{2}}
\big ( T C^2 p(1+\nicefrac{1}{\varepsilon})
\big)^{\nicefrac{1}{p}}
\\
&
\cdot
\bigg[
\left( \tfrac{T}{N} \right)^{\nicefrac{1}{2}-\delta}
\!\!
\sqrt{p(2p-1)}
\sup\nolimits_{u\in [0,T]}
\| B(Y_u)\|_{L^{2p}(\P; \HS(U, H) )}
+
\!
\left( \tfrac{T}{N} \right)^{\eta-\delta}
\!\!
\sup\nolimits_{u\in [0,T]}
\| Y_u\|_{L^{2p}(\P; H_\eta)}
\\
&
+
\tfrac{\max \{1, T^{\nicefrac{3}{2}} \}N^{  \nicefrac{(\delta -\rho)}{2} }}{1-\delta-\rho}
\Big [
1
+
\!
\sup_{u\in [0,T]} \|F( Y_u )  \|_{L^{\nicefrac{ 2 p}{\theta}}(\P; H)}
+
\sqrt{p(2p-1)}
\!
\sup_{u \in [0,T]} \| B( Y_u ) \|_{L^{\nicefrac{ 2 p}{\theta}}(\P; \HS(U,H))}
\Big]^{1 + \frac{1}{ 2 \theta}}
\\
&
+
\!
\left( \tfrac{T}{N} \right)^{1-\delta}
\sup\nolimits_{u\in [0,T]}
\| F(Y_u)\|_{L^{2p}(\P; H)}
\bigg]
\!
\left[
1 
+  
\sup\nolimits_{u\in [0,T]} \|\bar Y_u\|_{L^{2pc}(\P; H_\gamma)}^c 
+ 
\sup\nolimits_{u\in [0,T]}
\| Y_u\|_{L^{2pc}(\P; H_\gamma)}^c
\right]
.
\end{split}
\end{equation}
The proof of Lemma \ref{Lemma13} is thus completed.
\end{proof}
\subsection{Analysis of the differences between nonlinearities-stopped exponential Euler approximations and solution processes of stochastic evolution equations}
\begin{corollary}
\label{corollary.combined}
Assume the setting in Section \ref{setting3}, let 
$\eta\in [\delta,\tfrac{1}{2}),$
$\varepsilon\in (0,\infty),$
assume that $\sup_{h\in \H} | \lambda_h |\!<\infty,$
and
assume that for all 
$x,y\in H_\gamma$ 
it holds that
$
\max \{
\| F(x)- F(y) \|_H 
,
\| B(x)  - B(y) \|_{\HS(U, H)}
\}
\leq 
\d \| x - y \|_{H_\del} ( 1 + \| x \|_{H_\gamma}^\c + \| y \|_{H_\gamma}^\c )
$
and
$
\langle x-y, Ax + F(x)- Ay - F(y)\rangle_H + \tfrac{(p-1)(1+\varepsilon)}{2} \| B(x) - B(y) \|_{\HS(U,H)}^2 
\leq C \| x-y\|_H^2.
$
Then
\begin{equation}
\begin{split}
&
\sup\nolimits_{t\in [0,T]}
\| X_t - Y_t \|_{L^p(\P;H)}
\\
&
\leq
N^{ \delta - \eta}
\tfrac{
\max\{1, T^2\}
}
{(1-2\eta)}
\big (C^2(1+\nicefrac{1}{\varepsilon}) p\big)^{\nicefrac{1}{p}}
\exp\!\left(\tfrac{T C^2 p ( 1 + 1 / \varepsilon )}{2} \right)
\bigg(
3
\Big[
1+
\!
\sup\nolimits_{s\in [0,T]} \|F( Y_s )  \|_{L^{\nicefrac{ 2p}{\theta}}(\P; H)}
\\
&
+
\sqrt{p(2p-1)}
\sup\nolimits_{s\in [0,T]} \| B( Y_s ) \|_{L^{\nicefrac{ 2p}{\theta}}(\P; \HS(U,H))}
\Big ]^{1+\frac{1}{ 2 \theta}}
+
\sup\nolimits_{s\in [0,T]}
\| Y_s\|_{L^{2p}(\P; H_\eta)}
\bigg)
\\
&
\cdot
\left(
1 
+  
2
\left[
\left \| \xi \right \|_{L^{2pc}(\P; H_\gamma)}
+
C
\Big[
\tfrac{ T^{1-(\gamma+\alpha)}}{1-(\gamma+\alpha)} 
+
\sqrt{\tfrac{pc(2pc-1)}{(1-2(\gamma + \beta))}}
T^{\nicefrac{1}{2}-(\gamma + \beta)}
\Big]
\big[1 + \sup\nolimits_{s\in [0,T]} \| Y_s \|_{L^{2pca}(\P; H)}^{a} \big]
\right]^c
\right)
.
\end{split}
\end{equation}
\end{corollary}
\begin{proof}[Proof of Corollary~\ref{corollary.combined}]
Note that
\begin{equation}
\label{moka}
\sup\nolimits_{t\in [0,T]}
\| X_t - Y_t\|_{L^p(\P; H)}
\leq
\sup\nolimits_{t\in [0,T]}
\left(
\| X_t - \bar Y_t \|_{L^p(\P;H)}
+
\| Y_t - \bar Y_t \|_{L^p(\P; H)}
\right)
.
\end{equation}
Combining Lemma \ref{Lemma12}, Lemma \ref{Lemma13} (with $\rho = 2 \eta - \delta$ in the notation of Lemma \ref{Lemma13}), and 
$\eqref{moka}$ proves that 
\begin{equation}
\label{once}
\begin{split}
&
\sup\nolimits_{t\in [0,T]}
\| X_t - Y_t \|_{L^p(\P;H)}
\leq
N^{ \delta - \eta}
\tfrac{
\max\{1, T^2\}
}
{(1-2\eta)}
\big (C^2(1+\nicefrac{1}{\varepsilon}) p\big)^{\nicefrac{1}{p}}
\exp\!\left(\tfrac{T C^2 p ( 1 + 1 / \varepsilon )}{2} \right)
\\
&
\cdot
\bigg(
3
\left [
1+
\!
\sup\nolimits_{s\in [0,T]} \|F( Y_s )  \|_{L^{\nicefrac{ 2p}{\theta}}(\P; H)}
+
\sqrt{p(2p-1)}
\sup\nolimits_{s\in [0,T]} \| B( Y_s ) \|_{L^{\nicefrac{ 2p}{\theta}}(\P; \HS(U,H))}
\right ]^{1+\frac{1}{ 2 \theta}}
\\
&
\quad
+
\sup\nolimits_{s\in [0,T]}
\| Y_s\|_{L^{2p}(\P; H_\eta)}
\bigg)
\left(
1 
+  
\sup\nolimits_{s\in [0,T]} \|\bar Y_s\|_{L^{2pc}(\P; H_\gamma)}^c 
+ 
\sup\nolimits_{s\in [0,T]}
\| Y_s\|_{L^{2pc}(\P; H_\gamma)}^c
\right)
.
\end{split}
\end{equation}
Moreover, Lemma \ref{bootstrap} proves that
\begin{equation}
\label{twice}
\begin{split}
&
\sup\nolimits_{s\in [0,T]} \|\bar Y_s\|_{L^{2pc}(\P; H_\gamma)}^c 
+ 
\sup\nolimits_{s\in [0,T]}
\| Y_s\|_{L^{2pc}(\P; H_\gamma)}^c
\\
&
\leq
2
\left[
\left \| \xi \right \|_{L^{2pc}(\P; H_\gamma)}
+
C
\Big[
\tfrac{ T^{1-(\gamma+\alpha)}}{1-(\gamma+\alpha)} 
+
\sqrt{\tfrac{pc(2pc-1)}{(1-2(\gamma + \beta))}}
T^{\nicefrac{1}{2}-(\gamma + \beta)}
\Big]
\big[1 + \sup\nolimits_{s\in [0,T]} \| Y_s \|_{L^{2pca}(\P; H)}^{a} \big]
\right]^c
.
\end{split}
\end{equation}
Combining 
\eqref{once} and \eqref{twice} completes the proof
of Corollary~\ref{corollary.combined}.
\end{proof}
\section{Temporal regularity properties of solution processes of SPDEs}
\label{section.known}
In this section we present a few elementary and essentially well-known 
temporal regularity properties for solution processes of stochastic 
partial differential equations with globally Lipschitz continuous 
coefficients. In the literature similar results can, e.g., be found in
Van Neerven et al.~\cite[Theorem 6.3]{vvw08} and in the references mentioned in 
Van Neerven et al.~\cite{vvw08}.
\subsection{Setting}
\label{setting.existence}
Assume the setting in Section \ref{Main.setting},
let 
$ b \in\![0,\infty), \eta\in [0,1),$
$
F\in\mathcal{C} 
( H_\gamma, 
H_{\gamma-\eta}
),$
$
B\in\mathcal{C}
(H_{\gamma},
\HS(U, H_{\gamma-\nicefrac{\eta}{2}}))
,
$
assume that for all $x, y\in H_\gamma$ it holds that
$\max \{ \| F(x) - F(y) \|_{H_{\gamma-\eta}},
\| B(x) - B(y) \|_{\HS(U,H_{\gamma-\nicefrac{\eta}{2}})} \}
\leq b \|x-y\|_{H_\gamma},$
and
let $X\colon [0,T]\times \Omega \to H_\gamma$
be an $(\mathcal{F}_t)_{t\in [0,T]}$-predictable
stochastic process
such 
that
for all $t\in [0,T]$
it holds $\P$-a.s.\,that
$\int_0^t \|e^{(t-s)A} F(X_s)\|_H
+
\|e^{(t-s)A} B(X_s)\|_{\HS(U,H)}^2
\, ds<\infty$
and
\begin{equation}
X_t
=
\xi
+
\int_0^t
e^{(t-s)A}F(X_s)
\, ds
+ 
\int_0^t
e^{(t-s)A}  B(X_s)
\,
dW_s.
\end{equation}
\subsection{Temporal regularity properties}
\begin{lemma}
\label{holder.prvic}
Assume the setting in Section \ref{setting.existence}
and
let $\beta \in [0, \nicefrac{(1-\eta)}{2}) $.
Then 
\begin{equation}
\begin{split}
&
\sup\nolimits_{s\in [0,T), t\in (s,T]}
\frac{
\| ( X_t - e^{tA} \xi) - (X_s - e^{sA}  \xi)
\|_{L^p(\P; H_\gamma)}
}
{
(t-s)^\beta 
}
\\
&
\leq
\tfrac{\max\{1,T\} \sqrt{2p(p-1)}}{1-\eta-2\beta}
\left ( 
\!
\| F(0)\|_{H_{\gamma-\eta}}
+
\|B(0)\|_{\HS(U, H_{\gamma-\nicefrac{\eta}{2}})}
+
2
b
\sup\nolimits_{u\in [0,T]} 
\|X_u\|_{L^p(\P; H_\gamma)}
\!
\right )
.
\end{split}
\end{equation}
\end{lemma}
\begin{proof}[Proof of Lemma~\ref{holder.prvic}]
First
observe that, e.g., \cite[Theorem~2.5.34 and Lemma~2.5.35]{Jentzen2014SPDElecturenotes} 
shows that for all $s\in [0,T], t\in [s,T]$
it holds that
\begin{equation}
\label{70}
\begin{split}
&
\left \| \int_0^s (e^{(t-u)A} - e^{(s-u)A}) F(X_u)\,du \right\|_{L^p(\P; H_\gamma)}
\leq
\int_0^s \left \| (e^{(t-u)A} - e^{(s-u)A}) F(X_u) \right\|_{L^p(\P; H_\gamma)} \,du
\\
&
\leq
\int_0^s
\| (-A)^{\eta + \beta} e^{(s-u)A}\|_{L(H)}
\| (-A)^{-\beta} (\operatorname{Id}_H - e^{(t-s)A} ) \|_{L(H)}
\|F(X_u)\|_{L^p(\P; H_{\gamma-\eta})} \, du
\\
&
\leq
\int_0^s
\tfrac{(t-s)^\beta}{(s-u)^{\eta+\beta}}
\left \| \|F(0)\|_{H_{\gamma-\eta}} + b \|X_u\|_{H_\gamma} \right\|_{L^p(\P; \R)}
\, du
\\
&
\leq
\left ( 
\|F(0)\|_{H_{\gamma-\eta}}
+
b
\sup\nolimits_{u\in [0,T]} 
\|X_u\|_{L^p(\P; H_\gamma)}
\right)
\tfrac{s^{1-\eta-\beta}}{1-\eta-\beta}
(t-s)^\beta
.
\end{split}
\end{equation}
Moreover, e.g., \cite[Theorem~2.5.34 and Lemma~2.5.35]{Jentzen2014SPDElecturenotes} combined with the Burkholder-Davis-Gundy type inequality 
in Lemma~7.7 in Da Prato \& Zabczyk~\cite{dz92}
proves that for all $ s \in [0,T] $, $ t \in [s,T] $
it holds that
\begin{equation}
\begin{split}
&
\left \| \int_0^s (e^{(t-u)A} - e^{(s-u)A}) B(X_u)\,du \right\|_{L^p(\P; H_\gamma)}
\\
&
\leq
\sqrt{
\tfrac{p(p-1)}{2}
\int_0^s \left \| (e^{(t-u)A} - e^{(s-u)A}) B(X_u) \right\|_{L^p(\P; \HS(U, H_\gamma))}^2 du
}
\\
&
\leq
\sqrt{
\tfrac{p(p-1)}{2}
\int_0^s
\| (-A)^{\nicefrac{\eta}{2} + \beta} e^{(s-u)A}\|_{L(H)}^2
\| (-A)^{-\beta} (\operatorname{Id}_H - e^{(t-s)A} ) \|_{L(H)}^2
\|B(X_u)\|_{L^p(\P; \HS(U, H_{\gamma-\nicefrac{\eta}{2}}))}^2 \, du
}
\\
&
\leq
\sqrt{
\tfrac{p(p-1)}{2}
\int_0^s
\tfrac{(t-s)^{2\beta}}{(s-u)^{\eta+2\beta}}
\left\| \|B(0)\|_{\HS(U, H_{\gamma-\nicefrac{\eta}{2}}) } + b \|X_u\|_{H_\gamma} \right\|_{L^p(\P; \R)}^2
\, du
}
\\
&
\leq
\left( 
\|B(0)\|_{\HS(U, H_{\gamma-\nicefrac{\eta}{2}}) }
+
b
\sup\nolimits_{u\in [0,T]} 
\|X_u\|_{L^p(\P; H_\gamma)}
\right)
\tfrac{\sqrt{p(p-1)}s^{\nicefrac{1}{2}-\nicefrac{\eta}{2}-\beta}}{\sqrt{2(1-\eta-2\beta)}}
(t-s)^\beta
.
\end{split}
\end{equation} 
Furthermore, note that, e.g., \cite[Theorem~2.5.34 and Lemma~2.5.35]{Jentzen2014SPDElecturenotes} implies 
that for all $ s \in [0,T] $, $ t \in [s,T] $
it holds that
\begin{equation}
\begin{split}
&
\left \| \int_s^t e^{(t-u)A}F(X_u) \, du \right \|_{L^p(\P; H_\gamma)}
\leq
\int_s^t \left \| e^{(t-u)A}F(X_u) \right \|_{L^p(\P; H_\gamma)} \, du 
\\
&
\leq
\int_s^t \| (-A)^{\eta} e^{(t-u)A} \|_{L(H)} \| F(X_u) \|_{L^p(\P; H_{\gamma-\eta})} \, du
\leq
\int_s^t 
\tfrac{1}{(t-u)^{\eta}}
( \| F(0)\|_{H_{\gamma-\eta}}\! + b \|X_u \|_{L^p(\P;H_\gamma)})
\, du
\\
&
\leq
\left( 
\|F(0)\|_{H_{\gamma-\eta}}
+
b
\sup\nolimits_{u\in [0,T]} 
\|X_u\|_{L^p(\P; H_\gamma)}
\right )
\tfrac{1}{1-\eta} (t-s)^{1-\eta}
.
\end{split}
\end{equation}
Again, e.g., \cite[Theorem~2.5.34 and Lemma~2.5.35]{Jentzen2014SPDElecturenotes} 
and the Burkholder-Davis-Gundy type inequality in Lemma~7.7 
in Da Prato \& Zabczyk~\cite{dz92}
show that for all $ s \in [0,T] $, $ t \in [s,T] $
it holds that
\begin{equation}
\begin{split}
\label{73}
&
\left \| \int_s^t e^{(t-u)A}B(X_u) \, du \right \|_{L^p(\P; H_\gamma)}
\leq
\sqrt{
\tfrac{p(p-1)}{2}
\int_s^t \left \| e^{(t-u)A}B(X_u) \right \|_{L^p(\P; \HS(U, H_\gamma) )}^2 du
} 
\\
&
\leq
\sqrt{
\tfrac{p(p-1)}{2}
\int_s^t \| (-A)^{\nicefrac{\eta}{2}} e^{(t-u)A} \|_{L(H)}^2 \| B(X_u) \|_{L^p(\P; \HS(U, H_{\gamma-\nicefrac{\eta}{2}}))}^2 \, du
}
\\
&
\leq
\sqrt{
\tfrac{p(p-1)}{2}
\int_s^t (t-u)^{-\eta}
\left ( \| B(0)\|_{\HS(U, H_{\gamma-\nicefrac{\eta}{2}})} + b \|X_u \|_{L^p(\P;H_\gamma)} \right)^2
\, du
}
\\
&
\leq
\left ( 
\|B(0)\|_{\HS(U, H_{\gamma-\nicefrac{\eta}{2}})}
+
b
\sup\nolimits_{u\in [0,T]} 
\|X_u\|_{L^p(\P; H_\gamma)}
\right )
\tfrac{\sqrt{p(p-1)}}{\sqrt{2(1-\eta)}} (t-s)^{\nicefrac{(1-\eta)}{2}}
.
\end{split}
\end{equation}
Combining \eqref{70}--\eqref{73} shows that
\begin{equation}
\begin{split}
&
\sup\nolimits_{s\in [0,T), t\in (s,T]}
\frac{
\| (X_t -e^{tA} \xi )
- (X_s- e^{sA}  \xi ) 
\|_{L^p(\P; H_\gamma)}
}
{
(t-s)^\beta 
}
\\
&
\leq
\left( 
\|F(0)\|_{H_{\gamma-\eta}}
+
b
\sup\nolimits_{u\in [0,T]} 
\|X_u\|_{L^p(\P; H_\gamma)}
\right )
\tfrac{2\max \{ 1, T\}}{1-\eta-\beta} 
\\
&
\quad
+
\left ( 
\!
\|B(0)\|_{\HS(U, H_{\gamma-\nicefrac{\eta}{2}})}
+
b
\sup\nolimits_{u\in [0,T]} 
\|X_u\|_{L^p(\P; H_\gamma)}
\!
\right )
\tfrac{\sqrt{2p(p-1)\max\{1,T\}}}{\sqrt{1-\eta-2\beta}} 
.
\end{split}
\end{equation}
The proof of Lemma \ref{holder.prvic} is thus completed.
\end{proof}
\begin{corollary}
\label{modification}
Assume the setting in Section \ref{setting.existence},
assume that $\sup_{t\in [0,T]} \| X_t \|_{L^p(\P;H_\gamma)}<\infty,$
and
assume that $\tfrac{1}{p}<\tfrac{(1-\eta)}{2}.$
Then there exist 
a stochastic process $Y\colon [0,T]\times \Omega \to H_\gamma$
with continuous sample paths
such that for all $t\in[0,T]$ it holds $\P$-a.s.\,that $X_t=Y_t.$
\end{corollary}
\begin{proof}[Proof of Corollary~\ref{modification}]
Note that Lemma \ref{holder.prvic} combined
with the Kolmogorov-Chentsov theorem proves that
there exists a modification with continuous sample paths of the stochastic process 
$[0,T] \times \Omega \ni (t, \omega) \mapsto X_t(\omega)-e^{t A}\xi(\omega) \in H_\gamma.$
In addition, observe that the fact that $A$ is the generator of a strongly continuous semigroup
implies that the 
stochastic process $[0, T] \times \Omega \ni (t, \omega) \mapsto e^{tA}\xi(\omega) \in H_\gamma$
has continuous sample paths.
The proof of Corollary \ref{modification}
is thus completed.
\end{proof}
\section{Convergence of spatial spectral Galerkin 
discretizations}
In this section we establish uniform convergence in probability of 
spatial spectral Galerkin approximations in the case of SEEs with 
semi-globally Lipschitz continuous coefficients (cf., e.g., 
Kurniawan~\cite{MasterRyan}); see Proposition \ref{convergence_probability} below. Proposition \ref{convergence_probability} 
(and its consequence in Corollary \ref{galerkin.convergence} respectively) is a tool used in the proof 
of our main result in Theorem \ref{Finale} below (see Proposition \ref{last} below). In 
our proof of Proposition \ref{convergence_probability} we employ Corollary 2.9 in Cox et al.~\cite{CoxHutzenthalerJentzen2015MonteCarloArxiv} 
(which is a generalization of Lemma~A1 in Bally et al.~\cite{BallyMilletSanz1995}) and a 
nowadays well-known localization procedure (see, e.g., Gy\"{o}ngy~\cite{g98b} 
and Printems \cite[Lemma 4.8]{p01}). 
There are a number of quite similar results in the literature (see, 
e.g., Cox et al.~\cite[Corollary 3.3]{CoxHutzenthalerJentzen2015MonteCarloArxiv},
Gy\"{o}ngy~\cite{g98b},
Kurniawan~\cite[Lemma 4.2.2]{MasterRyan}, 
Printems~\cite[Lemma 4.8]{p01}) 
and Proposition \ref{convergence_probability} is a minor extension of the results in the 
literature. The main difference between Proposition \ref{convergence_probability} and known 
results in the literature is that Proposition \ref{convergence_probability} does only prove 
convergence in probability with no rate of convergence but Proposition 
\ref{convergence_probability} does not assume any growth condition of the eigenvalues of the 
dominant linear operator appearing in the considered SEE; see \eqref{see68} 
below. In particular, Proposition \ref{convergence_probability} also applies to SEEs in which the 
dominant linear operator $ A $ in \eqref{see68} is a bounded linear operator.
\label{section.galerkin}
\subsection{Setting}
\label{setting.probability}
Assume the setting in Section \ref{Main.setting},
let 
$b \in [0, \infty), \eta\in [0, 1),$
$
F\in \mathcal{C} 
( H_\gamma, 
H_{\gamma-\eta}
)
,
B\in \mathcal{C}
(H_{\gamma},
\HS(U, H_{\gamma-\nicefrac{\eta}{2}}) )
,
$
let $I_n \in \mathcal{P}(\H), n\in \N_0,$ satisfy 
$\cup_{n\in \N} \left( \cap_{m\in \{n+1, n+2,\ldots \}} I_m\right) 
=
\H
=
I_0 ,$
let $P_I \in L( H_{-1} ), I \in \mathcal{P}(\H),$ be the linear
operators with the property that for all $x\in H, I \in \mathcal{P}(\H)$ it holds
that
$P_I x = \sum_{h\in I} \langle h, x \rangle_H h,$
and
let $X^{n} \colon [0,T]\times \Omega \to H_\gamma,
n\in \N_0,$
be $(\mathcal{F}_t)_{t\in [0,T]}$-adapted
stochastic processes with continuous sample paths
such that
for all $t\in [0,T], n\in \N_0$ it holds $\P$-a.s.\,that
\begin{equation}
\label{see68}
X_t^{n}
=
e^{(t-s)A}
P_{I_n}(\xi)
+
\int_0^t
e^{(t-s)A}P_{I_n} F(X^{n}_s)
\, ds
+ 
\int_0^t
e^{(t-s)A} P_{I_n} B(X^{n}_s)
\,
dW_s.
\end{equation}
\subsection{Convergence in the case of globally Lipschitz continuous coefficients}
\begin{corollary}
\label{holder}
Assume the setting in Section \ref{setting.probability},
assume that $\xi \in L^p(\P; H_\gamma),$
assume that for all $x, y\in H_\gamma$ it holds that
$\max \{ \| F(x) - F(y) \|_{H_{\gamma-\eta}},
\| B(x) - B(y) \|_{\HS(U,H_{\gamma-\nicefrac{\eta}{2}})} \}
\leq b \|x-y\|_{H_\gamma},$
let $\beta \in [0, \nicefrac{(1-\eta)}{2}),$
and
assume that for all $n\in \N$ it holds that
$\sup_{t\in [0,T]}\| X_t^{n}\|_{L^p(\P; H_{\gamma})} <\infty.$
Then 
\begin{equation}
\label{lak}
\sup_{n\in \N}
\left(\!
\sup_{s\in [0,T), t\in (s,T]}
\frac{
\| ( X_t^{n} - e^{tA} P_{I_n} \xi ) - (X_s^{n} - e^{sA} P_{I_n} \xi)
\|_{L^p(\P; H_\gamma)}
}
{
(t-s)^\beta 
}
\right)
<\infty
.
\end{equation}
\end{corollary}
\begin{proof}[Proof of Corollary \ref{holder}]
Note that, e.g., \cite[Corollary 6.1.8]{Jentzen2014SPDElecturenotes}
shows that
$\sup_{n\in \N}\sup_{t\in [0,T]} \| X_t^{n}\|_{L^p(\P; H_{\gamma})}<\infty.$
Lemma \ref{holder.prvic} hence proves \eqref{lak}. The proof of Corollary \ref{holder}
is thus completed.
\end{proof}
\begin{lemma}
\label{convergence_probability_lipschitz}
Assume the setting in Section \ref{setting.probability},
assume that for all $x, y\in H_\gamma$ it holds that
$\max \{ \| F(x) - F(y) \|_{H_{\gamma-\eta}},
\| B(x) - B(y) \|_{\HS(U,H_{\gamma-\nicefrac{\eta}{2}})} \}
\leq b \|x-y\|_{H_\gamma},$
and
assume that for all $n\in \N_0$ it holds that 
$ \sup_{t\in [0,T]} \| X_t^{n} \|_{L^p(\P; H_\gamma)} < \infty.$
Then 
\begin{equation}
\label{dvojnica}
\lim\nolimits_{ n\to \infty} 
\left(
\sup\nolimits_{t\in [0,T]} \| ( X^0_t - e^{tA} \xi ) - (X_t^{n} - e^{tA}P_{I_n} \xi) \|_{L^p(\P; H_\gamma)}
\right)
=0
\end{equation}
and
\begin{equation}
\label{srna3}
\lim\nolimits_{n\to \infty}
\left(
\sup\nolimits_{t\in [0,T]} \| X^0_t - X_t^{n} \|_{L^p(\P; H_\gamma)}
\right)
=0
.
\end{equation}
\end{lemma}
\begin{proof}[Proof of Lemma~\ref{convergence_probability_lipschitz}]
Let $\mathcal{E}_r\colon [0,\infty) \to [0,\infty), r\in (0,\infty),$
be the functions with the property that
for all $x\in [0,\infty), r\in (0,\infty)$
it holds that
$\mathcal{E}_r(x)= 
\big(
\sum_{n=0}^\infty \tfrac{(x^2 \Gamma(r))^n}{\Gamma(nr+1)}
\big)^{\nicefrac{1}{2}}
$ 
(cf., e.g., Henry~\cite[Section 7.1]{h81} and \cite[Definition 3.3.1]{Jentzen2014SPDElecturenotes}).
Observe that, e.g., \cite[Lemma 6.1.7]{Jentzen2014SPDElecturenotes} proves that
for all $ n \in \N $ it holds that
\begin{equation}
\begin{split}
\label{srna21}
\!\!\!\!
&
\sup\nolimits_{t\in [0,T]}
\| X_t^0 - X_t^{n} \|_{L^p(\P; H_\gamma)}
\!
\leq
\!
\sqrt{2}
\sup\nolimits_{t\in [0,T]}
\| P_{\H \setminus I_n} X_t^0 
\|_{L^p(\P; H_\gamma)}
\mathcal{E}_{1-\eta}
\Big(
\tfrac{ \sqrt{2} T^{1-\eta} b}{\sqrt{1-\eta}}
+
\!
\sqrt{T^{1-\eta}p(p-1)}b
\Big)
.
\end{split}
\end{equation}
Let $ J_n \subseteq \H $, $ n \in \N,$ 
be the sets with the property 
that for all $ n \in \N $ it holds that $J_n = \cap_{m\in \{n+1, n+2,\ldots \} } I_m.$
Next, let 
$f_n \colon [0,T] \to [0,\infty), n\in \N,$ 
be the functions
with the property that for all $t\in [0,T], n \in \N$ 
it holds that
$f_n(t)=\| P_{\H\setminus J_n} (X_t^0 - e^{tA}\xi) \|_{L^p(\P; H_\gamma)}.$
Corollary \ref{holder} proves that
the functions $f_n, n\in \N,$ 
are continuous.
Moreover, note that the sequence $(f_n)_{n\in \N}$ is non-increasing. Furthermore, observe that Lebesgue's dominated convergence theorem shows that for all $ t \in [0,T] $ it holds that $\lim_{ n \to \infty } f_n( t ) = 0.$ We can thus apply Dini's theorem to obtain that $\lim_{ n \to \infty } \sup_{ t \in [0,T] } f_n(t) = 0,$ i.e., that
\begin{equation}
\begin{split}
\lim\nolimits_{ n \to \infty} 
\left(
\sup\nolimits_{t\in [0,T]}
\| P_{\H \setminus J_n} (X_t^0 - e^{tA} \xi) \|_{L^p(\P; H_\gamma)}
\right)
=
0.
\end{split}
\end{equation}
This proves that
\begin{equation}
\label{srna1}
\begin{split}
\lim\nolimits_{ n \to \infty} 
\left(
\sup\nolimits_{t\in [0,T]}
\| P_{\H \setminus I_n} (X_t^0- e^{tA}\xi) \|_{L^p(\P; H_\gamma)}
\right)
=
0.
\end{split}
\end{equation}
Moreover,
Lebesgue's theorem of dominated convergence proves that 
$\lim_{n \to \infty} \| P_{ \H \setminus I_n} \xi \|_{L^p(\P; H_\gamma)} = 0.$
This and the fact that for all $n \in \N$ it holds that
$\sup_{t\in [0,T]} \| P_{ \H \setminus I_n} e^{tA} \xi \|_{L^p(\P; H_\gamma)}
\leq  \| P_{ \H \setminus I_n} \xi \|_{L^p(\P; H_\gamma)}$
imply that
\begin{equation}
\label{srna2}
\lim\nolimits_{ n\to \infty} 
\left(
\sup\nolimits_{t\in [0,T]} 
\| P_{ \H \setminus I_n} e^{tA} \xi \|_{L^p(\P; H_\gamma)} 
\right)
= 0
.
\end{equation}
Combining \eqref{srna21}, \eqref{srna1}, \eqref{srna2}, and the triangle inequality
proves \eqref{srna3}.
Finally, observe that \eqref{srna3}, \eqref{srna2}, and the triangle
inequality prove \eqref{dvojnica}.
The proof of Lemma \ref{convergence_probability_lipschitz}
is thus completed.
\end{proof}
\begin{corollary}
\label{sup.inside}
Assume the setting in Section \ref{setting.probability},
assume that $p(1-\eta)>2,$
assume that for all $n\in \N_0 $ it holds that 
$\sup_{t\in [0,T]}\| X_t^{n}\|_{L^p(\P; H_\gamma)} 
 <\infty,$
and
assume that for all $x, y\in H_\gamma$ it holds that
$\max \{ \| F(x) - F(y) \|_{H_{\gamma-\eta}},
\| B(x) - B(y) \|_{\HS(U,H_{\gamma-\nicefrac{\eta}{2}})} \}
\leq b \|x-y\|_{H_\gamma} $.
Then 
\begin{equation}
\label{result}
\lim\nolimits_{ n\to \infty} 
\left(
\E \! \left [\sup\nolimits_{t\in [0,T]} \|  X^0_t -  X_t^{n} \|_{H_\gamma}^p \right ]
\right)
=0.
\end{equation}
\end{corollary}
\begin{proof}[Proof of Corollary~\ref{sup.inside}]
First of all, observe that Corollary \ref{holder}, Lemma \ref{convergence_probability_lipschitz}, and
Corollary 2.9 in 
Cox et al.~\cite{CoxHutzenthalerJentzen2015MonteCarloArxiv} (cf.\ also \cite[Lemma~A1]{BallyMilletSanz1995})
prove that
\begin{equation}
\label{prvi.del}
\lim\nolimits_{ n \to \infty} 
\left(
\E
\!
\left[
\sup\nolimits_{t\in [0,T]}
\|  ( X_t^0 - e^{tA}  \xi ) - ( X_t^{n} - e^{tA} P_{I_n}  \xi) \|_{H_\gamma}^p
\right]
\right)
=
0
.
\end{equation}
Next note that Fatou's lemma shows that
\begin{equation}
\label{drugi.del}
\begin{split}
&
\limsup\nolimits_{ n \to \infty} 
\left(
\E
\!
\left[
\sup\nolimits_{t\in [0,T]}
\|e^{tA}( P_\H -P_{I_n})\xi \|_{H_\gamma}^p
\right]
\right)
\leq
\limsup\nolimits_{ n \to \infty} 
\left(
\E
\!
\left[
\|(P_\H-P_{I_n} )\xi \|_{H_\gamma}^p
\right]
\right)
\\
&
\leq
\E
\!
\left[
\limsup\nolimits_{ n \to \infty} 
\|(P_\H -P_{I_n})\xi \|_{H_\gamma}^p
\right]
=
0.
\end{split}
\end{equation}
Combining \eqref{prvi.del} and \eqref{drugi.del}
with the triangle inequality proves \eqref{result}.
The proof of Corollary \ref{sup.inside} is thus completed.
\end{proof}
\subsection{Convergence in the case of semi-globally Lipschitz 
continuous coefficients}
\begin{proposition}
\label{convergence_probability}
Assume the setting in Section \ref{setting.probability}
and assume that for every $R\in [0,\infty)$ 
there exists a real number $K\in [0,\infty)$ such that for all $x, y \in H_\gamma$ with $\max \{ \| x\|_{H_\gamma}, \| y \|_{H_\gamma} \}\leq R$
it holds that 
$\max\{ \|F(x) - F(y)\|_{H_{\gamma-\eta}}, \|B(x)- B(y) \|_{\HS(U,H_{\gamma-\nicefrac{\eta}{2})}}\}
\leq K \|x-y\|_{H_\gamma}.$
Then for all $\varepsilon\in (0,\infty)$ it holds that
\begin{equation}
\lim\nolimits_{ n \to \infty} 
\left(
\P\!\left[ \sup\nolimits_{t\in [0,T]} \| X_t^0 - X_t^{n}\|_{H_\gamma}\geq \varepsilon \right] 
\right)
= 0.
\end{equation}
\end{proposition}
\begin{proof}[Proof of Proposition~\ref{convergence_probability}]
Throughout this proof let $ q \in ( \nicefrac{ 2 }{ (1 - \eta )
} , \infty) $ be a real number and let 
$\phi_R\colon \R \to [0,1], R\in (0,\infty),$
be infinitely often differentiable functions such that for all $x\in [-R,R]$ it holds that $\phi_R(x)=1$
and such that for all  $x\in (-\infty, -R-1]\cup [R+1,\infty)$ it holds that $\phi_R(x)=0.$
Moreover, let
$F_R\colon H_\gamma \to H_{\gamma-\eta}, R\in (0,\infty),$
and
$B_R\colon H_\gamma \to \HS(U,H_{\gamma- \nicefrac{\eta}{2}}), R\in (0,\infty),$
be the functions 
with the property
that for all $x\in H_\gamma, R\in (0,\infty)$ it holds that
$F_R(x)=F(x)\phi_R(\|x \|_{H_\gamma})$
and
$B_R(x)=B(x) \phi_R(\| x \|_{H_\gamma}).$
In the next step we observe that, 
e.g., Theorem~5.1 in \cite{JentzenKloeden2011}
and, e.g., Corollary~\ref{modification} (see also, e.g., Van Neerven et al.~\cite[Theorem~6.2]{vvw08}) prove that there
exist
up to modification unique
$(\mathcal{F}_t)_{t\in [0,T]}$-adapted
stochastic processes
$X^{n, R} \colon [0,T]\times \Omega \to H_\gamma, R\in (0,\infty), n\in \N_0,$
with continuous sample paths
such that for all $R\in (0,\infty), n\in \N_0$ it holds
$\sup_{t\in [0,T]} \|X_t^{n,R}\|_{L^q(\P; H_\gamma)}<\infty$
and
such that for all
$t\in [0,T], R\in (0,\infty), n\in \N_0$ it holds $\P$-a.s.\,that
\begin{equation}
X^{n,R}_t
=
e^{tA}
\1_{\{ \| P_{I_n} \xi \|_{H_\gamma} < R\} }
P_{I_n}
(\xi)
+ 
\int_0^t e^{(t-s)A} P_{I_n} F_R(X_s^{n,R})\, ds
+
\int_0^t
e^{(t-s)A}
P_{I_n}
B_R(X_s^{n, R})
\, dW_s.
\end{equation}
Furthermore, note that for all $x\in H_\gamma, R\in [0,\infty)$ 
with $\| x \|_{H_\gamma}\leq R$
it holds that $F_R(x)=F(x)$ and $B_R(x)=B(x).$
Next, let $\tau^{n,R}\colon \Omega \to [0,T], n\in \N_0, R\in (0,\infty),$ and
$\rho^{n,R}\colon \Omega \to [0,T], n\in \N, R\in (0,\infty),$
be the stopping times with the property that
for all $n\in \N_0, R\in (0,\infty)$ it holds that
\begin{equation}
\tau^{n,R}
=
\inf \!\left( \{t\in [0,T]\colon \|X_t^{n}\|_{H_\gamma}\geq R \} \cup \{T\} \right)
\end{equation}
and
\begin{equation}
\rho^{n, R}
=
\inf \!\left( \{ t\in [0,T]\colon \|X_t^0 - X_t^{n}\|_{H_\gamma} \geq R \} \cup \{T\} \right).
\end{equation}
Observe that, e.g.,
Lemma~4.2.2 
in Kurniawan~\cite{MasterRyan} and Markov's inequality 
prove that
for all $n\in \N, R\in (0,\infty), q\in (\nicefrac{2}{(1-\eta)}, \infty), \varepsilon\in (0,1)$
it holds that
\begin{equation}
\begin{split}
\label{gran0}
&
\P \!\left[ \sup\nolimits_{t\in [0,T]} \| X_t^0 - X_t^{n} \|_{H_\gamma} \geq \varepsilon \right ]
-
\P\!\left[ \sup\nolimits_{t\in [0,T]} \|X_t^0 \|_{H_\gamma} \geq R \right]
\\
&
\leq
\P \!\left[ \!\left\{ \sup\nolimits_{t\in [0,T]}\| X_t^0 -X_t^{n}\|_{H_\gamma} \geq \varepsilon \right\}
\cap
\{ \sup\nolimits_{t\in [0,T]} \|X_t^0 \|_{H_\gamma} < R \} \right]
\\
&
=
\P \!\left[ \!\left\{ \sup\nolimits_{t\in [0, \rho^{n, \varepsilon} ]}\| X_t^0 -X_t^{n} \|_{H_\gamma} \geq \varepsilon \right\}
\cap
\!\left\{ \sup\nolimits_{t\in [0,T]} \|X_t^0 \|_{H_\gamma} < R \right \} \right]
\\
&
=
\P \!\left[ \!\left\{ \sup\nolimits_{t\in [0, \rho^{n, \varepsilon} ]}
\1_{
\{
\tau^{0, R}>t, \rho^{n, \varepsilon} \geq t 
\}
}
\| X_t^0 -X_t^{n}\|_{H_\gamma} \geq \varepsilon \right\}
\cap
\!\left\{ \sup\nolimits_{t\in [0,T]} \|X_t^0  \|_{H_\gamma} < R \right \} \right]
\\
&
=
\P \!\left[ \!\left\{ \sup\nolimits_{t\in [0, \rho^{n, \varepsilon} ]}
\1_{
\{
\tau^{0, R}>t, \rho^{n, \varepsilon} \geq t, \tau^{n, R+1}>t
\}
}
\| X_t^0 -X_t^{n} \|_{H_\gamma} \geq \varepsilon \right\}
\cap
\!\left\{ \sup\nolimits_{t\in [0,T]} \|X_t^0 \|_{H_\gamma} < R \right \} \right]
\\
&
\leq
\P \!\left[  \sup\nolimits_{t\in [0, \rho^{n, \varepsilon} ]}
\1_{
\{
\tau^{0, R}>t, \rho^{n, \varepsilon} \geq t, \tau^{n, R+1}>t 
\}
}
\| X_t^0 -X_t^{n} \|_{H_\gamma} \geq \varepsilon 
 \right]
\\
&
=
\P \!\left[ \sup\nolimits_{t\in [0, \rho^{n, \varepsilon} ]}
\1_{
\{
\tau^{0, R}>t, \rho^{n, \varepsilon} \geq t, \tau^{n, R+1}>t 
\}
}
\big\| X_{\min\{t, \tau^{0, R+1} \} }^0 -X_{\min\{t, \tau^{n, R+1} \} }^{n} \big\|_{H_\gamma} \geq \varepsilon 
 \right]
\\
&
=
\P \!\left[ \! \sup_{t\in [0, \rho^{n, \varepsilon} ]}
\!\!
\1_{
\{
\tau^{0, R}>t, \rho^{n, \varepsilon} \geq t, \tau^{n, R+1}>t 
\}
}
\big \| 
\1_{ \{ \| \xi \|_{H_\gamma}<R+1 \}}  X_{\min\{t, \tau^{0, R+1} \} }^0
-
\1_{ \{ \|P_{I_n} \xi \|_{H_\gamma}<R+1 \}}
X_{\min\{t, \tau^{n, R+1} \} }^{n} \big\|_{H_\gamma} \geq \varepsilon 
 \right]
\\
&
=
\P \!\left[\!  \sup_{t\in [0, \rho^{n, \varepsilon} ]}
\!\!
\1_{
\{
\tau^{0, R}>t, \rho^{n, \varepsilon} \geq t, \tau^{n, R+1}>t 
\}
}
\big\| 
\1_{ \{ \|\xi \|_{H_\gamma}<R+1 \}}  X_{\min\{t, \tau^{0, R+1} \} }^{0, R+1} 
-
\1_{ \{ \|P_{I_n} \xi \|_{H_\gamma}<R+1 \}}
X_{\min\{t, \tau^{n, R+1} \} }^{n, R+1} \big\|_{H_\gamma} \geq \varepsilon 
\right]
\\
&
=
\P \!\left[  \sup\nolimits_{t\in [0, \rho^{n, \varepsilon} ]}
\1_{
\{
\tau^{0, R}>t, \rho^{n, \varepsilon} \geq t, \tau^{n, R+1}>t 
\}
}
\big\| 
X_{t }^{0, R+1} 
-
X_{t}^{n, R+1} \big\|_{H_\gamma} \geq \varepsilon 
 \right]
\\
&
\leq
\P \!\left[  \sup\nolimits_{t\in [0, T ]}
\big\| 
X_{t }^{0, R+1} 
-
X_{t}^{n, R+1} \big\|_{H_\gamma} \geq \varepsilon 
 \right]
%
\leq
\varepsilon^{-q}
\cdot
\E \!\left [ \sup\nolimits_{t\in [0,T]} \big\|X_t^{0, R+1} - X_t^{n, R+1} \big\|_{H_\gamma}^q \right]
.
\end{split}
\end{equation}
Corollary \ref{sup.inside} therefore proves that for all $R\in (0,\infty), \varepsilon\in (0,1)$ it holds that
\begin{equation}
\label{toinfty}
\lim_{n\to \infty}
\P \!\left[ \sup\nolimits_{t\in [0,T]} \| X_t^0 - X_t^{n} \|_{H_\gamma} \geq \varepsilon \right ]
=
\P\!\left[ \sup\nolimits_{t\in [0,T]} \|X_t^0 \|_{H_\gamma} \geq R \right]
.
\end{equation}
In the next step we let $ R \in (0,\infty) $ in \eqref{toinfty} tend to $ \infty $ to obtain that for all $ \varepsilon \in (0,\infty) $ it holds that
$
\lim\nolimits_{ n \to \infty } 
\P\!\left[ \sup\nolimits_{t\in [0,T]} \| X_t^0 - X_t^{n} \|_{H_\gamma} \geq \varepsilon \right ]
=
0
.
$
The proof of Proposition \ref{convergence_probability} is thus completed.
\end{proof}
\begin{corollary}
\label{galerkin.convergence}
Assume the setting in Section \ref{setting.probability},
let $q\in (0,p),$
assume that for every $R\in [0,\infty)$ 
there exists a real number $K\in [0,\infty)$ such that for all $x, y \in H_\gamma$ with $\max \{ \| x\|_{H_\gamma}, \| y \|_{H_\gamma} \}\leq R$
it holds that 
$\max\{ \|F(x) - F(y)\|_{H_{\gamma-\eta}}, \|B(x)- B(y) \|_{\HS(U,H_{\gamma-\nicefrac{\eta}{2}})}\}
\leq K \|x-y\|_{H_\gamma},$
and
assume that 
$
\limsup_{n \to \infty} 
\left(
\sup_{t\in [0,T]} \|X_t^{n} \|_{L^p(\P; H_\gamma)}
\right) <\infty.$
Then  $\sup\nolimits_{t\in [0,T]} \|X_t^0 \|_{L^p(\P; H_\gamma)}<\infty$ and
\begin{equation}
\label{hitra}
\begin{split}
\lim\nolimits_{n \to \infty} 
\left(
\sup\nolimits_{t\in [0,T]}
\| X_t^0 - X_t^{n}\|_{L^q(\P; H_\gamma)}
\right)
=
0
.
\end{split}
\end{equation}
\end{corollary}
\begin{proof}[Proof of Corollary~\ref{galerkin.convergence}]
Observe that Proposition \ref{convergence_probability}
combined with, e.g., 
Lemma 4.10 in Kurniawan~\cite{MasterRyan} 
(see also, e.g., \cite[Section 3.4.1]{HutzenthalerJentzen2014Memoires}) 
proves 
$
  \sup\nolimits_{ t \in [0,T] } 
  \| X_t^0 \|_{ L^p( \P; H_\gamma) } < \infty
$
and
\eqref{hitra}.
The proof of Corollary~\ref{galerkin.convergence} is thus competed.
\end{proof}
\section{Strong convergence rates for full discrete 
nonlinearities-stopped approximation schemes}
In this section we use the results established in Sections \ref{section2}, \ref{section3}, \ref{section4}, and
\ref{section.galerkin} as well as consequences of the perturbation estimate 
in Theorem~2.10 in Hutzenthaler \& Jentzen~\cite{HutzenthalerJentzen2014PerturbationArxiv} to
prove Theorem \ref{Finale} (the main result of this article).
\label{section5}
\subsection{Setting}
\label{setting6}
Assume the setting in Section \ref{Main.setting},
let
$
F\in \mathcal{C} 
( H_\gamma, 
H
)
,
B\in \mathcal{C}
(H_{\gamma},
\HS(U, H) )
,
$
$\varepsilon \in (0,\infty),$
assume that 
$\gamma < \min \{ 1-\alpha, \nicefrac{1}{2}-\beta\},$
let $X \colon [0,T]\times \Omega \to H_\gamma$ be 
an $(\mathcal{F}_t)_{t\in [0,T]}$-adapted stochastic process
with continuous sample paths
such that for all $t\in [0, T]$ it holds $\P$-a.s.\,that
\begin{equation}
X_t = e^{tA} \xi + \int_0^t e^{(t-s)A}F(X_s) \, ds
+
\int_0^t e^{(t-s)A} B( X_s ) \, dW_s,
\end{equation}
let $(P_I)_{I \in \mathcal{P}(\H) }\subseteq L(H)$ 
be the linear operators with the property that for all
$x\in H, I\in \mathcal{P}(\H)$ it holds that
$P_I(x) =\sum_{h \in I} \langle h, x \rangle_H h,$
let 
$Y^{N,I,R}\colon [0, T] \times \Omega \to P_I(H), N\in \N, I\in \mathcal{P}_0(\H), R\in L(U),$ 
be 
$(\mathcal{F}_t)_{t\in [0,T]}$-adapted 
stochastic processes 
such that for all 
$t\in [0,T], N\in \N, I \in \mathcal{P}_0(\H), R\in L(U)$ 
it holds $\P$-a.s.\,that
\begin{equation}
\begin{split}
Y^{N,I,R}_t 
&
=
e^{tA} P_I \xi
+
\int_0^t
e^{(t- \lfloor s \rfloor_{T/N})A}
\,
\1_{ \big \{ 
 \| 
P_I F  (Y^{N,I,R}_{\lfloor s \rfloor_{T/N}}  )
 \|_H 
+ 
 \| P_I B( Y^{N,I,R}_{\lfloor s \rfloor_{T/N}} )  \|_{HS(U, H)} 
\,
\leq 
\left ( \frac{N}{T}  \right)^\theta \big \}}
P_I F( Y^{N,I,R}_{\lfloor s \rfloor_{T/N}} ) 
\,
ds
\\
&
\quad
+
\int_0^t e^{(t- \lfloor s \rfloor_{T/N})A}
\,
\1_{ \big\{ \| P_I F(Y^{N,I,R}_{\lfloor s \rfloor_{T/N}} )\|_H 
+ 
 \| P_I B( Y^{N,I,R}_{\lfloor s \rfloor_{T/N}}  )  \|_{HS(U, H)} 
\leq 
\left ( \frac{N}{T} \right )^\theta \big \}}
\,
 P_I B( Y^{N,I,R}_{\lfloor s \rfloor_{T/N}}  )R \,dW_s
 ,
\end{split}
\end{equation}
and
assume that for all $x, y \in H_\gamma$ it holds that
$
\max \{ \left \| F(x) - F(y) \right \|_H, \left \| B(x) - B(y) \right \|_{\HS(U, H)} \}
\leq 
\d
\left \| x - y \right \|_{H_\delta}
( 1 + \left \| x \right \|_{H_\gamma}^\c + \left \| y \right \|_{H_\gamma}^\c )
$
and
$\max \{ \|F(x)\|_{H_{-\alpha}},
\|B(x)\|_{\HS(U, H_{-\beta})} \}\leq C ( 1 + \|x\|_H^a).$
\subsection{Strong convergence rates for space discretizations}
\begin{lemma}
\label{lem:spatial_rate}
Assume the setting in Section \ref{setting6},
let $I_1, I_2 \in \mathcal{P}_0(\H),  q, r \in [1,\infty]$ satisfy $I_1\subseteq I_2$ and $\tfrac{1}{q}+\tfrac{1}{r}=1,$
assume that for all $ x, y \in H_1 $ it holds that
$\langle x-y, Ax - Ay + F(x) - F(y) \rangle_H + \tfrac{(p-1)(1+\varepsilon)}{2} \| B(x) - B(y)\|_{\HS(U, H)}^2  \leq C \| x- y \|_H^2,$
and let
$X^{I_k} \colon [0,T]\times \Omega \to P_{I_k}(H_\gamma), k\in \{1,2\},$ be $(\mathcal{F}_t)_{t\in [0,T]}$-adapted
stochastic processes with continuous sample paths such that
for all $s\in [0,T], k\in \{1,2\}$ it holds $\P$-a.s.\,that
\begin{equation}
X^{I_k}_s = P_{I_k}(\xi) + \int_0^s [ A X_u^{I_k} +  P_{I_k} F( X^{I_k}_u) ] \, du
+
\int_0^s P_{I_k} B ( X^{I_k}_u ) \, dW_u
.
\end{equation}
Then 
\begin{equation}
\begin{split}
&
\sup\nolimits_{t\in [0,T]}
\|X_t^{I_1}-X_t^{I_2}\|_{L^p(\P; H)}
\\
&
\leq
\big(
\|(-A)^{-\delta} \|_{L(H)}
+
e^{
(C+1)T
}
C
p(1+\tfrac{1}{\varepsilon}) 
\big)
\sup_{u\in [0,T]} \| P_{\H \setminus I_1} X_u^{I_2} \|_{L^{qp} (\P; H_\delta)}
\bigg[
1
+
2
\sup_{u\in [0,T]}
\| X_u^{I_2} \|_{L^{rpc}(\P; H_\gamma)}^c
\bigg]
.
\end{split}
\end{equation}
\end{lemma}
\begin{proof}[Proof of Lemma~\ref{lem:spatial_rate}]
We intend to prove Lemma~\ref{lem:spatial_rate} through an application of Proposition~3.6 in 
Hutzenthaler \& Jentzen~\cite{HutzenthalerJentzen2014PerturbationArxiv}.
For this we now check the assumptions in Proposition~3.6 in 
Hutzenthaler \& Jentzen~\cite{HutzenthalerJentzen2014PerturbationArxiv}.
Note that for all $x\in H_\gamma, y\in P_{I_1}(H)$ it holds that
\begin{equation}
\begin{split}
\label{start}
&
\langle P_{I_1} x - y, P_{I_1} [AP_{I_1} x +  F(P_{I_1} x)] - P_{I_1} [ Ay + F(y)]
\rangle_H
+
\tfrac{(p-1)(1+\varepsilon)}{2} \| P_{I_1}[B(P_{I_1}x) - B(y)] \|_{\HS(U,H)}^2
\\
&
\leq
\langle P_{I_1} x - y, P_{I_1} [AP_{I_1} x +  F(P_{I_1} x)] - P_{I_1} [ Ay + F(y)]
\rangle_H
+
\tfrac{(p-1)(1+\varepsilon)}{2} \| B(P_{I_1}x) - B(y) \|_{\HS(U,H)}^2
\\
&
=
\langle P_{I_1} x - y, AP_{I_1} x +  F(P_{I_1} x) -  Ay -  F(y)]
\rangle_H
+
\tfrac{(p-1)(1+\varepsilon)}{2} \| B(P_{I_1}x) - B(y) \|_{\HS(U,H)}^2
\\
&
\leq
C \| P_{I_1}x - y \|_H^2.
\end{split}
\end{equation}
Moreover, observe that for all $x\in H_\gamma, y\in P_{I_1}(H)$ it holds that
\begin{equation}
\begin{split}
&
\langle y- P_{I_1} x, 
P_{I_1} [A P_{I_1} x + F(P_{I_1}x)]
-
P_{I_1} [ A x + F(x)]
\rangle_H 
+
\tfrac{(p-1)(1+\nicefrac{1}{\varepsilon})}{2}
\| P_{I_1} [B(P_{I_1}x)- B(x)]\|_{\HS(U,H)}^2
\\
&
=
\langle y- P_{I_1} x, 
A P_{I_1} x + F(P_{I_1}x)
-
A P_{I_1}x 
- 
F(x)
\rangle_H 
+
\tfrac{(p-1)(1+\nicefrac{1}{\varepsilon})}{2}
\| P_{I_1}[ B(P_{I_1}x)- B(x)]\|_{\HS(U,H)}^2
\\
&
\leq
\|y- P_{I_1} x \|_H
\|
F(P_{I_1}x)
- 
F(x)
\|_H 
+
\tfrac{(p-1)(1+\nicefrac{1}{\varepsilon})}{2}
\| B(P_{I_1}x)- B(x)\|_{\HS(U,H)}^2
\\
&
\leq
\tfrac{1}{2}
\|y- P_{I_1} x \|_H^2
+
\tfrac{1}{2}
\|
F(P_{I_1}x)
- 
F(x)
\|_H^2
+
\tfrac{(p-1)(1+\nicefrac{1}{\varepsilon})}{2}
\| B(P_{I_1}x)- B(x)\|_{\HS(U,H)}^2
\\
&
\leq
\tfrac{1}{2}
\|y- P_{I_1} x \|_H^2
+
\tfrac{1}{2}
\Big(
C
\sqrt{1 + (p-1)(1+\nicefrac{1}{\varepsilon})}
\| P_{I_1}x - x \|_{H_\delta}
(1 +  \| P_{I_1}x \|_{H_\gamma}^c + \| x \|_{H_\gamma}^c )
\Big)^2
.
\end{split}
\end{equation}
Next observe that the H\"older inequality proves that
\begin{equation}
\begin{split}
\label{konc}
&
\big\|
\| P_{I_1} X^{I_2} - X^{I_2} \|_{H_\delta}
(1 +  \| P_{I_1} X^{I_2} \|_{H_\gamma}^c + \| X^{I_2} \|_{H_\gamma}^c )
\big\|_{L^p(\mu_{[0,T]}\otimes \P; \R)}
\\
&
\leq
\big\|
\| P_{I_1} X^{I_2} - X^{I_2} \|_{H_\delta}
\big \|_{L^{qp}(\mu_{[0,T]}\otimes \P; \R)}
\big\|
1 +  \| P_{I_1} X^{I_2} \|_{H_\gamma}^c + \| X^{I_2} \|_{H_\gamma}^c 
\big \|_{L^{rp}(\mu_{[0,T]}\otimes \P; \R)}
\\
&
\leq
T^{\nicefrac{1}{p}}
\sup_{u\in [0,T]} \| P_{\H \setminus I_1} X^{I_2}_u\|_{L^{qp} (\P; H_\delta)}
\bigg[
1
+
\sup_{u\in [0,T]}
\|P_{I_1} X^{I_2}_u \|_{L^{rpc}(\P; H_\gamma)}^c
+
\sup_{u\in [0,T]}
\| X^{I_2}_u \|_{L^{rpc}(\P; H_\gamma)}^c
\bigg]
.
\end{split}
\end{equation}
Combining \eqref{start}--\eqref{konc} with 
Proposition~3.6 in Hutzenthaler \& Jentzen~\cite{HutzenthalerJentzen2014PerturbationArxiv} 
shows that for all $ t \in [0,T] $ it holds that
\begin{equation}
\begin{split}
&
\|X_t^{I_1}-X_t^{I_2}\|_{L^p(\P; H)}
\leq
e^{
\frac{1}{2}-\frac{1}{p}
+ (C+\frac{1}{2})T
}
C
\sqrt{T
\big (1 + (p-1)(1+\tfrac{1}{\varepsilon}) \big)
}
\!
\sup\nolimits_{u\in [0,T]} \| P_{\H \setminus I_1} X_u^{I_2} \|_{L^{qp} (\P; H_\delta)}
\\
&
\cdot
\!
\big[
1
+
\sup\nolimits_{u\in [0,T]}
\|P_{I_1} X_u^{I_2} \|_{L^{rpc}(\P; H_\gamma)}^c
+
\sup\nolimits_{u\in [0,T]}
\|X_u^{I_2} \|_{L^{rpc}(\P; H_\gamma)}^c
\big]
+
\sup\nolimits_{u\in [0,T]}
\| P_{\H \setminus I_1} X_u^{I_2} \|_{L^p(\P; H)}
.
\end{split}
\end{equation}
The proof of Lemma \ref{lem:spatial_rate} is thus completed.
\end{proof}
\subsection{Strong convergence rates for noise discretizations}
\begin{lemma}
\label{noise.galerkin}
Assume the setting in Section \ref{setting6},
let $\kappa \in (0,\infty), \nu \in (0,\nicefrac{1}{2}-\delta), 
\eta\in [\gamma, \infty), I\in \mathcal{P}_0(\H), R_1, R_2\in L(U),$ 
assume that for all $ x, y \in H_1 $ it holds that
$\langle x-y, Ax - Ay + F(x) - F(y) \rangle_H + \tfrac{(p-1)(1+\varepsilon)}{2} \| B(x) - B(y)\|_{\HS(U, H)}^2  \leq C \| x- y \|_H^2,$
and let
$X^{I, R_k} \colon [0,T]\times \Omega \to P_{I}(H_\gamma), k\in \{1,2\},$ be $(\mathcal{F}_t)_{t\in [0,T]}$-adapted
stochastic processes with continuous sample paths such that
for all $t\in [0,T], k\in \{1,2\}$ it holds $\P$-a.s.\,that
\begin{equation}
X^{I, R_k}_t = P_{I}(\xi) + \int_0^t [ A X_s^{I, R_k} +  P_{I} F( X^{I, R_k}_s) ] \, ds
+
\int_0^t P_{I} B ( X^{I, R_k}_s )R_k \, dW_s
.
\end{equation}
Then 
\begin{equation}
\begin{split}
&
\sup\nolimits_{t\in [0,T]}
\| X_t^{I, R_1}- X_t^{I, R_2} \|_{L^p(\P; H)}
\\
&
\leq
\tfrac{\max\{1,T^{1/2}\}\sqrt{p(2p-1)}}{\sqrt{1-2(\delta+\nu)}}
\left[
\sup\nolimits_{v\in H_\eta}
\tfrac{\| B(v) (R_2-R_1) \|_{\HS(U, H_{-\nu})}}
{1+ \| v \|_{H_\eta}^\kappa}
\right]
\left(1 + \sup\nolimits_{s\in [0,T]} 
\|X_s^{I, R_2} \|_{L^{2p\kappa}(\P; H_\eta)}^\kappa
\right)
\\
&
\quad
\cdot
\left(1
+
\left(
TC^2\!\max\{1,\|R_1\|_{L(U)}^2 \}
p(1+\tfrac{1}{\varepsilon})
\right )^{\frac{1}{p}}
\exp\!\left(\tfrac{T C^2\max\{1,\|R_1\|_{L(U)}^2\} p ( 1 + 1 / \varepsilon )}{2}\!\right)
\right)
\\
&
\quad
\cdot
\Bigg[
3
\big(1 + 
\| \xi \|_{L^{2pc}(\P; H_\gamma)}
\big)
C
\!
\max\{1,T\}
\max\{1,\|R_1\|_{L(U)}, \|R_2\|_{L(U)}\}
\\
&
\quad
\cdot
\bigg[
\tfrac{ 1}{1-(\gamma+\alpha)} 
+
\sqrt{\tfrac{pc(2pc-1)}{1-2(\gamma + \beta)}}
\bigg]
\Big[1 + \sup\nolimits_{t\in [0,T]} \| X_t^{I,R_2} \|_{L^{2p c a}(\P; H)}^a\Big]
\Bigg]^c
\!\!
.
\end{split}
\end{equation}
\end{lemma}
\begin{proof}[Proof of Lemma~\ref{noise.galerkin}]
Throughout this proof let $\chi\in [0,\infty)$ be the real number given by $\chi=\tfrac{C(p-1)}{p}\big(1+\tfrac{C\|R_1\|_{L(U)}^2(p-2)(1+\nicefrac{1}{\varepsilon})}{2}\big).$
Let $\hat X\colon [0,T]\times \Omega \to P_I(H_\gamma)$ be an
$(\mathcal{F}_t)_{t\in [0,T]}$-adapted stochastic process 
with continuous sample paths such that for all $t\in [0,T]$ it holds $\P$-a.s.\,that
\begin{equation}
\hat X_t = P_I \xi + \int_0^{t} [ A \hat X_s + P_I F(X_s^{I, R_2}) ]\,ds
+
\int_0^t P_I B(X^{I, R_2}) R_1 \,dW_s
.
\end{equation}
Next observe that
Corollary 2.11 in Hutzenthaler \& Jentzen~\cite{HutzenthalerJentzen2014PerturbationArxiv} 
combined with the Cauchy-Schwarz inequality
proves that for all $t\in [0,T]$ it holds that
\begin{equation}
\label{combi1}
\begin{split}
&
\| X_t^{I, R_2} - X_t^{I, R_1}\|_{L^p(\P; H)}
\\
&
\leq
\!
\| \hat X_t - X_t^{I, R_2} \|_{L^p(\P; H)}
+
e^{(C+\chi)T}
\bigg\|
p \|X^{I, R_1}-\hat X\|_H^{p-2}
\Big[
\| X^{I, R_1} - \hat X\|_H
\|P_IF(\hat X) - P_I F(X^{I, R_2})\|_H
\\
&
\quad
+
\tfrac{(p-1)(1+\nicefrac{1}{\varepsilon})}{2}
\| P_I B(X^{I, R_2}) R_1 - P_I B(\hat X) R_1 \|_{\HS(U,H)}^2
-
\chi \| X^{I, R_1} - \hat X \|_H^2
\Big]^+
\bigg\|_{L^1(\mu_{[0,T]}\otimes \P; \R)}^{\nicefrac{1}{p}}
.
\end{split}
\end{equation}
In addition, note that for all $s\in [0,T]$ it holds that
\begin{equation}
\label{prva.ocena}
\begin{split}
&
\| P_I F(\hat X_s) - P_I F(X_s^{I, R_2})\|_H
\leq
C
\| \hat X_s - X_s^{I, R_2} \|_{H_\delta}
(1 + \|\hat X_s\|_{H_\gamma}^c + \| X_s^{I, R_2} \|_{H_\gamma}^c )
\end{split}
\end{equation}
and
\begin{equation}
\label{druga.ocena}
\begin{split}
\big\| P_I \big( B(X_s^{I, R_2}) - B( \hat X_s)\big) R_1 \big\|_{\HS(U,H)}
\leq
C
\|R_1\|_{L(U)}
\|X_s^{I, R_2} - \hat X_s \|_{H_\delta}
(1 + \|X_s^{I, R_2}\|_{H_\gamma}^c + \| \hat X_s\|_{H_\gamma}^c).
\end{split}
\end{equation}
Combining \eqref{prva.ocena} and \eqref{druga.ocena} shows that 
\begin{equation}
\label{eq103}
\begin{split}
&
\bigg\|
p \|X^{I, R_1}-\hat X\|_H^{p-2}
\Big[
\| X^{I, R_1} - \hat X\|_H
\|P_IF(\hat X) - P_I F(X^{I, R_2})\|_H
\\
&
+
\tfrac{(p-1)(1+\nicefrac{1}{\varepsilon})}{2}
\| P_I B(X^{I, R_2}) R_1 - P_I B(\hat X) R_1 \|_{\HS(U,H)}^2
-
\chi \| X^{I, R_1} - \hat X \|_H^2
\Big]^+
\bigg\|_{L^1(\mu_{[0,T]}\otimes \P; \R)}^{\nicefrac{1}{p}}
\\
&
\leq
\bigg\|
p \|X^{I, R_1}-\hat X\|_H^{p-2}
\Big[
\| X^{I, R_1} - \hat X\|_H
C
\| \hat X - X^{I, R_2} \|_{H_\delta}
(1 + \|\hat X\|_{H_\gamma}^c + \| X^{I, R_2} \|_{H_\gamma}^c )
\\
&
+
\tfrac{(p-1)(1+\nicefrac{1}{\varepsilon})C^2\|R_1\|_{L(U)}^2}{2}
\|X^{I, R_2} - \hat X \|_{H_\delta}^2
(1 + \|X^{I, R_2}\|_{H_\gamma}^c + \| \hat X \|_{H_\gamma}^c)^2
-
\chi \| X^{I, R_1} - \hat X \|_H^2
\Big]^+
\bigg\|_{L^1(\mu_{[0,T]}\otimes \P; \R)}^{\nicefrac{1}{p}}
.
\end{split}
\end{equation}
Moreover, observe that Young's inequality proves that for all $s\in [0,T]$ it holds that
\begin{equation}
\label{eq104}
\begin{split}
&
\| X_s^{I, R_1} - \hat X_s \|_H^{p-1}
\| \hat X_s - X_s^{I, R_2} \|_{H_\delta}
(1 + \|\hat X_s\|_{H_\gamma}^c + \| X_s^{I, R_2} \|_{H_\gamma}^c )
\\
&
\leq
\tfrac{p-1}{p}
\| X_s^{I, R_1} - \hat X_s \|_H^p
+
\tfrac{1}{p}
\| \hat X_s - X_s^{I, R_2} \|_{H_\delta}^p
(1 + \|\hat X_s\|_{H_\gamma}^c + \| X_s^{I, R_2} \|_{H_\gamma}^c )^p
\end{split}
\end{equation}
and
\begin{equation}
\label{eq105}
\begin{split}
&
\| X_s^{I, R_1} - \hat X_s \|_H^{p-2}
\| \hat X_s - X_s^{I, R_2} \|_{H_\delta}^2
(1 + \|\hat X_s\|_{H_\gamma}^c + \| X_s^{I, R_2} \|_{H_\gamma}^c )^2
\\
&
\leq
\tfrac{p-2}{p}
\| X_s^{I, R_1} - \hat X_s \|_H^p
+
\tfrac{2}{p}
\| \hat X_s - X_s^{I, R_2} \|_{H_\delta}^p
(1 + \|\hat X_s\|_{H_\gamma}^c + \| X_s^{I, R_2} \|_{H_\gamma}^c )^p
.
\end{split}
\end{equation}
Combining \eqref{eq103}, \eqref{eq104}, and \eqref{eq105} shows that
\begin{equation}
\label{combi2}
\begin{split}
&
\bigg\|
p \|X^{I, R_1}-\hat X\|_H^{p-2}
\Big[
\| X^{I, R_1} - \hat X\|_H
\|P_IF(\hat X) - P_I F(X^{I, R_2})\|_H
\\
&
\quad
+
\tfrac{(p-1)(1+\nicefrac{1}{\varepsilon})}{2}
\| P_I B(X^{I, R_2}) R_1 - P_I B(\hat X) R_1 \|_{\HS(U,H)}^2
-
\chi \| X^{I, R_1} - \hat X \|_H^2
\Big]^+
\bigg\|_{L^1(\mu_{[0,T]}\otimes \P; \R)}^{\nicefrac{1}{p}}
\\
&
\leq
\bigg\|
(p-1)C
\| X^{I, R_1} - \hat X \|_H^p
+
C
\| \hat X - X^{I, R_2} \|_{H_\delta}^p
(1 + \|\hat X\|_{H_\gamma}^c + \| X^{I, R_2} \|_{H_\gamma}^c )^p
\\
&
\quad
+
\tfrac{(p-1)(p-2)(1+\frac{1}{\varepsilon})C^2\|R_1\|_{L(U)}^2}{2}
\| X^{I, R_1} - \hat X \|_H^p
\\
&
\quad
+
C^2
\|R_1\|_{L(U)}^2
(p-1)(1+\tfrac{1}{\varepsilon})
\| \hat X - X^{I, R_2} \|_{H_\delta}^p
(1 + \|\hat X\|_{H_\gamma}^c + \| X^{I, R_2} \|_{H_\gamma}^c )^p
\\
&
\quad
-
(p-1)C\Big(1+\tfrac{\|R_1\|_{L(U)}^2 C(p-2)(1+\nicefrac{1}{\varepsilon})}{2}\Big)\|X^{I, R_1} - \hat X\|_H^p
\bigg\|_{L^1(\mu_{[0,T]}\otimes \P; \R)}^{\nicefrac{1}{p}}
\\
&
=
\left(
C
+
C^2
\|R_1\|_{L(U)}^2
(p-1)(1+\tfrac{1}{\varepsilon})
\right )^{\frac{1}{p}}
\left\|
\| \hat X - X^{I, R_2} \|_{H_\delta}
(1 + \|\hat X\|_{H_\gamma}^c + \| X^{I, R_2} \|_{H_\gamma}^c )
\right\|_{L^p(\mu_{[0,T]}\otimes \P; \R)}
.
\end{split}
\end{equation}
Furthermore, note that the H\"older inequality implies that
\begin{equation}
\label{combi3}
\begin{split}
&
\left\|
\| \hat X - X^{I, R_2} \|_{H_\delta}
(1 + \|\hat X\|_{H_\gamma}^c + \| X^{I, R_2} \|_{H_\gamma}^c )
\right\|_{L^p(\mu_{[0,T]}\otimes \P; \R)}
\\
&
\leq
T^{\frac{1}{p}}
\sup_{t\in[0,T]}
\| \hat X_t - X_t^{I, R_2} \|_{L^{2p}(\P; H_\delta)}
\bigg(1 + 
\sup_{t\in[0,T]} \|\hat X_t\|_{L^{2pc}(\P; H_\gamma)}^c 
+ 
\sup_{t\in[0,T]} \| X_t^{I, R_2} \|_{L^{2pc}(\P; H_\gamma)}^c 
\bigg)
.
\end{split}
\end{equation}
Combining \eqref{combi1}, \eqref{combi2}, and \eqref{combi3} proves that
\begin{equation}
\begin{split}
\label{combi107}
&
\sup\nolimits_{t\in [0,T]}
\| X_t^{I, R_2} - X_t^{I, R_1}\|_{L^p(\P; H)}
\\
&
\leq
\sup\nolimits_{t\in [0,T]}
\| \hat X_t - X_t^{I, R_2} \|_{L^p(\P; H)}
+
e^{(C+\chi)T}
\!\!
\left(
T
C
+
T
C^2
\|R_1\|_{L(U)}^2
(p-1)(1+\tfrac{1}{\varepsilon})
\right )^{\frac{1}{p}}
\\
&
\cdot
\sup\nolimits_{t\in[0,T]}
\| \hat X_t - X_t^{I, R_2} \|_{L^{2p}(\P; H_\delta)}
\left(1 + 
\sup\nolimits_{t\in[0,T]} \|\hat X_t\|_{L^{2pc}(\P; H_\gamma)}^c 
+ 
\sup\nolimits_{t\in[0,T]} \| X_t^{I, R_2} \|_{L^{2pc}(\P; H_\gamma)}^c \right)
.
\end{split}
\end{equation}
In the next step observe that 
the Burkholder-Davis-Gundy type inequality in Lemma~7.7 
in Da Prato \& Zabczyk~\cite{dz92}
shows that
for all $t\in [0,T], r\in [0, \gamma], q\in [2,\infty)$ it holds that
\begin{equation}
\begin{split}
\label{combi108}
&
\| \hat X_t - X_t^{I, R_2} \|_{L^q(\P; H_r)}
=
\left \| \int_0^t e^{(t-s)A} P_I B(X_s^{I, R_2}) (R_2-R_1) \,dW_s \right\|_{L^q(\P; H_r)}
\\
&
\leq
\sqrt{
\tfrac{q(q-1)}{2}
\int_0^t \left \|  
e^{(t-s)A} P_I B(X_s^{I, R_2}) (R_2-R_1)
\right\|_{L^q(\P; \HS(U, H_r))}^2 ds}
.
\end{split}
\end{equation}
Moreover, note that,
e.g., \cite[Theorem~2.5.34]{Jentzen2014SPDElecturenotes} 
implies that
for all $r\in [0,\gamma], t\in [0,T], q\in [2,\infty)$ 
it holds that
\begin{equation}
\begin{split}
\label{combi109}
&
\sqrt{
\int_0^t \left \|  
e^{(t-s)A} P_I B(X_s^{I, R_2}) ( R_2 - R_1 )
\right\|_{L^q(\P; \HS(U, H_r))}^2 ds}
\\
&
\leq
\sqrt{
\int_0^t  \|  
(-A)^{r+\nu} e^{(t-s)A} 
\|_{L(H)}^2
\|(-A)^{-\nu}  B(X_s^{I, R_2}) (R_2-R_1)
\|_{L^q(\P; \HS(U, H))}^2 ds}
\\
&
\leq
\sqrt{
\int_0^t 
(t-s)^{-2(r+\nu)}
\|  B(X_s^{I, R_2}) (R_2-R_1)
\|_{L^q(\P; \HS(U, H_{-\nu}))}^2 ds}
\\
&
\leq
\tfrac{T^{\nicefrac{1}{2}-(r+\nu)}}{\sqrt{1-2(r+\nu)}}
\sup\nolimits_{s\in [0,T]}
\|  B(X_s^{I, R_2}) (R_2-R_1)
\|_{L^q(\P; \HS(U, H_{-\nu}))}
\\
&
\leq
\tfrac{T^{\nicefrac{1}{2}-(r+\nu)}}{\sqrt{1-2(r+\nu)}}
\left(
\sup_{v\in H_\eta}
\tfrac{\| B(v) (R_2-R_1) \|_{\HS(U, H_{-\nu})}}
{1+ \| v \|_{H_\eta}^\kappa}
\right)
\left (1 + \sup\nolimits_{s\in [0,T]} 
\|X_s^{I, R_2} \|_{L^{q \kappa}(\P; H_\eta)}^\kappa
\right)
.
\end{split}
\end{equation}
In addition, note that
Lemma \ref{bootstrap} shows that
\begin{equation}
\begin{split}
\label{use.bootstrap}
&
\max \{
\sup\nolimits_{t\in[0,T]}\|\hat X_t\|_{L^{2pc}(\P; H_\gamma)},
\sup\nolimits_{t\in[0,T]}\|X^{I,R_2}_t\|_{L^{2pc}(\P; H_\gamma)}
\}
\leq
\left \| \xi \right \|_{L^{2pc}(\P; H_\gamma)}
+
C
\max\{1,T\}
\\
&
\cdot
\max\{1,\|R_1\|_{L(U)}, \|R_2\|_{L(U)}\}
\bigg[
\tfrac{ 1}{1-(\gamma+\alpha)} 
+
\sqrt{\tfrac{pc(2pc-1)}{1-2(\gamma + \beta)}}
\bigg]
\big[1 + \sup\nolimits_{t\in [0,T]} \| X_t^{I,R_2} \|_{L^{2p c a}(\P; H)}^{a} \big]
.
\end{split}
\end{equation}
Combining \eqref{combi107}-\eqref{use.bootstrap}
proves that 
\begin{equation}
\begin{split}
&
\sup_{t\in [0,T]}
\| X_t^{I, R_2}- X_t^{I, R_1} \|_{L^p(\P; H)}
\\
&
\leq
\tfrac{\sqrt{p(2p-1)}T^{\nicefrac{1}{2}-\nu}}{\sqrt{1-2\nu}}
\left(
\sup_{v\in H_\eta}
\tfrac{\| B(v) (R_2-R_1) \|_{\HS(U, H_{-\nu})}}
{1+ \| v \|_{H_\eta}^\kappa}
\right)
\left(\!1 + \sup\nolimits_{s\in [0,T]} 
\|X_s^{I, R_2}\|_{L^{p \kappa}(\P; H_\eta)}^\kappa
\right)
\\
&
+
\tfrac{\sqrt{p(2p-1)} e^{(C+\chi)T}
\left(
\max\{1,\|R_1\|_{L(U)}^2 \}
TC^2
p(1+\nicefrac{1}{\varepsilon})
\right )^{1/p}T^{\nicefrac{1}{2}-(\delta+\nu)}}{\sqrt{1-2(\delta+\nu)}}
\left(
\sup_{v\in H_\eta}
\tfrac{\| B(v) (R_2-R_1) \|_{\HS(U, H_{-\nu})}}
{1+ \| v \|_{H_\eta}^\kappa}
\right)
\\
&
\cdot
\Big(\!1 + \sup\nolimits_{s\in [0,T]} 
\|X_s^{I, R_2} \|_{L^{2p\kappa}(\P; H_\eta)}^\kappa
\!
\Big)
\!
\Bigg(\!1 +
2\Bigg[\|\xi\|_{L^{2pc}(\P; H_\gamma)}
+
C
\max\{1,T\}
\max\{1,\|R_1\|_{L(U)}, \|R_2\|_{L(U)}\}
\\
&
\cdot
\bigg[
\tfrac{ 1}{1-(\gamma+\alpha)} 
+
\sqrt{\tfrac{pc(2pc-1)}{1-2(\gamma + \beta)}}
\bigg]
\big[1 + \sup\nolimits_{t\in [0,T]} \| X_t^{I,R_2} \|_{L^{2p c a}(\P; H)}^{a} \big]
\Bigg]^c
\Bigg)
.
\end{split}
\end{equation}
The proof of Lemma \ref{noise.galerkin} is thus completed.
\end{proof}
\subsection{Strong convergence rates for space-time-noise
discretizations}
\begin{proposition}
\label{last}
Assume the setting in Section \ref{setting6}, 
let $q\in (0,p),$
assume that $\xi\in L^{pa}(\P; H_{\gamma}),$
assume that for all $ x \in H_1 $ it holds that
$\langle x, F(x) \rangle_H + \tfrac{pa-1}{2}
\| B(x)\|_{\HS(U,H)}^2 \leq C(1+\| x\|_H^2),$
let $I_n \in \mathcal{P}_0(\H), n\in \N,$ satisfy 
$\cup_{n\in \N} \left( \cap_{m\in \{n+1, n+2,\ldots\}} I_m \right) =\H,$
and let
$X^I\colon [0,T]\times \Omega \to P_{I}(H_\gamma),$
$I\in \mathcal{P}_0(\H),$
be $(\mathcal{F}_t)_{t\in [0,T]}$-adapted stochastic processes
with continuous sample paths such that for all $t\in [0,T], I\in \mathcal{P}_0(\H)$ it 
holds $\P$-a.s.\,that
\begin{equation}
X_t^{I} = P_{I} \xi + \int_0^t [ A X_s^{I} + P_{I}F( X_s^{I})] \, ds
+
\int_0^t P_{I} B ( X^{I}_s ) \, dW_s
.
\end{equation}
Then it holds that
$
\lim\nolimits_{ n\to \infty} 
\left(
\sup\nolimits_{t\in [0,T]}
\| X_t - X_t^{I_n} \|_{L^q(\P; H_\gamma)}
\right)
=
0
$
and
\begin{equation}
\begin{split}
&
\sup_{
\substack{R\in L^1(U), I\in \mathcal{P}_0(\H),\\ N\in \N, t\in [0,T]}} 
\E\!\left[
\big\|Y_t^{N,I,R} \big\|_H^{pa}
+ 
\|X_t^I\|_H^{pa}
+
\big\|Y_t^{N,I,R} \big\|_{H_\gamma}^p
+
\| X_t \|_{H_\gamma}^p
+
\| X_t^I \|_{H_\gamma}^p
\right]
<\infty
.
\end{split}
\end{equation}
\end{proposition}
\begin{proof}[Proof of Proposition~\ref{last}]
It\^o's formula and Young's inequality prove that for all 
$t\in [0,T], I\in \mathcal{P}_0(\H)$ it holds that
\begin{equation}
\begin{split}
&
\E [ \| X_t^{I}\|_H^{pa} ]
\\
&
\leq
\E [ \| X_0^{I} \|_H^{pa} ]
+
pa
\int_0^t
\E
\Big[ \|X_s^{I}\|_H^{pa-2} 
\big( \langle X_s^{I}, P_{I} F(X_s^{I})\rangle_H + \tfrac{pa-1}{2}
\| P_{I} B(X_s^{I})\|_{\HS(U,H)}^2
\big )
\Big ]
\, ds
\\
&
\leq
\E [ \| X_0^{I} \|_H^{pa} ]
+
paC
\int_0^t
\E
\big[ \|X_s^{I}\|_H^{pa-2} 
(1 + \|X_s^{I}\|_H^2)
\big ]
\, ds
\\
&
\leq
\E [ \| X_0^{I} \|_H^{pa} ]
+
2C
\int_0^t
\E
\big[
(pa-1)
\|X_s^{I}\|_H^{pa}
+ 1
\big ]
ds
.
\end{split}
\end{equation}
Therefore, Gronwall's lemma proves that for all $t\in [0,T], I\in \mathcal{P}_0(\H)$ it holds that
$\E[ \| X_t^{I}\|_H^{pa} ] \leq ( \E[ \| \xi \|_H^{pa} ] + 2CT) \, e^{ 2 C ( p a - 1 ) T }. $
Lemma \ref{bootstrap} hence implies that
$
\sup\nolimits_{I\in \mathcal{P}_0(\H)} 
\sup\nolimits_{t\in [0,T]}  \| X_t^{I}\|_{L^{p}(\P; H_\gamma)}<\infty.
$
Combining this with Corollary \ref{galerkin.convergence} shows that
$
\lim\nolimits_{ n\to \infty} 
\left(
\sup\nolimits_{t\in [0,T]}
\| X_t - X_t^{I_n} \|_{L^q(\P; H_\gamma)}
\right)
=
0
$
and that
$\sup_{t\in [0,T]} \| X_t \|_{L^p(\P; H_\gamma)} <\infty.$
Moreover, note that combining Lemma \ref{Lemma1} and Lemma \ref{bootstrap} proves that
\begin{equation}
\sup\nolimits_{R\in L^1(U)}
\sup\nolimits_{I\in \mathcal{P}_0(\H)}
\sup\nolimits_{N\in \N}
\sup\nolimits_{t\in [0,T]} 
\left( \|Y_t^{N,I,R} \|_{L^{pa}(\P; H)}
+ \|Y_t^{N,I,R} \|_{L^p(\P; H_\gamma)}
\right)<\infty.
\end{equation}
The proof of Proposition \ref{last} is thus completed.
\end{proof}
\begin{corollary}
\label{finite.function}
Assume the setting in Section \ref{setting6},
assume that 
$ \xi \in L^{p(c+1)a}(\P; H_\gamma) $,
assume that for all 
$ x \in H_1 $ 
it holds that
$
\langle x, F(x)\rangle_H
+
\tfrac{  p (c+1) a -1}{2}
\left \| B(x) \right \|_{\HS(U,H)}^2
\leq
\d (1 + \left \| x \right\|_{H}^2)
$,
and let
$X^I\colon [0,T]\times \Omega \to P_{I}(H_\gamma),$
$I\in \mathcal{P}_0(\H),$
be $(\mathcal{F}_t)_{t\in [0,T]}$-adapted stochastic processes
with continuous sample paths such that for all $t\in [0,T], I\in \mathcal{P}_0(\H)$ it 
holds $\P$-a.s.\,that
\begin{equation}
X_t^{I} = P_{I} \xi + \int_0^t [ A X_s^{I} + P_{I}F( X_s^{I})] \, ds
+
\int_0^t P_{I} B ( X^{I}_s ) \, dW_s
.
\end{equation}
Then 
\begin{equation}
\label{machu}
\sup_{\substack{I\in \mathcal{P}_0(\H), N\in \N,
\\ R\in L^1(U), t \in [0,T] }}
\!
\left \|  
\| F(Y^{N,I,R}_t)\|_H +  \|F(X_t)\|_H  +  \| B(Y^{N,I,R}_t)\|_{\HS(U,H)} + \| B(X_t)\|_{\HS(U,H)} 
\right \|_{L^p(\P; \R)}
\!\!\!\!
<\infty.
\end{equation}
\end{corollary}
\begin{proof}[Proof
of Corollary~\ref{finite.function}]
Observe that for all $t\in [0,T], N\in \N, I\in \mathcal{P}_0(\H), R\in L^1(U)$ it holds that
\begin{equation}
\begin{split}
\label{ratata}
&
\max\{
\| F(Y^{N,I,R}_t) \|_{L^p(\P; H)}
,
\| B(Y^{N,I,R}_t) \|_{L^p(\P; \HS(U,H))}
\}
\\
&
\leq
\| F( 0 ) \|_H
+
\| B( 0 ) \|_{\HS(U,H)}
+
C
\| (-A)^{\delta - \gamma} \|_{L(H)}
\big (
1
+
\| Y^{N,I,R}_t \|_{L^{p (c+1)}(\P; H_\gamma)}
\big )^{c+1}
.
\end{split}
\end{equation}
Proposition \ref{last} hence proves that
\begin{equation}
\sup\nolimits_{R\in L^1(U)}
\sup\nolimits_{I \in \mathcal{P}_0(\H)}
\sup\nolimits_{N\in \N}
\sup\nolimits_{t\in [0,T]}
\!\left(
\|F(Y^{N,I,R}_t)\|_{L^p(\P; H)}
+
\|B(Y^{N,I,R}_t)\|_{L^p(\P; H)}
\right)
<\infty.
\end{equation}
Moreover, note that for all $t\in [0,T]$ it holds that
\begin{equation}
\begin{split}
&
\max\{
\| F(X_t) \|_{L^p(\P; H)}
,
\| B(X_t) \|_{L^p(\P; \HS(U,H))}
\}
\\
&
\leq
\| F( 0 ) \|_H
+
\| B( 0 ) \|_{\HS(U,H)}
+
C
\| (-A)^{\delta - \gamma} \|_{L(H)}
\big (
1
+
\| X_t \|_{L^{p (c+1)}(\P; H_\gamma)}
\big )^{c+1}
.
\end{split}
\end{equation}
Proposition \ref{last} 
hence
shows that
\begin{equation}
\sup\nolimits_{t\in [0,T]}
\|F(X_t)\|_{L^p(\P; H)}
+
\sup\nolimits_{t\in [0,T]}
\|B(X_t)\|_{L^p(\P; H)}
<\infty.
\end{equation}
The proof of Corollary \ref{finite.function} is thus completed.
\end{proof}
\begin{corollary}
\label{rex}
Assume the setting in Section \ref{setting6}, 
let $\eta\in [\gamma,\nicefrac{1}{2}),$
assume that $\xi(\Omega)\subseteq H_\eta,$
assume that
$\E\big[\|\xi\|_{H_\eta}^{p(c+1)a}\big]<\infty,$
assume that for all $ x \in H_1 $ it holds that
$\langle x, F(x) \rangle_H + \tfrac{p(c+1)a-1}{2}
\| B(x)\|_{\HS(U,H)}^2 \leq C(1+\| x\|_H^2),$
and let
$X^I\colon [0,T]\times \Omega \to P_{I}(H_\gamma), I\in \mathcal{P}_0(\H),$
be $(\mathcal{F}_t)_{t\in [0,T]}$-adapted stochastic processes
with continuous sample paths such that for all $t\in [0,T], I\in \mathcal{P}_0(\H)$ it 
holds $\P$-a.s.\,that
\begin{equation}
X_t^{I} = P_{I} \xi + \int_0^t [ A X_s^{I} + P_{I}F( X_s^{I})] \, ds
+
\int_0^t P_{I} B ( X^{I}_s ) \, dW_s
.
\end{equation}
Then it holds for all $ t \in [0,T] $ that $ 
\P( X_t \in H_{ \eta } ) = 1 $ and it holds that 
\begin{equation}
\label{eq11111}
\sup\nolimits_{R\in L^1(U)}
\sup\nolimits_{I\in \mathcal{P}(\H)}
\sup\nolimits_{N\in \N}
\sup\nolimits_{t\in [0,T]}
\big(
\| Y_t^{N,I,R}\|_{L^p(\P; H_\eta)}
+
\|X_t^I\|_{L^p(\P; H_\eta)}
\big)
<\infty.
\end{equation}
\end{corollary}
\begin{proof}[Proof of Corollary~\ref{rex}]
Combining Lemma \ref{more.improved} and Corollary \ref{finite.function} proves $ \forall 
\, t \in [0,T] \colon P( X_t \in H_{ \eta } ) = 1 $ and \eqref{eq11111}. The proof 
of Corollary \ref{rex} is thus completed.
\end{proof}
\begin{theorem}
\label{Finale}
Assume the setting in Section \ref{setting6},
let
$ \nu \in ( 0, \nicefrac{ 1 }{ 2 } - \delta ) $,
$ \eta \in [ \max\{ \delta, \gamma \}, \nicefrac{ 1 }{ 2 } ) $,
$ \kappa \in ( \nicefrac{ 2 }{ p }, \infty) $,
assume that
$
  \xi( \Omega ) \subseteq H_{ \eta }
$,
assume that
$
  \E\big[
    \| \xi \|_{ H_{ \eta } }^{ 2 a (c+1) p \max\{ \kappa, \nicefrac{ 1 }{ \theta } \} } 
  \big]
  < \infty
$,
and assume
that for all $ x , y \in H_1 $ it holds that
$
  \langle x, F(x) \rangle_H
  +
  \frac{ 2a(c+1)p\max \{ \kappa, \nicefrac{1}{\theta}\} - 1}{ 2  }
  \left\| 
    B(x) 
  \right \|_{\HS(U,H)}^2
\leq
  \d ( 1 + \| x \|_H^2 )
$
and
$
  \langle x - y , Ax - Ay + F(x) - F(y) \rangle_H 
  + 
  \tfrac{ ( p - 1 )( 1 + \varepsilon ) }{ 2 } 
  \left\| B(x) - B(y) \right\|_{\HS( U, H) }^2  
  \leq C \left\| x - y \right\|_H^2
$.
Then 
there exists a real number $K \in [0,\infty)$ such that for all $N \in \N, I\in \mathcal{P}_0(\H),$
$R\in L^1(U)$
it holds that
\begin{equation}
\begin{split}
  \sup_{ t \in [0,T] }
  \| 
    X_t - Y_t^{ N, I, R } 
  \|_{ L^p(\P; H) }
\leq
  K
  \left[
    N^{  \delta - \eta }
    +
    \| P_{ \H \setminus I } \|_{ L(H, H_{\delta-\eta}) }
    +
    \sup_{v\in H_\eta}
    \Big(
      \tfrac{
        \| 
          B(v) ( \operatorname{Id}_U - R ) 
        \|_{
          \HS( U, H_{ - \nu } ) }
        }{
          ( 1 + \| v \|_{ H_{ \eta } } )^\kappa
        }
    \Big)
  \right]
.
\end{split}
\end{equation}
\end{theorem}
\begin{proof}[Proof of Theorem~\ref{Finale}]
First of all, observe that it is well 
known that the fact that the functions
$P_I(H)\ni x\mapsto P_I(F(x))\in P_I(H),I\in \mathcal{P}_0(\H),$
and 
$P_I(H)\ni x \mapsto P_I(B(x))R \in \HS(P(U),P_I(H)), I\in \mathcal{P}_0(\H), R\in L^1(U),$
are 
locally Lipschitz continuous and the fact that 
$ \forall\,x\in H_1 \colon 
\langle x, F(x)\rangle_H
+\frac{2a(c+1)p\max \{ \kappa, \nicefrac{1}{\theta}\}-1}{2}\|B(x)\|_{\HS(U,H)}^2\leq C(1+\|x\|_H^2)$ ensure that there exist $(\mathcal{F}_t)_{t\in [0,T]}$-adapted stochastic processes
$X^{I,R}\colon [0,T]\times \Omega \to P_I(H_\gamma), I\in \mathcal{P}_0(\H),
R\in L^1(U),$
with continuous sample paths such that for all $t\in [0,T], I\in \mathcal{P}_0(\H),R\in L^1(U)$ it
holds $\P$-a.s.\,that
\begin{equation}
X_t^{I,R} = P_I\xi + \int_0^t [ A X_s^{I,R} + P_IF( X_s^{I,R})] \, ds
+
\int_0^t P_IB ( X^{I,R}_s )R \, dW_s
.
\end{equation}
Moreover, note that the triangle inequality 
proves that for all $t\in [0,T], N \in \N, I, \tilde I \in \mathcal{P}_0(\H),$
$R\in L^1(U)$ 
with $I\subseteq \tilde I$
it holds that
\begin{equation}
\begin{split}
\| X_t - Y_t^{N,I,R}  \|_{L^p(\P;H)}
&
\leq
 \| X_t - X^{\tilde I, \operatorname{Id}_U }_t  \|_{L^p(\P;H)}
+
 \| X_t^{\tilde I, \operatorname{Id}_U } - X^{I, \operatorname{Id}_U }_t \|_{L^p(\P;H)}
\\
&
+
\| X_t^{ I,  \operatorname{Id}_U } - X^{I, R}_t\|_{L^p(\P;H)}
+
\| X^{I,R}_t - Y_t^{N,I,R}  \|_{L^p(\P;H)}
.
\end{split}
\end{equation}
In the next step
we note that the assumption that $H$ is separable implies that there exist non-decreasing sets $ I_n \in 
\mathcal{P}_0( \H ) $, $ n \in \N $, with the property that $ \cup_{ n
\in \N } I_n = \H $. 
Next we combine Corollary \ref{corollary.combined}, Lemma \ref{lem:spatial_rate},
and
Lemma \ref{noise.galerkin}
to obtain that for all $ t \in 
[0,T], N, n \in \N, I \in \mathcal{P}_0( \H ), R\in L^1(U)$
with 
$ I 
\subseteq I_n$
it holds that
\begin{align}
\label{kombi}
\nonumber
&
\| X_t - Y_t^{N,I,R}  \|_{L^p(\P;H)}
\leq
\| X_t - X_t^{I_n, \operatorname{Id}_U }\|_{L^p(\P; H)}
+
\big(
\| (-A)^{-\delta}\|_{L(H)}
+
e^{
(C+1)T
}
C
p(1+\tfrac{1}{\varepsilon}) 
\big)
\\
\nonumber
&
\cdot
\sup\nolimits_{u\in [0,T]} \| P_{\H \setminus I} X_u^{I_n, \operatorname{Id}_U } \|_{L^{2p} (\P; H_\delta)}
\big[
1
+
2
\sup\nolimits_{u\in [0,T]}
\| X_u^{I_n, \operatorname{Id}_U } \|_{L^{2pc}(\P; H_\gamma)}^c
\big]
+
\tfrac{\sqrt{p(2p-1)}\max\{1,T^{1/2}\}}{\sqrt{1-2(\delta+\nu)}}
\\
\nonumber
&
\cdot
\left(1
+
\exp\!\left(\tfrac{T C^2
p ( 1 + 1 / \varepsilon )}{2}\right)
\left(
TC^2
p(1+\tfrac{1}{\varepsilon})
\right )^{\frac{1}{p}}
\right)
\left(1 + \sup\nolimits_{s\in [0,T]} 
\|X_s^{I,  \operatorname{Id}_U } \|_{L^{2p\kappa}(\P; H_\eta)}^\kappa
\right)
\\
&
\cdot
\left[
\sup\nolimits_{v\in H_\eta}
\tfrac{\| B(v) (\operatorname{Id}_U-R) \|_{\HS(U, H_{-\nu})}}
{1+ \| v \|_{H_\eta}^\kappa}
\right]
\Bigg[
3
\big(1 + 
\| \xi \|_{L^{2pc}(\P; H_\gamma)}
\big)
C
\!
\max\{1,T\}
\bigg[
\tfrac{ 1}{1-(\gamma+\alpha)} 
+
\sqrt{\tfrac{pc(2pc-1)}{1-2(\gamma + \beta)}}
\bigg]
\\
\nonumber
&
\cdot
\Big[1 + \sup\nolimits_{t\in [0,T]} \| X_t^{I,  \operatorname{Id}_U } \|_{L^{2p c a}(\P; H)}^a\Big]
\Bigg]^c
+
N^{ \delta - \eta}
\tfrac{
\max\{1, T^2\}
}
{(1-2\eta)}
\big (C^2(1+\nicefrac{1}{\varepsilon}) p\big)^{\nicefrac{1}{p}}
\exp\!\left(\tfrac{T C^2 p ( 1 + 1 / \varepsilon )}{2} \right)
\\
\nonumber
&
\cdot
\bigg(
3
\Big[
1+
\!
\sup\nolimits_{s\in [0,T]} \|F( Y^{N,I,R}_s )  \|_{L^{\nicefrac{ 2p}{\theta}}(\P; H)}
+
\sqrt{p(2p-1)}
\sup\nolimits_{s\in [0,T]} \| B( Y^{N,I,R}_s ) \|_{L^{\nicefrac{ 2p}{\theta}}(\P; \HS(U,H))}
\Big ]^{1+\frac{1}{ 2 \theta}}
\\
\nonumber
&
+
\sup\nolimits_{s\in [0,T]}
\| Y^{N,I,R}_s\|_{L^{2p}(\P; H_\eta)}
\bigg)
\bigg(
1 
+  
2
\bigg[
\| \xi \|_{L^{2pc}(\P; H_\gamma)}
+
C
\Big[
\tfrac{ T^{1-(\gamma+\alpha)}}{1-(\gamma+\alpha)} 
+
\sqrt{\tfrac{pc(2pc-1)}{(1-2(\gamma + \beta))}}
T^{\frac{1}{2}-(\gamma + \beta)}
\Big]
\\
\nonumber
&
\cdot
\Big[1 + \sup\nolimits_{s\in [0,T]} \| Y^{N,I,R}_s \|_{L^{2pca}(\P; H)}^{a} \Big]
\bigg]^c
\bigg)
\!
.
\end{align}
Moreover, observe that
Proposition \ref{last}
shows that
\begin{equation}
\label{eq111}
\sup_{I\in \mathcal{P}_0(\H)}
\sup_{R\in L^1(U)} 
\sup_{t\in [0,T]}
\sup_{N\in\N}
\left(
\| X_t \|_{L^{2pc}(\P; H_\gamma)}
+
\| X^{I, \operatorname{Id}_U}_t \|_{L^{2pca}(\P; H)}
+
\| Y_t^{N,I,R}\|_{L^{2pca}(\P; H)}
\right)
<\infty
\end{equation}
and
$\lim_{n\to \infty} 
\sup_{t\in [0,T]}\E [ \| X_t - X_t^{I_n, \operatorname{Id}_U } \|_{H_\gamma}^{2pc} ] = 0.$
Combining this with \eqref{kombi} implies that for all $N\in \N, I \in \mathcal{P}_0( \H ),$ 
$R\in L^1(U)$
it holds that
\begin{align}
\label{kombi.implies}
\nonumber
&
\sup\nolimits_{t\in [0,T]}
\| X_t - Y_t^{N,I,R}  \|_{L^p(\P;H)}
\leq
\big(
\| (-A)^{-\delta}\|_{L(H)}
+
e^{
(C+1)T
}
C
p(1+\tfrac{1}{\varepsilon}) 
\big)
\sup\nolimits_{u\in [0,T]} \| P_{\H \setminus I} X_u \|_{L^{2p} (\P; H_\delta)}
\\
\nonumber
&
\cdot
\!
\big[
1
+
2
\sup\nolimits_{u\in [0,T]}
\| X_u \|_{L^{2pc}(\P; H_\gamma)}^c
\big]
+
\tfrac{\sqrt{p(2p-1)}\max\{1,T^{1/2}\}}{\sqrt{1-2(\delta+\nu)}}
\left[
\sup\nolimits_{v\in H_\eta}
\tfrac{\| B(v) (\operatorname{Id}_U - R) \|_{\HS(U, H_{-\nu})}}
{1+ \| v \|_{H_\eta}^\kappa}
\right]
\\
&
\nonumber
\cdot
\left(1
+
\exp\!\left(\tfrac{T C^2
p ( 1 + 1 / \varepsilon )}{2}\right)
\left(
TC^2
p(1+\tfrac{1}{\varepsilon})
\right )^{\frac{1}{p}}
\right)
\left(1 + \sup\nolimits_{s\in [0,T]} 
\|X_s^{I,  \operatorname{Id}_U } \|_{L^{2p\kappa}(\P; H_\eta)}^\kappa
\right)
\\
&
\cdot
\Bigg[
3\big(1 + 
\| \xi \|_{L^{2pc}(\P; H_\gamma)}
\big)C\!\max\{ 1, T \}
\bigg[
\tfrac{ 1}{1-(\gamma+\alpha)} 
+
\sqrt{\tfrac{pc(2pc-1)}{1-2(\gamma + \beta)}}
\bigg]
\big[1 + \sup\nolimits_{t\in [0,T]} \| X_t^{I,  \operatorname{Id}_U } \|_{L^{2p c a}(\P; H)}^a\big]
\Bigg]^c
\\
\nonumber
&
+
N^{ \delta - \eta}
\tfrac{
\max\{1, T^2\}
}
{(1-2\eta)}
\big (C^2(1+\nicefrac{1}{\varepsilon}) p\big)^{\nicefrac{1}{p}}
\exp\!\left(\tfrac{T C^2 p ( 1 + 1 / \varepsilon )}{2} \right)
\bigg(
3
\Big[
1+
\!
\sup\nolimits_{s\in [0,T]} \|F( Y^{N,I,R}_s )  \|_{L^{\nicefrac{ 2p}{\theta}}(\P; H)}
\\
\nonumber
&
+
\sqrt{p(2p-1)}
\sup\nolimits_{s\in [0,T]} \| B( Y^{N,I,R}_s ) \|_{L^{\nicefrac{ 2p}{\theta}}(\P; \HS(U,H))}
\Big ]^{1+\frac{1}{ 2 \theta}}
+
\sup\nolimits_{s\in [0,T]}
\| Y^{N,I,R}_s\|_{L^{2p}(\P; H_\eta)}
\bigg)
\\
&
\nonumber
\cdot
\bigg(
1 
+  
2
\bigg[
\|\xi  \|_{L^{2pc}(\P; H_\gamma)}
+
C
\Big[
\tfrac{ T^{1-(\gamma+\alpha)}}{1-(\gamma+\alpha)} 
+
\sqrt{\tfrac{pc(2pc-1)}{(1-2(\gamma + \beta))}}
T^{\frac{1}{2}-(\gamma + \beta)}
\Big]
\Big[1 + \sup\nolimits_{s\in [0,T]} \| Y^{N,I,R}_s \|_{L^{2pca}(\P; H)}^{a} \Big]
\bigg]^c
\bigg)
.
\end{align}
Furthermore, note that
Corollary~\ref{rex} proves that for all $t\in [0,T]$ it holds that
$\P(X_t\in H_\eta)=1$ and
\begin{equation}
\label{rex.implies}
\begin{split}
\sup_{I\in \mathcal{P}_0(\H)}
\sup_{R\in L^1(U)}
\sup_{t\in [0,T]}
\sup_{N\in\N}
\left(
\| Y^{N,I,R}_t\|_{L^{2p}(\P; H_\eta)}
+
\|X_t^{I,  \operatorname{Id}_U } \|_{L^{2p\kappa}(\P; H_\eta)}
+
\|X_t \|_{L^{2p}(\P; H_{\eta})}
\right)
<\infty.
\end{split}
\end{equation}
Combining
Corollary~\ref{finite.function} and \eqref{eq111}--\eqref{rex.implies} proves that
there exists a real number $\tilde K \in [0,\infty)$
such that for all $N\in \N, I \in \mathcal{P}_0(\H), R\in L^1(U)$  it holds that
\begin{equation}
\label{kocka}
\begin{split}
\sup_{t\in [0,T]}
\| X_t - Y_t^{N,I,R} \|_{L^p(\P; H)}
\leq
\tilde K
\bigg[
\!
\sup_{t\in [0,T]} \|  P_{\H \setminus I} X_t \|_{L^{2p}(\P; H_\delta)}
+
N^{\delta - \eta }
+
\sup_{v\in H_\eta}\!\Big(\tfrac{\|B(v) (\operatorname{Id}_U - R) \|_{\HS(U, H_{-\nu})}}{1+\|v\|_{H_\eta}^\kappa}\Big)
\!
\bigg]
.
\end{split}
\end{equation}
Moreover, note that for all $I\in \mathcal{P}_0(\H)$ it holds that
\begin{equation}
\begin{split}
\sup\nolimits_{t\in [0,T]} \| P_{\H \setminus I} X_t \|_{L^{2p}(\P; H_\delta)}
\leq
\|P_{\H\setminus I} \|_{L(H, H_{\delta-\eta})}
\sup\nolimits_{t\in [0,T]} \|X_t \|_{L^{2p}(\P; H_{\eta})}.
\end{split}
\label{krneki}
\end{equation}
Combining this, \eqref{rex.implies}, and \eqref{kocka} completes the proof of Theorem~\ref{Finale}.
\end{proof}
\section{A stochastic reaction-diffusion partial differential equation}
\label{section8}
In this section we illustrate Theorem \ref{Finale} by a simple example, that is, 
we illustrate Theorem \ref{Finale} in the case of a 
stochastic reaction-diffusion partial differential equation. 
More formally, 
let
$ 
  \left(
    H, 
    \langle \cdot, \cdot \rangle_H ,
    \left\| \cdot \right\|_H 
  \right)
$
be the $ \R $-Hilbert space of equivalence classes
of Lebesgue square integrable functions 
from $ (0,1) $ to $ \R $,
let
$ T, \rho, \kappa, \varepsilon, \sigma \in (0,\infty) $,
$ \theta \in ( 0, \nicefrac{ 1 }{ 4 } ] $,  
$ \gamma \in ( \nicefrac{ 1 }{ 4 }, \nicefrac{ 1 }{ 2 } ) $,
$
  ( e_n )_{ n \in \N }
  \subseteq H
$,
$ 
  ( r_n )_{ n \in \N } \subseteq \R
$,
$ 
  ( \lambda_n )_{ n \in \N } \subseteq \R
$
satisfy 
that for all $ n \in \N $
and $ \mu_{ (0,1) } $-almost all 
$ x \in (0,1) $
it holds that
$  
  e_n( x ) = \sqrt{2} \sin( n \pi x )
$,
$ 
  \lambda_n = - \varepsilon \pi^2 n^2 
$,
and 
$
  \sup_{ m \in \N } \left( m \cdot \left| r_m \right| \right) < \infty
$,
let
$
  A \colon D( A ) \subseteq H \to H
$ 
be the linear operator such that
$ 
  D(A) 
  = 
  \big\{
    v \in H \colon 
    \sum_{ n = 1 }^{ \infty }
    | 
      \lambda_n \langle e_n, v \rangle_H 
    |^2 < \infty
  \big\} 
$
and such that for all 
$v \in D(A)$
it holds that
$  
  A v = \sum_{ n = 1 }^{ \infty } \lambda _n 
  \langle e_n, v \rangle _H e_n 
$,
let $ Q \in L_1(H) $ 
be the linear operator such that 
for all $ v \in H $ it holds that
$ 
  Q v = \sum_{ n = 1 }^{ \infty } r_n \langle e_n, v \rangle_H e_n 
$,
let
$ 
  (H_r, \langle \cdot, \cdot \rangle_{H_r}, \left \| \cdot \right \|_{H_r} )
$, $ r \in \R $, 
be a family of interpolation spaces associated to $ - A $
(see, e.g., Theorem and Definition~2.5.32 
in \cite{Jentzen2014SPDElecturenotes}),
let
$ ( \Omega, \mathcal{ F }, \P) $
be a probability space with a normal filtration 
$ ( \mathcal{ F }_t )_{t \in [ 0, T]} $,
let 
$ \xi \in H_{ 1 / 2 } $
satisfy that for $ \mu_{ (0,1) } $-almost all $ x \in (0,1) $
it holds that
$
  \xi( x ) \geq 0
$,
let
$ \left( W_t \right)_{t \in [0, T]}$
be a cylindrical
$\operatorname{Id}_H$-Wiener process with respect to 
$ ( \mathcal{F}_t )_{t \in [0, T]} $,
let 
$
  F \in \mathcal{C}(H_\gamma, H)
$ 
and 
$
  B \in \mathcal{C}(H_\gamma, \HS(H))
$ 
be the functions with the properties that for all 
$ v\in H_\gamma $, $ u \in H $
and $ \mu_{ (0,1) } $-almost all $ x \in (0,1) $ 
it holds that 
$
  \big( F(v) \big)( x ) =
  \kappa \left| v(x) \right| \left( \rho - v(x) \right)
$
and 
$
  \big(B(v)(u)\big)(x) = \sigma \cdot v(x) \cdot \big( \sqrt{Q} u \big)(x)
$,
let 
$ ( P_n )_{ n \in \N } \subseteq L(H) $
be the linear operators with the property that for all
$
  x \in H
$, 
$ n \in \N $ 
it holds that
$
  P_n(x) = \sum_{l = 1}^n \langle e_l , x \rangle_H e_l
$,
let 
$
  Y^{ N, n, m} \colon [0, T] \times \Omega \to P_n(H_\gamma) 
$, 
$ N, n, m, \in \N,$ 
be 
$(\mathcal{F}_t)_{t\in [0,T]}$-adapted 
stochastic processes 
such that for all 
$ t \in [0,T] $, $ N, n, m \in \N $
it holds $\P$-a.s.\ that
\begin{equation}
\begin{split}
Y^{N,n,m}_t 
&
=
e^{tA} P_n \xi
+
\int_0^t
e^{(t- \lfloor s \rfloor_{T/N})A}
\,
\1_{ \big \{ 
 \| 
P_n F  (Y^{N,n,m}_{\lfloor s \rfloor_{T/N}}  )
 \|_H 
+ 
 \| P_n B( Y^{N,n,m}_{\lfloor s \rfloor_{T/N}} )  \|_{HS(H)} 
\,
\leq 
\left ( \frac{N}{T}  \right)^\theta \big \}}
P_n F( Y^{N,n,m}_{\lfloor s \rfloor_{T/N}} ) 
\,
ds
\\
&
\quad
+
\int_0^t e^{(t- \lfloor s \rfloor_{T/N})A}
\,
\1_{ \big\{ \| P_n F(Y^{N,n,m}_{\lfloor s \rfloor_{T/N}} )\|_H 
+ 
 \| P_n B( Y^{N,n,m}_{\lfloor s \rfloor_{T/N}}  )  \|_{HS(H)} 
\leq 
\left ( \frac{N}{T} \right )^\theta \big \}}
\,
 P_n B( Y^{N,n,m}_{\lfloor s \rfloor_{T/N}}  )  P_m \,dW_s
 ,
\end{split}
\end{equation}
and let $X \colon [0,T]\times \Omega \to H_\gamma$ be 
an $(\mathcal{F}_t)_{t\in [0,T]}$-adapted stochastic process
with continuous sample paths
such that for all $t\in [0, T]$ it holds $\P$-a.s.\,that
\begin{equation}
X_t = e^{tA} \xi + \int_0^t e^{(t-s)A}F(X_s) \, ds
+
\int_0^t e^{(t-s)A} B( X_s ) \, dW_s
  .
\end{equation}
The stochastic process $ X $ is thus a solution process 
of the SPDE 
\begin{equation*}
  d X_t( x ) = 
  \left[
    \varepsilon
    \tfrac{ \partial^2 }{ \partial x^2 } X_t( x )
    +
    \kappa X_t( x ) \left( \rho - X_t( x ) \right)
  \right] 
  dt
  +
  \sigma X_t( x ) \, dW_t(x) ,
  \,
  X_0( x ) = \xi( x ) ,
  \, 
  X_t( 0 ) = X_t( 1 ) = 0 
\end{equation*}
for $ t \in [0,T] $, $ x \in (0,1) $.
We intend to apply Theorem~\ref{Finale}
to estimate the quantities
$
  \sup_{ t \in [0,T] }
  \| 
    X_t - Y_t^{ N, n, m } 
  \|_{ L^p(\P ; H) }
$
for $ p \in [2,\infty) $, $ N, n, m \in \N $.
To this end, we now check the assumptions 
of Theorem~\ref{Finale}.
First, observe that the assumption that 
$ \gamma > \nicefrac{ 1 }{ 4 } $ 
ensures that for all $ v \in H_{ \gamma } $ it holds that
\begin{equation}
\label{eq:F_bound_example}
\begin{split}
  \left\| 
    F(v) 
  \right \|_{
    H_{ - \gamma } 
  }
& =
  \sup_{ z \in H_{ \gamma } \backslash \{ 0 \} }
  \left(
  \tfrac{
    |
      \langle F(v), z \rangle_H
    |
  }{
    \left\| z \right\|_{ H_{ \gamma } }
  }
  \right)
=
  \sup_{ z \in H_{ \gamma } \backslash \{ 0 \} }
  \left(
    \tfrac{ 1 }{
      \left\| z \right\|_{
        H_{ \gamma }
      }
    }
    \left|
      \int_0^1
        \kappa \cdot
        \left| v(x) \right| \cdot ( \rho - v(x) ) \cdot z(x) 
        \,
      dx
    \right|
  \right)
\\
  &
  \leq 
  \sup_{ z \in H_{ \gamma } \backslash \{ 0 \} }
  \left(
    \tfrac{ 
      \| z \|_{ L^{ \infty }( \mu_{ (0,1) }; \R ) }
    }{
      \| z \|_{ H_{ \gamma } }
    }
    \int_0^1
    \kappa \cdot
    | v(x) |
    \cdot
    | \rho - v(x)| 
    \, dx
  \right)
\\ & 
\leq
  \kappa
  \left(
    \sup_{ z \in H_{ \gamma } \backslash \{ 0 \} }
    \tfrac{ 
      \| z \|_{ L^{ \infty }( \mu_{ (0,1) }; \R ) }
    }{
      \| z \|_{ H_{\gamma} }
    }
  \right)
  \left\| v \right\|_H 
  \left( 
    \rho + \left\| v \right\|_H
  \right)
\\ & \leq
  \left(
    \sup_{ z \in H_{ \gamma } \backslash \{ 0 \} }
    \tfrac{ 
      \| z \|_{ L^{ \infty }( \mu_{ (0,1) }; \R ) }
    }{
      \| z \|_{ H_{ \gamma } }
    }
  \right)
  \tfrac{ 3 \, \kappa \max\{ 1, \rho \} }{ 2 }
  \left( 
    1 + \left\| v \right\|_H^2
  \right)
  < \infty
  .
\end{split}
\end{equation}
Next note that the triangle inequality shows 
that for all $x,y \in \R$ it holds that
\begin{equation}
\label{trivial.ineq}
\begin{split}
&
  \left|
    \left| x \right|
    \left( \rho - x \right)
    -
    \left| y \right|
    \left( \rho - y \right)
  \right|
\leq 
  \rho \left| x - y \right|
  +
  \left|
    \left| x \right| x
    -
    \left| y \right| y
  \right|
\\ & \leq 
  \rho \left| x - y \right|
  +
  \left|
    \left| x \right| x
    -
    \left| y \right| x
  \right|
  +
  \left|
    \left| y \right| x
    -
    \left| y \right| y
  \right|
\leq 
  \rho \left| x - y \right|
  +
  \left| x \right|
  \left|
    \left| x \right|
    -
    \left| y \right|
  \right|
  +
  \left| y \right|
  \left|
    x - y
  \right|
\\
&
\leq
  \rho \left| x - y \right|
  +
  \left| x \right| 
  \left| x - y \right|
  + 
  \left| y \right|
  \left| x - y \right|
  \leq
  \max\{ 1, \rho \}
  \left| x - y \right|
  \left( 1 + \left| x \right| + \left| y \right| \right)
  .
\end{split}
\end{equation}
This implies that for all $ u, v \in H_\gamma $ it holds that 
\begin{equation}
\label{seven}
\begin{split}
  \left\| F( u ) - F( v ) \right\|_H
&
  =
  \kappa
  \left(
    \int_0^1
    \left|
      \left| u(s) \right| 
      \left( \rho - u(s) \right)
      -
      \left| v(s) \right| 
      \left( \rho - v(s) \right)
    \right|^2
    ds
  \right)^{
    \nicefrac{ 1 }{ 2 } 
  }
\\ 
&
  \leq 
  \kappa
  \max\{ 1, \rho \}
  \left(
    \int_0^1
    \left| u(s) - v(s) \right|^2
    \left[
      1 + \left| u(s) \right| + \left| v(s) \right|
    \right]^2
    ds
  \right)^{
    \nicefrac{ 1 }{ 2 } 
  }
\\
&
  \leq
  \left(
    \sup_{ z \in H_{ \gamma } \backslash \{ 0 \} }
    \tfrac{ 
      \| z \|_{ L^{ \infty }( \mu_{ (0,1) }; \R ) }
    }{
      \| z \|_{ H_{ \gamma } }
    }
  \right)
  \kappa
  \max\{ 1, \rho \}
  \left\| u - v \right\|_H 
  \left(
    1  
    +
    \|u \|_{H_{\gamma}}  
    +
    \|v \|_{H_{\gamma}}
  \right)
  .
\end{split}
\end{equation}
Moreover, note that 
for all $ u, v \in H_{ \gamma } $ 
it holds that
\begin{equation}
\label{help.estimate}
\begin{split}
  \left\| B( u ) - B( v ) \right\|_{ \HS(H) }
& 
=
  \left\| B( u - v ) \right\|_{ \HS(H) }
=
  \sigma
  \left[
    \sum_{ n = 1 }^{ \infty }
    \left\| 
      \left( u - v \right) ( Q^{ 1 / 2 } e_n ) 
    \right\|_H^2
  \right]^{ \nicefrac{ 1 }{ 2 } }
\\ & 
=
  \sigma
  \left[
    \sum_{ n = 1 }^{ \infty }
    r_n
    \left\| \left( u - v \right) e_n \right\|_H^2
  \right]^{ \nicefrac{ 1 }{ 2 } }
\leq
  \sigma
  \sqrt{ 
    2 
    \operatorname{Trace}( Q )
  } 
  \left\| u - v \right\|_H
.
\end{split}
\end{equation}
This shows that 
for all $ v \in H_{ \gamma } $, $ p \in [0,\infty) $ it holds that
\begin{equation}
\label{eq:coercivity_example}
\begin{split}
  \langle v, F(v) \rangle_H 
  + 
  p
  \left\| B(v) \right\|_{ \HS(H) }^2
&
\leq
  \kappa \rho
  \left\| v \right\|_H^2
  +
  p
  \left\| B(v) \right\|^2_{ \HS(H) }
\leq 
  \left(
    \kappa \rho 
    +
      2 p \left| \sigma \right|^2 
      \operatorname{Trace}( Q )
  \right)
  \left\| v \right\|_H^2
.
\end{split}
\end{equation}
Furthermore, note that
\eqref{help.estimate} 
and the fact that the function 
$
  \R \ni x \mapsto - x \left| x \right| \in \R
$
is non-decreasing show that
for all $ u, v \in H_1 $, $ p \in [0,\infty) $
it holds that
\begin{equation}
\label{eq:monotonicity_example}
\begin{split}
&
\langle u-v, Au - Av + F(u) - F(v) \rangle_H
+
p
\left\| 
  B(u) - B(v)
\right\|_{ \HS(H) }^2
\\
&
\leq
\kappa \rho
\left\| u - v \right\|_H^2
-
\kappa
\left< 
  u - v , u \left| u \right| - v \left| v \right|
\right>_H
+
2 p \left| \sigma \right|^2 
\operatorname{Trace}( Q )
\left\| u - v \right\|_H^2 
\\
&
\leq
\left(
  \kappa \rho + 2 p \left| \sigma \right|^2 
  \operatorname{Trace}( Q )
\right) 
\left\| u - v \right\|_H^2 
.
\end{split}
\end{equation}
Combining \eqref{eq:F_bound_example}, \eqref{seven}, \eqref{help.estimate}, 
\eqref{eq:coercivity_example}, and \eqref{eq:monotonicity_example} 
allows us to apply Theorem~\ref{Finale}
to obtain 
that 
for every 
$ \eta \in [ \gamma, \nicefrac{ 1 }{ 2 } ) $,
$ \nu \in (0, \nicefrac{ 1 }{ 2 } ) $,
$ \kappa \in (0,\infty) $,
$ p \in [2,\infty) $
there exists a real number $ K \in [0,\infty) $ such that 
for all $ N, n, m \in \N $ it holds that
\begin{equation}
\label{eq:XY}
  \sup_{
    t \in [0,T]
  }
  \big\| 
    X_t - Y_t^{N, n, m} 
  \big\|_{
    L^p( \P; H ) 
  }
\leq 
  K
  \left[ 
    N^{-\eta}
    +
    n^{-2\eta}
    +
    \sup_{v\in H_\eta}
    \Big(
      \tfrac{
        \| 
          B(v) ( \operatorname{Id}_H - P_m ) 
        \|_{
          \HS( H, H_{ - \nu } ) }
        }{
          ( 1 + \| v \|_{ H_{ \eta } } )^\kappa
        }
    \Big)
  \right]
  .
\end{equation}
In the next step we intend to estimate 
the third summand on the right hand side 
of \eqref{eq:XY}.
For this let
$
  \left\| \cdot \right\|_{ H^r( (0,1), \R) } \colon H \to [0, \infty]
$,
$ r \in (0,1) $,
be the functions with the property that for all 
$ r \in (0,1) $, $ v \in H $ 
it holds that
\begin{equation}
  \left\| v \right\|_{ H^r( (0,1), \R ) }
=
  \left[
    \int_0^1 |v(x)|^2\,dx
    +
    \int_0^1 \int_0^1
    \tfrac{ 
      | v(x) - v(y) |^2
    }{  
      \left| x - y \right|^{ 1 + 2 r } 
    }
    \, dx \, dy
  \right]^{ \nicefrac{ 1 }{ 2 } }
.
\end{equation}
Note that there exists real numbers 
$ \vartheta_r \in [1,\infty) $, $ r \in (0, \nicefrac{ 1 }{ 2 } ) $,
such that
for all 
$ r \in ( 0, \nicefrac{ 1 }{ 2 } ) $, $ v \in H_r $ 
it holds that
\begin{equation}
\label{eq143}
  \tfrac{ 1 }{ \vartheta_r } 
  \left\| v \right\|_{ H^{ 2 r }( (0,1), \R ) } 
  \leq 
  \left\| v \right\|_{ H_r }
  \leq 
  \vartheta_r
  \left\| v \right\|_{
    H^{ 2 r }( (0,1), \R ) 
  }
\end{equation}
(see, e.g., A. Lunardi~\cite{l09} or also (A.46) in Da Prato \& Zabczyk~\cite{dz92}).
In addition, we observe that for all 
$ u, v \in \mathcal{M}( \mathcal{B}( (0,1) ) , \mathcal{B}( \R ) ) $, 
$ r \in (0,1) $, $ s \in (r,\infty) $ 
it holds that 
\begin{equation}
\label{yield1}
\begin{split}
  \left\| u \cdot v \right\|_{ 
    H^r((0,1),\R)
  }
&
  \leq
  \sqrt{2}
  \left\| v \right\|_{ H^r( (0,1) , \R ) }
  \left(
    \sup\nolimits_{ x \in (0,1) } |u(x)|  
    +
    \tfrac{
      \sqrt{ 3 }
    }{
      \sqrt{ 2 s - 2 r }
    }
    \sup\nolimits_{ x \in (0,1), y \in (x,1) }
    \tfrac{
      | u(x) - u(y) | 
    }{
      | x - y |^s
    }
  \right)
\end{split}
\end{equation}
(cf., e.g., Jentzen \& R\"{o}ckner~\cite[(22)--(23)]{jr12}).
This and \eqref{eq143}
prove that 
for all 
$ m \in \N $, 
$ r \in [ \gamma, \nicefrac{ 1 }{ 2 } ) $,
$ s \in ( 2 r - \nicefrac{ 1 }{ 2 } , \infty) $, 
$ \nu \in ( \nicefrac{ s }{ 2 } + \nicefrac{ 1 }{ 4 } , \infty ) $, $ v \in H_r $ 
it holds that
\begin{equation}
\begin{split}
\label{ithol}
&
  \left\| 
    B(v) ( \operatorname{Id}_H - P_m ) 
  \right\|_{
    \HS( H, H_{ - \nu } ) 
  }
=
  \left\| 
    ( - A )^{ - \nu } B(v) ( \operatorname{Id}_H - P_m ) 
  \right\|_{ \HS( H ) }
\\
&
  =
  \left\| 
    ( \operatorname{Id}_H - P_m ) B(v)^* ( - A )^{ \nu } 
  \right\|_{ \HS( H ) }
\leq
  \left\| 
    ( \operatorname{Id}_H - P_m) ( - A )^{ - r } 
  \right\|_{ L(H) }
  \left\| 
    ( - A )^r B(v)^* ( - A )^{ - \nu } 
  \right\|_{ \HS( H ) }
\\ &
  =
  \left| \lambda_{ m + 1 } \right|^{ - r }    
  \left\| 
    ( - A )^r B(v)^* (-A)^{ - \nu } 
  \right\|_{
    \HS( H )
  }
=
  \left| \lambda_m \right|^{ - r }    
  \left[
    \sum_{ n = 1 }^{ \infty }
    \left\| 
      ( - A )^r B(v)^* ( - A )^{ - \nu } e_n 
    \right\|_H^2
  \right]^{ \nicefrac{ 1 }{ 2 } }
\\ & =
  \frac{ 1 }{ 
    \left| \lambda_m \right|^{ r }    
  }
  \left[
    \sum_{ n, k = 1 }^{ \infty }
      \left| r_k \right|
      \left| \lambda_n \right|^{ - 2 \nu }
      \left| \left< e_k, ( - A )^r ( v \cdot e_n ) \right>_H \right|^2
  \right]^{ \nicefrac{ 1 }{ 2 } }
\\ & \leq
  \frac{ 1 }{ 
    \left| \lambda_m \right|^{ r }    
  }
  \left[
    \sup_{ k \in \N }
      \left| r_k \right|
      \left( \lambda_k \right)^{ 1 / 2 }
  \right]^{ \nicefrac{ 1 }{ 2 } }
  \left[
    \sum_{ n, k = 1 }^{ \infty }
      \left| \lambda_n \right|^{ - 2 \nu }
      \left| \left< e_k, ( - A )^{ ( r - 1 / 4 ) } ( v \cdot e_n ) \right>_H \right|^2
  \right]^{ \nicefrac{ 1 }{ 2 } }
\\ &
\leq 
  \frac{ \vartheta_{ r - 1 / 4 } }{ 
    \left| \lambda_m \right|^{ r }    
  }
  \left[
    \sup_{ k \in \N }
      \left| r_k \right|^{ 1 / 2 }
      \left( \lambda_k \right)^{ 1 / 4 }
  \right]
  \left[
    \sum_{ n = 1 }^{ \infty }
      \left| \lambda_n \right|^{ - 2\nu }
      \left\| v \cdot e_n \right\|_{ H^{ 2 r - 1 / 2 } }^2
  \right]^{ \nicefrac{ 1 }{ 2 } }
\\ & \leq
  \tfrac{
    \vartheta_{ r - 1 / 4 } 
    \,
    \sqrt{ 2 }
    \,
    \|
      ( - A )^{ 1 / 4 } Q^{ 1 / 2 }
    \|_{ L(H) }
    \,
    \| v \|_{ H^{ 2 r - 1 / 2 }( (0,1), \R ) }
  }{
    \left| \lambda_m \right|^{ r }    
  }
  \left[
  \sum_{ n = 1 }^{ \infty }
  \tfrac{ 
  \left[
    \sqrt{ 2 }
    +
    \frac{ \sqrt{ 3 } }{ \sqrt{ 2 s - 4 r } }
    \sup_{ x \in (0,1), y \in (x,1) }
    \tfrac{ | e_n(x) - e_n(y) | }{ |x - y|^s }
  \right]^2
  }{
    \left| \lambda_n \right|^{ 2 \nu }
  }
  \right]^{ \nicefrac{ 1 }{ 2 } }
\\
&
  =
  \tfrac{
    2
    \,
    \vartheta_{ r - 1 / 4 } 
    \,
    \|
      ( - A )^{ 1 / 4 } Q^{ 1 / 2 }
    \|_{ L(H) }
    \,
    \| v \|_{ H^{ 2 r - 1 / 2 }( (0,1), \R ) }
  }{
    \left| \lambda_m \right|^{ r }    
  }
  \left[
    \sum_{ n = 1 }^{ \infty }
    \tfrac{ 
    \left[
      1
      +
      \frac{ 
        \sqrt{ 3 }
      }{ 
        \sqrt{ 2 s - 4 r } 
      }
      \sup_{ x \in (0,1) }
      \sup_{ y \in (x,1) }
      \tfrac{
        \left|
          \sin( n \pi x ) - \sin( n \pi y )
        \right|^s
      }{
        \left| x - y \right|^s
      }
    \right]^2
    }{ 
      \left| \lambda_n \right|^{ 2 \nu }
    }
  \right]^{ \nicefrac{ 1 }{ 2 } }
\\
&
\leq
  \frac{
    2
    \left| \vartheta_{ r - 1 / 4 } \right|^2
    \|
      ( - A )^{ 1 / 4 } Q^{ 1 / 2 }
    \|_{ L(H) }
    \left\| v \right\|_{ H_{ r - 1 / 4 } }
  }{
    \varepsilon^{ \nu }
    \,
    \pi^{ 2 \nu }
    \left| \lambda_m \right|^{ r }    
  }
  \left[
    \sum_{ n = 1 }^{ \infty }
    n^{ - 4 \nu }
    \left[
      1
      +
      \tfrac{ \sqrt{ 3 } }{
        \sqrt{ 1 + 2 s - 4 r }
      }
      \left( n \pi \right)^s
    \right]^2
  \right]^{ \nicefrac{ 1 }{ 2 } }
\\ & \leq 
  \frac{
    2
    \left| \vartheta_{ r - 1 / 4 } \right|^2
    \|
      ( - A )^{ 1 / 4 } Q^{ 1 / 2 }
    \|_{ L(H) }
    \left\| v \right\|_{ H_{ r - 1 / 4 } }
  }{
    \varepsilon^{ ( \nu + r ) }
    \,
    \pi^{ ( 2 \nu + 2 r ) }
    \,
    m^{ 2 r }
  }
    \left[
      1
      +
      \tfrac{ \pi^s \sqrt{ 3 } }{
        \sqrt{ 1 + 2 s - 4 r }
      }
    \right]
  \left[
    \sum_{ n = 1 }^{ \infty }
    n^{ ( 2 s - 4 \nu ) }
  \right]
  < \infty
  .
\end{split}
\end{equation}
This implies that
for all 
$ m \in \N $, 
$ \eta \in [ \gamma, \nicefrac{ 1 }{ 2 } ) $,
$ s \in ( 2 \eta - \nicefrac{ 1 }{ 2 } , \nicefrac{ 1 }{ 2 } ) $, 
$ \nu \in ( \nicefrac{ s }{ 2 } + \nicefrac{ 1 }{ 4 } , \nicefrac{ 1 }{ 2 } ) $
it holds that
\begin{equation}
\begin{split}
&
  \sup_{ v \in H_{ \eta } }
  \left(
  \frac{
    \left\| 
      B(v) ( \operatorname{Id}_H - P_m ) 
    \right\|_{
      \HS( H, H_{ - \nu } ) 
    }
  }{
    1 + \left\| v \right\|_{ H_{ \eta } }
  }
  \right)
\leq 
  \frac{
    | \vartheta_{ \eta - 1 / 4 } |^2
    \,
    \|
      ( - A )^{ 1 / 4 } Q^{ 1 / 2 }
    \|_{ L(H) }
  }{
    \varepsilon^{ ( \nu + \eta + 1 / 4 ) }
    \,
    m^{ 2 \eta }
    \,
    ( \nicefrac{ s }{ 2 } + \nicefrac{ 1 }{ 4 } - \eta )
  }
  \left[
    \sum_{ n = 1 }^{ \infty }
    n^{ ( 2 s - 4 \nu ) }
  \right]
  < \infty
  .
\end{split}
\end{equation}
Combining this with \eqref{eq:XY} shows that
for every 
$ \eta \in [ \gamma, \nicefrac{ 1 }{ 2 } ) $,
$ p \in [2,\infty) $
there exists a real number $ K \in [0,\infty) $ such that 
for all $ N, n, m \in \N $ it holds that
\begin{equation}
  \sup_{
    t \in [0,T]
  }
  \big\| 
    X_t - Y_t^{N, n, m} 
  \big\|_{
    L^p( \P; H ) 
  }
\leq 
  K
  \left[ 
    N^{-\eta}
    +
    n^{-2\eta}
    +
    m^{-2\eta}
  \right]
  .
\end{equation}
This ensures that
for every $ p, \iota \in (0, \infty ) $ 
there exists a real number $ K \in [0,\infty) $ 
such that for all $ N, n, m \in \N $
it holds that
\begin{equation}
\label{eq:bedaneca}
\begin{split}
  \sup_{ t \in [0,T] }
  \big\| 
    X_t - Y_t^{ N, n, m } 
  \big\|_{ L^p(\P ; H) }
\leq
K
\left( 
  \frac{ 
    1 
  }{
    N^{ 
      ( \nicefrac{ 1 }{ 2 } - \iota ) 
    }
  }
  +
  \frac{ 1 }{
    n^{ ( 1 - \iota ) }
  }
  +
  \frac{ 1 }{ 
    m^{ ( 1 - \iota ) }
  }
\right)
.
\end{split}
\end{equation}
In particular, this shows that for every 
$ p, \iota \in (0, \infty ) $ 
there exists a real number $ K \in [0,\infty) $ 
such that for all $ n \in \N $
it holds that
$
  \sup_{ t \in [0,T] }
  \| 
    X_t - Y_t^{ n^2, n, n } 
  \|_{ L^p(\P ; H) }
\leq
  K \cdot 
  n^{ ( \iota - 1 ) }
$.
\bibliographystyle{acm}
\bibliography{bibfile}

\end{document}